%% file: Localizations.tex
\documentclass[10pt]{amsart}

\usepackage{anysize}
\usepackage[utf8]{inputenc}
\usepackage[english]{babel}
\usepackage{amssymb, amsmath, amsthm}
\usepackage{enumitem}
\usepackage{mathtools}
\usepackage{float}
\usepackage{tikz-cd}
\usetikzlibrary{arrows}
\usepackage{url}
\usepackage[all,cmtip]{xy}
\usepackage[hidelinks,bookmarksnumbered]{hyperref}

\let\cal\mathcal
\def\AA{{\cal A}}
\def\BB{{\cal B}}
\def\CC{{\cal C}}
\def\DD{{\cal D}}
\def\EE{{\cal E}}
\def\FF{{\cal F}}

\def\II{{\cal I}}

\def\TT{{\cal T}}
\def\UU{{\cal U}}

\let\blb\mathbb
 
\def\bD{{\blb D}}

\def\bR{{\blb R}}

\author{Ruben Henrard}
\address{Ruben Henrard \\ Universiteit Hasselt \\ Campus Diepenbeek \\ Departement WNI \\ 3590 Diepenbeek \\ Belgium}
\email{ruben.henrard@uhasselt.be}
\author{Adam-Christiaan van Roosmalen}
\address{Adam-Christiaan van Roosmalen \\ Universiteit Hasselt \\ Campus Diepenbeek \\ Departement WNI \\ 3590 Diepenbeek \\ Belgium}
\email{adamchristiaan.vanroosmalen@uhasselt.be}
\title{Localizations of (one-sided) exact categories}

\newtheorem{theorem}{Theorem}[section]
\newtheorem{proposition}[theorem]{Proposition}
\newtheorem{lemma}[theorem]{Lemma}
\newtheorem{corollary}[theorem]{Corollary}

\theoremstyle{definition}
\newtheorem{definition}[theorem]{Definition}
\newtheorem{remark}[theorem]{Remark}
\newtheorem{example}[theorem]{Example}
\newtheorem{construction}[theorem]{Construction}

\DeclareMathOperator{\im}{im}
\DeclareMathOperator{\coim}{coim}
\DeclareMathOperator{\coker}{coker}

\DeclareMathOperator{\Hom}{Hom}
\DeclareMathOperator{\Fun}{Fun}
\DeclareMathOperator{\Mor}{Mor}
\DeclareMathOperator{\Adm}{Adm}

\DeclareMathOperator{\cofib}{cofib}
\DeclareMathOperator{\fib}{fib}
\DeclareMathOperator{\rep}{rep}
\DeclareMathOperator{\add}{add}

\DeclareMathOperator{\Ab}{\mathsf{Ab}}
\DeclareMathOperator{\LCA}{\mathsf{LCA}}
\DeclareMathOperator{\LCAf}{\mathsf{LCA}_{\mathrm{f}}}
\DeclareMathOperator{\LC}{\mathsf{LC}}
\DeclareMathOperator{\LCAcon}{\mathsf{LCA}_\mathrm{con}}
\DeclareMathOperator{\LCAtd}{\mathsf{LCA}_\mathrm{td}}
\DeclareMathOperator{\Ob}{Ob}

\DeclareMathOperator{\Db}{\mathbf{D}^b}
\DeclareMathOperator{\DAb}{\mathbf{D}_{\mathcal{A}}^b}

\newcommand{\inflation}{\rightarrowtail}
\newcommand{\deflation}{\twoheadrightarrow}

\newcommand{\dq}{{/\mkern-6mu/}}

\makeatletter
\newcommand{\myitem}[1]{%
\item[#1]\protected@edef\@currentlabel{#1}%
}
\makeatother


\begin{document}

\begin{abstract}
In this paper, we introduce quotients of exact categories by percolating subcategories.  This approach extends earlier localization theories by Cardenas and Schlichting for exact categories, allowing new examples. Let $\AA$ be a percolating subcategory of an exact category $\EE$, the quotient $\EE\dq \AA$ is constructed in two steps. In the first step, we associate a set $S_\AA \subseteq \Mor(\EE)$ to $\AA$ and consider the localization $\EE[S^{-1}_\AA]$. In general, $\EE[S_\AA^{-1}]$ need not be an exact category, but will be a one-sided exact category. In the second step, we take the exact hull $\EE \dq \AA$ of $\EE[S_\AA^{-1}]$. The composition $\EE \to \EE[S_\AA^{-1}] \to \EE \dq \AA$ satisfies the 2-universal property of a quotient in the 2-category of exact categories.

We formulate our results in slightly more generality, allowing to start from a one-sided exact category.  Additionally, we consider a type of percolating subcategories which guarantee that the morphisms of the set $S_\AA$ are admissible.

In related work, we show that these localizations induce Verdier localizations on the level of the bounded derived category.
\end{abstract}

\maketitle

\tableofcontents

\input{Introduction}
\input{Preliminaries}
\input{DeflationExact}
\input{PercolatingSubcategories}
\input{LocalizationOfOneSidedExactCategories}
\input{AbelianPercolatingSubcategories}
\input{StabilityOfQuillensObscureAxiom}
\input{ExamplesAndApplications}

\providecommand{\bysame}{\leavevmode\hbox to3em{\hrulefill}\thinspace}
\providecommand{\MR}{\relax\ifhmode\unskip\space\fi MR }
\providecommand{\MRhref}[2]{%
  \href{http://www.ams.org/mathscinet-getitem?mr=#1}{#2}
}
\providecommand{\href}[2]{#2}

\end{document}

%% file: Introduction.tex
\section{Introduction}
\addtocontents{toc}{\protect\setcounter{tocdepth}{1}}
Quotients of abelian and triangulated categories are ubiquitous in geometry, representation theory, and $K$-theory.  In these settings, the quotients are obtained by localizing the abelian or triangulated category at an appropriate multiplicative set of morphisms.

Quillen exact categories provide a more flexible framework to study homological and $K$-theoretic properties than abelian categories.  For an exact category $\CC$, quotient constructions similar to those mentioned before were given in \cite{Cardenas98, Schlichting04}.  Here, the Serre subcategory (of an abelian category) or thick subcategory (of a triangulated category) is replaced by a subcategory that \emph{localizes} $\CC$ (see \cite[4.0.35]{Cardenas98} or see \cite[appendix A]{Levine85}; in our terminology, such subcategories are called \emph{two-sided admissibly percolating} subcategories) or by a left or right \emph{special filtering} subcategory (see \cite{Schlichting04}).  In both cases, the quotient categories are obtained by localizing at a class of morphisms, and satisfy the expected homological and $K$-theoretic properties.

However, there are some natural examples that are not accommodated by these constructions.  It was observed in \cite[example 4]{Braunling20} that, in the category $\LCA$ of locally compact abelian groups, neither the subcategory $\LCA_{\mathsf{C}}$ of compact abelian groups nor the subcategory $\LCA_{\mathsf{D}}$ of discrete abelian groups satisfy the s-filtering condition used in \cite{Schlichting04}, nor do these examples satisfy the conditions in \cite{Cardenas98}.  Another example comes from the theory of glider representations \cite{CaenepeelVanOystaeyen19book}.  Here, one can construct the category of glider representations as a localization of an exact category (namely the category of pregliders \cite{HenrardvanRoosmalen20a}), but this localization is not described by \cite{Cardenas98, Schlichting04}.  Such examples indicate a need for a more general localization theory.

In this paper, we extend the previous localization theories.  We introduce a (left or right) percolating subcategory $\AA$ of an exact category $\CC$ and describe the quotient category $\CC \dq \AA$ (we opt for the notation $\CC \dq \AA$ for the quotient in the 2-category of exact categories; the notation $\CC / \AA$ will be used for the quotient in the 2-category of conflation categories, see below).  When $\CC$ is an abelian category, the notion of a left/right percolating subcategory reduces to the notion of a Serre subcategory.  This generalizes the construction from both \cite{Cardenas98, Schlichting04}.

\begin{theorem}\label{theorem:IntroductionExact}
Let $\AA$ be a left or right percolating subcategory in an exact category $\CC$.  There is an exact functor $Q\colon \CC \to \CC \dq \AA$ between exact categories, satisfying the usual 2-universal property of a quotient.
\end{theorem}

To discuss this result, it will be useful to have the following definition.  A \emph{conflation category} is an additive category $\CC$ together with a class of distinguished kernel-cokernel pairs, called \emph{conflations}, closed under isomorphisms.  A functor between conflation categories is called \emph{(conflation-)exact} if it preserves conflations.  The exact category $\CC$ in theorem \ref{theorem:IntroductionExact} is a conflation category.

In contrast to the aforementioned quotient construction for abelian or triangulated categories (as well as the localizations of exact categories in \cite{Cardenas98, Schlichting04}), the quotient $\CC \dq \AA$ in theorem \ref{theorem:IntroductionExact} is not necessarily a localization of $\CC$ at a set of morphisms.  Instead, our construction passes two steps.  In the first step, we invert a class of morphisms in $\CC$ that need to become invertible under the (exact) quotient functor $Q$, namely inflations with cokernel in $\AA$ and deflations with kernel in $\AA$, as well as their compositions (we write $S_\AA$ for this set; a morphism in $S_\AA$ is called a weak isomorphism).  When $\AA$ is a left or right percolating subcategory, the set $S_\AA$ of weak isomorphisms is a left or right multiplicative system, respectively, and we show that $\CC[S_\AA^{-1}]$ is the quotient $\CC / \AA$ in the 2-category of conflation categories.  In general, the localization $\CC[S_\AA^{-1}]$ can fail to be an exact category.

However, we show that $\CC[S_\AA^{-1}]$ is a one-sided exact category (see \S\ref{subsection:IntroductionExact}).  By \cite{Rosenberg11}, a one-sided exact category admits a 2-universal embedding into its exact hull.  The second step of the proof of theorem \ref{theorem:IntroductionExact} is to embed the localization $\CC[S_\AA^{-1}]$ into its exact hull; this is then the quotient $\CC \dq \AA$ of $\CC$ by $\AA$ in the 2-category of exact categories.

In this paper, we focus on the first step, namely understanding the category $\CC[S_\AA^{-1}].$

\subsection{Exact and one-sided exact categories}\label{subsection:IntroductionExact}  By considering localizations of exact categories at percolating subcategories, we naturally arrive at the setting of one-sided exact categories.  We start by establishing some terminology.

A Quillen \emph{exact category} is a conflation category, satisfying some additional axioms, which we recall in detail in definition \ref{definition:RightExact}.  Following \cite{Keller90}, we refer to the kernel morphism of a conflation as an \emph{inflation} and to the cokernel morphism as a \emph{deflation}.

The original axioms of a Quillen exact category, as given in \cite{Quillen73}, can be partitioned into two dual sets: one set only referring to inflations, and one set only referring to deflations.  In a \emph{one-sided exact category}, only one of these sets is required to hold (see \cite{BazzoniCrivei13,Rump11}).  Thus, for a \emph{deflation-exact category}, we require a set of conflations such that the class of deflations contains all isomorphisms, and is stable under both composition and base change.  One-sided exact categories still enjoy many useful homological properties (see \cite{BazzoniCrivei13}).

Similar one-sided exact structures have occurred in several guises throughout the literature. Our main source of examples is based on left or right almost abelian categories (see \cite{Rump01}).  The axioms of a one-sided exact category are closely related to those of a Grothendieck pretopology (see \cite{Rosenberg11}), to homological categories (see \cite{BorceuxBourn04}), and to categories with fibrations (or cofibrations) and Waldhausen categories (see \cite{Weibel13}).  The latter allows for a $K$-space to be associated to a one-sided exact category.

The main part of the proof of theorem \ref{theorem:IntroductionExact} is to understand the localization of a (one- or two-sided) exact category $\CC$ with respect to the set of weak isomorphisms $S_\AA$ associated to a percolating subcategory $\AA \subseteq \CC$.  In the 2-category of conflation categories, the localization $\CC \to \CC[S_\AA^{-1}]$ satisfies the 2-universal property of a quotient $\CC \to \CC / \AA.$  This is expressed in the following theorem (see theorem \ref{theorem:Maintheorem} in the text, as well as propositions \ref{proposition:MinimalConditionsRMS}, \ref{proposition:InterpretationOfP4}, and \ref{proposition:QuotientInConflationCategories}).

\begin{theorem}\label{theorem:IntroductionMain}
Let $\CC$ be a deflation-exact category and let $\AA \subseteq \CC$ be a deflation-percolating subcategory.  There is an exact functor $Q\colon \CC \to \CC / \AA$ between deflation-exact categories such that $Q(\AA) = 0$ and which is 2-universal with respect to this property.

Furthermore,
\begin{enumerate}
	\item the corresponding set of weak isomorphisms $S_\AA$ is a right multiplicative system and the category $\CC / \AA$ is equivalent to the corresponding localization $\CC[S^{-1}_\AA]$, and
	\item $Q(f\colon X \to Y)$ is zero if and only if $f$ factors through an object of $\AA$.
\end{enumerate}
\end{theorem}

As $Q$ is given by a localization at a right multiplicative set, we know that it commutes with finite limits (but not necessarily with finite colimits).  The main part of the proof of theorem \ref{theorem:IntroductionMain} is to establish that $Q$ maps a conflation in $\CC$ to a kernel-cokernel pair in $\CC[S^{-1}_\AA].$  This is done in proposition \ref{proposition:CokernelsDescend}.

\subsection{Admissibly percolating subcategories}  Even though the set $S_\AA$ of weak isomorphisms in theorem \ref{theorem:IntroductionMain} is a right multiplicative system, it might be difficult to determine whether a given morphism in $\CC$ is a weak isomorphism.  In many examples of interest, weak isomorphisms are \emph{admissible} or \emph{strict} (meaning that they admit a deflation-inflation factorization).  We provide a class of percolating subcategories for which this is always the case, called \emph{admissibly percolating subcategories} or \emph{strictly percolating subcategories} (see definition \ref{Definition:AbelianPercolating}) and show that the associated weak isomorphisms are admissible morphisms.  Moreover, admissibly percolating subcategories satisfy some additional desirable properties, such as the two-out-of-three-property (proposition \ref{proposition:2OutOf3}) and saturation (proposition \ref{proposition:Saturation}) for the weak isomorphisms.  The subcategories considered in \cite{Cardenas98} and the aforementioned subcategories $\LCA_{\mathsf{C}}$ and $\LCA_{\mathsf{D}}$ of the category of locally compact abelian groups $\LCA$ are examples of admissibly percolating subcategories.

Our main results about admissibly percolating subcategories are summarized in the following theorem (combining theorem \ref{theorem:WeakIsomorphismsEqualAAInverseIsomorphisms}, and propositions \ref{proposition:2OutOf3} and \ref{proposition:Saturation}).

\begin{theorem}
Let $\CC$ be a deflation-exact category, and let $\AA \subseteq \CC$ be an admissibly deflation-percolating subcategory.
\begin{enumerate}
  \item Every weak isomorphism is admissible.
	\item The set $S_\AA$ of weak isomorphisms is saturated and satisfies the 2-out-of-3 property.
\end{enumerate}
\end{theorem}

\subsection{Applications and recognition results}  In section \S\ref{section:Examples} we provide several tools for recognizing percolating subcategories in various contexts. In particular, we obtain the following useful proposition.

\begin{proposition}\label{proposition:IntroductionRecognitionToolI}
	Let $\CC$ be a deflation quasi-abelian category.  A full subcategory $\AA$ of $\CC$ is a (strongly) deflation-percolating if and only if $\AA$ is a Serre subcategory which is additionally closed under subobjects.
\end{proposition}

As an easy consequence of this proposition (and its dual), the subcategory $\LCAf\subset \LCA$ of finite abelian groups is a two-sided percolating subcategory. It follows that $\LCA/\LCAf$ is an exact category; in fact, proposition \ref{proposition:TwoSidedQuotientOfQuasiAbelianIsQuasiAbelian} shows that two-sided quotients of quasi-abelian categories are quasi-abelian, hence, $\LCA/\LCAf$ is a quasi-abelian category.

Torsion theories in exact categories give rise to percolating subcategories as well (see proposition \ref{proposition:TorsionfreeIsPercolating} and corollary \ref{corollary:RightConflationExact} in the text).

\begin{theorem}
	Let $(\TT, \FF)$ be a cohereditary torsion theory in an exact category $\CC$.
		\begin{enumerate}
			\item The torsion-free class $\FF$ is a right filtering subcategory of $\CC$.
			\item If the functor $\CC \to \CC\colon C \mapsto C_F$ mapping an object of $\CC$ to its torsion-free quotient is right conflation-exact (i.e.~for any conflation $X \inflation Y \deflation Z$, the map $Y_F \to Z_F$ is a deflation and is the cokernel of $X_F \to Y_F$, see definition \ref{definition:RightConflationExact}), then $\FF$ is a deflation-percolating subcategory of $\CC$.
			\item If the torsion-free functor $L\colon \CC\to \FF$ is an exact functor, then $\FF$ is a special right filtering subcategory of $\CC$.
		\end{enumerate}
\end{theorem}

As a consequence of this theorem, we show that the category $\mathsf{TAb}_{\mathsf{triv}}$ of topological abelian groups with the trivial (or indiscrete) topology is admissibly inflation-percolating in the category $\mathsf{TAb}$ of all topological abelian groups (see example \ref{example:IndiscreteAndHausdorff}).

In a follow-up paper (see \cite{HenrardvanRoosmalen19b}), we consider the derived category of a one-sided exact category and show that the localization sequence $\AA \to \CC \to \CC / \AA$ given by a percolating subcategory in a one-sided exact category, yields a Verdier localization $\Db(\CC) \to \Db (\CC / \AA)$ as in \cite{Schlichting04}.  Moreover, we show that the natural embedding $\CC \hookrightarrow \overline{\CC}$ of a one-sided exact category into its exact hull lifts to a triangle equivalence $\Db(\CC) \to \Db(\overline{\CC})$.

\textbf{Acknowledgments.}  We are grateful to Frederik Caenepeel and Freddy Van Oystaeyen for useful discussions and ideas leading to this paper.  The authors are grateful to Sven Ake-Wegner, Oliver Braunling, and Sondre Kvamme for motivating us to extend our earlier results to the current generality.  The second author is currently a postdoctoral researcher at FWO (12.M33.16N).
\addtocontents{toc}{\protect\setcounter{tocdepth}{2}}

%% file: Preliminaries.tex
\section{Preliminaries}\label{section:Preliminaries}

Throughout this paper, we will assume that all categories are small.

\subsection{Properties of pullbacks and pushouts}
For easy reference, it will be convenient to collect some properties of pullbacks and pushouts.  We start by recalling the pullback lemma.

\begin{proposition}[Pullback lemma] Consider the following commutative diagram in any category:
	\[\xymatrix{
			X\ar[r] \ar[d] & Y\ar[r] \ar[d] & Z \ar[d] \\
		X' \ar[r] & Y' \ar[r] & Z'
	}\]
	Assume that the right square is a pullback.  The left square is a pullback if and only if the outer rectangle is a pullback.
\end{proposition}

The following statement is \cite[proposition~I.13.2]{Mitchell65} together with its dual.

\begin{proposition}\label{proposition:MitchellPullbackPushout}\label{proposition:MitchellPullback}
Let $\CC$ be any pointed category.
\begin{enumerate}
  \item\label{enumerate:MitchellPullback} Consider a  diagram
	\[ \xymatrix{ {X'} \ar[r]^{f'} & {Y'} \ar[d]^{h} & \\ {X} \ar[r]^{f} & {Y} \ar[r]^{g} & Z } \]
where $f$ is the kernel of $g$.  The left-hand side can be completed to a pullback square if and only if $f'$ is the kernel of $gh$.
\item\label{enumerate:MitchellPushout} Consider a diagram
	\[ \xymatrix{ {X} \ar[r]^{f} & {Y} \ar[d]^{h} \ar[r]^{g} & Z \\ & {Y'} \ar[r]^{g'} & Z'} \]
where $g$ is the cokernel of $f$.  The right-hand side can be completed to a pushout square if and only if $g'$ is the cokernel of $hf$.
\end{enumerate}
\end{proposition}

The following well-known proposition states that pullbacks preserve kernels and pushouts preserve cokernels (see \cite[lemma 1]{Rump11}).

\begin{proposition}\label{proposition:PullbacksPreserveKernels}\label{proposition:PushoutsPreserveCokernels}
Let $\CC$ be any pointed category.  Consider the following diagram in $\CC$:
\[\xymatrix{
A \ar[d]^{g} \ar[r]^{f} & B \ar[d]^{h} \\
C \ar[r]^{f'} & D
}\]
\begin{enumerate}
	\item Assume that the commutative square is a pullback.  The morphism $f$ admits a kernel if and only if $f'$ admits a kernel.  In this case, the composition $\ker(f) \to A \to C$ is the kernel of $f'.$
  \item Assume that the commutative square is a pushout.  The morphism $f$ admits a cokernel if and only if $f'$ admits a cokernel.  In this case, the composition $B \to D \to \coker(f')$ is the cokernel of $f.$
\end{enumerate}
\end{proposition}

\subsection{Localizations and right calculus of fractions}\label{subsection:RMS}

The material of this section is based on \cite{GabrielZisman67, KashiwaraSchapira06}.

\begin{definition}\label{Definition:LocalizationWithRespectToMorphisms}
Let $\CC$ be any category and let $S \subseteq \Mor \CC$ be any subset of morphisms of $\CC$.  The \emph{localization of $\CC$ with respect to $S$} is a universal functor $Q\colon \CC \to S^{-1} \CC$ such that $Q(s)$ is invertible, for all $s \in S$.
\end{definition}

\begin{remark}
By universality, we mean that any functor $F\colon \CC \to \DD$ such that every morphism in $S$ becomes invertible in $\DD$ factors uniquely through $Q\colon \CC \to S^{-1} \CC$.  Put differently, for every category $\DD$, the functor $(Q \circ -)\colon \Fun(S^{-1}\CC, \DD) \to \Fun(\CC, \DD)$ induces an isomorphism between $\Fun(S^{-1}\CC, \DD)$ and the full subcategory of $\Fun(\CC, \DD)$ consisting of those functors $F\colon \CC \to \DD$ which make every $s \in S$ invertible.
\end{remark}

\begin{remark}
Since all the categories in this paper are small, localizations always exist.
\end{remark}

In this paper, we often consider localizations with respect to so-called right multiplicative systems.

\begin{definition}\label{definition:RMS}
Let $\CC$ be a category and let $S$ be a set of arrows. Then $S$ is called a \emph{right multiplicative system} if it has the following properties:
\begin{enumerate}[label=\textbf{RMS\arabic*},start=1]
	\item\label{RMS1} For every object $A$ of $\CC$ the identity $1_A$ is contained in $S$. Composition of composable arrows in $S$ is again in $S$.
	\item\label{RMS2} Every solid diagram
	\[\xymatrix{
		X \ar@{.>}[r]^{g} \ar@{.>}[d]_{t}^{\rotatebox{90}{$\sim$}}& Y\ar[d]_{s}^{\rotatebox{90}{$\sim$}}\\
		Z\ar[r]_{f} & W
	}\] with $s\in S$ can be completed to a commutative square with $t\in S$. 
	\item\label{RMS3} For every pair of morphisms $f,g\colon X\rightarrow Y$ and $s\in S$ with source $Y$ such that $s\circ f= s\circ g$ there exists a $t\in S$ with target $X$ such that $f\circ t =g\circ t$.
\end{enumerate}
Often arrows in $S$ will be endowed with $\sim$.
\end{definition}

For localizations with respect to a right multiplicative system, we have the following description of the localization.

\begin{construction}\label{construction:Localization}
	Let $\CC$ be a category and $S$ a right multiplicative system in $\CC$. We define a category $S^{-1}\CC$ as follows:
	\begin{enumerate}
		\item We set $\Ob(S^{-1}\CC)=\Ob(\CC)$.
		\item Let $f_1\colon X_1\rightarrow Y, s_1\colon X_1\rightarrow X, f_2\colon X_2\rightarrow Y, s_2\colon X_2\rightarrow X$ be morphisms in $\CC$ with $s_1,s_2\in S$. We call the pairs $(f_1,s_1), (f_2,s_2) \in (\Mor \CC) \times S$ equivalent (denoted by $(f_1,s_1) \sim (f_2,s_2)$) if there exists a third pair $(f_3\colon X_3\rightarrow Y,s_3\colon X_3\rightarrow X) \in (\Mor \CC) \times S$ and morphisms $u\colon X_3\rightarrow X_1, v\colon X_3\rightarrow X_2$ such that 
		\[\xymatrix@!{
			& X_1\ar[ld]_{s_1}^{}\ar[rd]^{f_1} & \\
			X &X_3\ar[d]^{v}\ar[u]_{u}\ar[l]_{s_3}^{}\ar[r]_{f_3} & Y\\
			& X_2 \ar[ul]^{s_2}_{}\ar[ur]_{f_2}&
		}\] is a commutative diagram.
		\item $\Hom_{S^{-1}\CC}(X,Y)=\left\{(f,s)\mid f\in \Hom_{\CC}(X',Y), s\colon X'\rightarrow X \mbox{ with } s\in S \right\} / \sim$
		\item The composition of $(f\colon X'\rightarrow Y, s\colon X'\rightarrow X)$ and $(g\colon Y'\rightarrow Z, t\colon Y'\rightarrow Y)$ is given by $(g\circ h\colon X''\rightarrow Z,s\circ u\colon X''\rightarrow X)$ where $h$ and $u$ are chosen to fit in a commutative diagram 
		\[\xymatrix{
		X''\ar[r]^{h}\ar[d]_{u}^{\rotatebox{90}{$\sim$}} & Y'\ar[d]_{t}^{\rotatebox{90}{$\sim$}}\\
		X'\ar[r]^{f} & Y
		}\] which exists by \ref{RMS2}.
	\end{enumerate}                                                                                                       
\end{construction}

\begin{proposition}\label{proposition:BasicPropertiesOfLocalization}
Let $\CC$ be a category and $S$ a right multiplicative system in $\CC$.
	\begin{enumerate}
			\item The assignment $X\mapsto X$ and $(f\colon X\rightarrow Y)\mapsto (f\colon X\rightarrow Y, 1_X\colon X\rightarrow X)$ defines a functor $Q\colon \CC\rightarrow S^{-1}\CC$ called the \emph{localization functor} and is a localization of $\CC$ with respect to the set $S$ as in definition \ref{Definition:LocalizationWithRespectToMorphisms}.
			\item For any $s\in S$, the map $Q(s)$ is an isomorphism.
			\item The localization functor commutes with finite limits.
			\item If $\CC$ is an additive category, then $S^{-1}\CC$ is an additive category and the localization functor $Q$ is an additive functor.
	\end{enumerate}
\end{proposition}

\begin{remark}\label{remark:PreservationOfKernelsAndPullbacks}
	It follows that if $\CC$ is an additive category and $S$ a right multiplicative system the functor $Q$ preserves kernels and pullbacks. 
\end{remark}

\begin{definition}
Let $\CC$ be any category and let $S \subseteq \Mor \CC$ be any subset.
\begin{enumerate}
\item We say that $S$ satisfies the \emph{2-out-of-3 property} if, for any two composable morphisms $f,g \in \Mor \CC$, we have that if two of $f,g,fg$ are in $S$, then so is the third.
\item Let $Q \colon \CC \to S^{-1}\CC$ be the localization of $\CC$ with respect to $S$.  We say that $S$ is \emph{saturated} if ${S} = \{f \in \Mor \CC \mid \mbox{$Q(f)$ is invertible}\}.$ 
\end{enumerate}
\end{definition}

%% file: DeflationExact.tex
\section{Basic results on one-sided exact categories}

We now recall the notion of a one-sided exact category as introduced by \cite{BazzoniCrivei13,Rosenberg11,Rump10}. In the remainder of the text we follow the conventions of Rosenberg \cite{Rosenberg11}, that is, one-sided exact categories containing all axioms referring to the deflation-side are called \emph{right exact categories}. This convention is opposite to the terminology used by \cite{BazzoniCrivei13}. In order to avoid further confusion, we prefer to use the terminology of \emph{deflation-exact categories} over right exact categories and dually \emph{inflation-exact} over left exact. 

\begin{definition}\label{Definition:ConflationCategory}
	Let $\CC$ be an additive category.  A sequence $A\xrightarrow{f} B\xrightarrow{g} C$ in $\CC$ where $f = \ker g$ and $g = \coker f$ is called a \emph{kernel-cokernel pair}. \\
	A \emph{conflation category} $\CC$ is an additive category $\CC$ together with a chosen class of kernel-cokernel pairs, closed under isomorphisms, called \emph{conflations}. A map that occurs as the kernel (or the cokernel) in a conflation is called an \emph{inflation} (or a \emph{deflation}). Inflations will often be denoted by $\rightarrowtail$ and deflations by $\twoheadrightarrow$.  A map $f\colon X\rightarrow Y$ is called an \emph{admissible morphism} if it admits a deflation-inflation factorization, i.e. $f$ factors as $X\twoheadrightarrow Z\rightarrowtail Y$. The set of admissible morphisms in $\CC$ is denoted by $\Adm(\CC)$.\\
	Let $\CC$ and $\DD$ be conflation categories.  An additive functor $F\colon \CC\rightarrow \DD$ is called \emph{exact} or \emph{conflation-exact} if conflations in $\CC$ are mapped to conflations in $\DD$.
\end{definition}

\begin{definition}\label{definition:RightExact}
  A \emph{right exact category} or a \emph{deflation-exact category} $\CC$ is a conflation category satisfying the following axioms:
	\begin{enumerate}[label=\textbf{R\arabic*},start=0]
		\item\label{R0} The identity morphism $1_0\colon 0\rightarrow 0$ is a deflation.
		\item\label{R1} The composition of two deflations is again a deflation.
		\item\label{R2} The pullback of a deflation along any morphism exists and is again a deflation, i.e. 
		\[\xymatrix{
		X\ar@{.>>}[d]\ar@{.>}[r] & Y \ar@{->>}[d]\\
		Z\ar[r] & W
		}\]
	\end{enumerate}
	
	Dually, we call an additive category $\CC$ \emph{left exact} or \emph{inflation-exact} if the opposite category $\CC^{op}$ is right exact. Explicitly, an inflation-exact category is a conflation category such that the inflations satisfy the following axioms:
		\begin{enumerate}[label=\textbf{L\arabic*},start=0]
		\item\label{L0} The identity morphism $1_0\colon 0\rightarrow 0$ is an inflation.
		\item\label{L1} The composition of two inflations is again an inflation.
		\item\label{L2} The pushout of an inflation along any morphism exists and is again an inflation, i.e. 
		\[\xymatrix{
		X\ar@{>->}[d]\ar@{->}[r] & Y \ar@{>.>}[d]\\
		Z\ar@{.>}[r] & W
		}\]
	\end{enumerate}
\end{definition}

\begin{definition}\label{definition:StrongRightExact}
Let $\CC$ be a conflation category. In addition to the properties listed in definition \ref{definition:RightExact}, we will also consider the following axioms:
\begin{enumerate}[align=left]
\myitem{\textbf{R0}$^\ast$}\label{R0*} For any $A\in \Ob(\CC)$, $A\rightarrow 0$ is a deflation. 
\myitem{\textbf{R3}}\label{R3} \hspace{0.175cm}If $i\colon A\rightarrow B$ and $p\colon B\rightarrow C$ are morphisms in $\CC$ such that $p$ has a kernel and $pi$ is a deflation, then $p$ is a deflation.
\myitem{\textbf{L0}$^\ast$}\label{L0*} For any $A\in \Ob(\CC)$, $0\rightarrow A$ is an inflation. 
\myitem{\textbf{L3}}\label{L3} \hspace{0.175cm}If $i\colon A\rightarrow B$ and $p\colon B\rightarrow C$ are morphisms in $\CC$ such that $i$ has a cokernel and $pi$ is an inflation, then $i$ is an inflation.
\end{enumerate}
A right exact category satisfying \ref{R3} is called \emph{strongly right exact} or \emph{strongly deflation-exact}.  Dually, a left exact category satisfying \ref{L3} is called \emph{strongly left exact} or \emph{strongly inflation-exact}.
\end{definition}

\begin{remark}\label{remark:DefExactAndQuillenObscureAxiom} 
	\begin{enumerate} 
	\item An exact category in the sense of Quillen (see \cite{Quillen73}) is a conflation category $\CC$ satisfying axioms \ref{R0} through \ref{R3} and \ref{L0} through \ref{L3}. In \cite[appendix~A]{Keller90}, Keller shows that axioms \ref{R0}, \ref{R1}, \ref{R2}, and \ref{L2} suffice to define an exact category.
	\item Axioms \ref{R3} and \ref{L3} are sometimes referred to as Quillen's \emph{obscure axioms} (see \cite{Buhler10,ThomasonTrobaugh90}). 
	\item In \cite{Rump11}, the notions of one-sided exact categories includes the obscure axiom.
	\end{enumerate}
\end{remark}

\begin{remark} {A deflation-exact category $\CC$ satisfies axiom \ref{R0*} if and only if every split kernel-cokernel pair is a conflation.}
\end{remark}

\begin{lemma}\label{Lemma:BasicProperties}
	Let $\CC$ be a deflation-exact category. Then:
	\begin{enumerate}
		\item Every isomorphism is a deflation.
		\item If $\CC$ is strongly deflation-exact, then $\CC$ satisfies \upshape{\ref{R0*}}.
		\item \itshape Every inflation is a monomorphism.  An inflation which is an epimorphism is an isomorphism.
		\item Every deflation is an epimorphism.  A deflation which is a monomorphism is an isomorphism.
	\end{enumerate}
\end{lemma}

\begin{proof}
	\begin{enumerate}
		\item Let $f\colon X \to Y$ be an isomorphism.  One easily checks that \[\xymatrix{X \ar[r]^f\ar[d] & Y\ar[d]\\
		0 \ar[r] & 0}\]
		is a pullback diagram.  By \ref{R0} and \ref{R2}, we know that $f\colon X \to Y$ is a deflation.
		\item Since $1_A\colon A\rightarrow A$ is the kernel of $p\colon A\rightarrow 0$ and the composition of $0\xrightarrow{i} A \xrightarrow{p} 0$ is a deflation by \ref{R0}, it follows that $p$ is a deflation.  This establishes \ref{R0*}.
		\item Every inflation is a kernel and kernels are monic.  If an inflation is an epimorphism, then the cokernel is zero.  As an inflation is the kernel of its cokernel, we infer that the inflation is an isomorphism.
		\item Similar.\qedhere
	\end{enumerate}
\end{proof}

\begin{proposition}\label{proposition:PushoutIfCokernel}\label{proposition:WhenPushout}\label{proposition:WhenPullback}\label{proposition:PullbackPushout}\label{proposition:InducedSequenceLemma}
	Let $\CC$ be a deflation-exact category.
\begin{enumerate}
\item For a commutative square
	\[\xymatrix{
		A\ar@{>->}[r]^{i'} \ar[d]^f & B \ar[d]^g\\
		A' \ar@{>->}[r]^{i} & B'
	}\] where the horizontal arrows are inflations, the following statements are equivalent:
		\begin{enumerate}
			\item\label{enumerate:CommutativeSquare1Pushout} the square is a pushout,
			\item\label{enumerate:CommutativeSquare1Bicartesian} the square is both a pushout and a pullback,
			\item\label{enumerate:CommutativeSquare1Extension} the square can be extended to a diagram
				\[\xymatrix{
			A\ar@{>->}[r]^{i} \ar[d]^f & B\ar@{->>}[r] \ar[d]^g & C \ar@{=}[d] \\
		A' \ar@{>->}[r]^{i'} & B' \ar@{->>}[r] & C
	}\]
	where the rows are conflations,
	    \item\label{enumerate:CommutativeSquare1Conflation} the induced sequence $\xymatrix@1{A\ar[r]^-{\begin{psmallmatrix}f\\ i\end{psmallmatrix}} & A'\oplus B\ar[r]^-{\begin{psmallmatrix}-i' & g\end{psmallmatrix}} & B'}$ is a conflation.
		\end{enumerate}
\item\label{enumerate:SecondPart} For a commutative square
	\[\xymatrix{
		B\ar@{->>}[r]^{p} \ar[d]^f & C \ar[d]^g\\
		B' \ar@{->>}[r]^{p'} & C'
	}\]
	where the horizontal arrows are deflations, the following statements are equivalent:
		\begin{enumerate}
			\item\label{enumerate:CommutativeSquare2Pullback} the square is a pullback,
			\item\label{enumerate:CommutativeSquare2Bicartesian} the square is both a pushout and a pullback,
			\item\label{enumerate:CommutativeSquare2Extension} the square can be extended to a diagram
		\[\xymatrix{
		A\ar@{>->}[r] \ar@{=}[d] & B \ar[d]^f \ar@{->>}[r]^{p} & C\ar[d]^g \\
		A' \ar@{>->}[r] & B'	\ar@{->>}[r]_{p'} & C'}\]
	where the rows are conflations.
	\end{enumerate}
	If $\CC$ satisfies axiom \ref{R0*}, then the previous are equivalent to:
	 \begin{enumerate}[resume]
		 \item\label{enumerate:CommutativeSquare2Conflation} the induced sequence $\xymatrix@1{B\ar[r]^-{\begin{psmallmatrix}f\\ p\end{psmallmatrix}} & B'\oplus C\ar[r]^-{\begin{psmallmatrix}-p' & g\end{psmallmatrix}} & C'}$ is a conflation.
	 \end{enumerate}
\end{enumerate}
\end{proposition}

\begin{proof}
\begin{enumerate}
\item The implication $\eqref{enumerate:CommutativeSquare1Pushout} \Rightarrow \eqref{enumerate:CommutativeSquare1Extension}$ is straightforward to prove.  For the reverse implication, one can verify that the proof of \cite[proposition~2.12]{Buhler10} still holds.  If \eqref{enumerate:CommutativeSquare1Extension} holds, then proposition \ref{proposition:MitchellPullbackPushout} shows that the given commutative square is a pullback.  This shows that \eqref{enumerate:CommutativeSquare1Bicartesian} holds.  The implication $\eqref{enumerate:CommutativeSquare1Bicartesian} \Rightarrow \eqref{enumerate:CommutativeSquare1Pushout}$ is trivial.

Assume now that \eqref{enumerate:CommutativeSquare1Extension} holds.  Consider the following commutative diagram:
	\[\xymatrix{
		& A'\ar@{=}[r]\ar[d]^{\begin{psmallmatrix}-1_{A'}\\0\end{psmallmatrix}} & A'\ar@{>->}[d]^{i'}\\
		A\ar[r]^-{\begin{psmallmatrix}f\\ i\end{psmallmatrix}}\ar@{=}[d] & A'\oplus B\ar[r]^-{\begin{psmallmatrix}-i' & g\end{psmallmatrix}}\ar[d]^{\begin{psmallmatrix}0&1_B\end{psmallmatrix}} & B'\ar@{->>}[d]^{p'}\\
		A\ar@{>->}[r]^i & B\ar@{->>}[r]^p & C	
	}\] One readily verifies that the lower right square is a pullback square. Axiom \ref{R2} implies that the middle row is a conflation as required.  This establishes the implication $\eqref{enumerate:CommutativeSquare1Extension} \Rightarrow \eqref{enumerate:CommutativeSquare1Conflation}.$  The implication $\eqref{enumerate:CommutativeSquare1Conflation} \Rightarrow \eqref{enumerate:CommutativeSquare1Pushout}$ is trivial.
\item The equivalence $\eqref{enumerate:CommutativeSquare2Pullback} \Leftrightarrow \eqref{enumerate:CommutativeSquare2Extension}$ is \cite[proposition~5.4]{BazzoniCrivei13}. The equivalence $\eqref{enumerate:CommutativeSquare2Pullback} \Leftrightarrow \eqref{enumerate:CommutativeSquare2Bicartesian}$ again follows from proposition \ref{proposition:MitchellPullbackPushout}.

Assume now that $\CC$ satisfies axiom \ref{R0*}.  It is shown in \cite[proposition~5.7]{BazzoniCrivei13} that $\eqref{enumerate:CommutativeSquare2Pullback} \Leftrightarrow \eqref{enumerate:CommutativeSquare2Conflation}$.\qedhere
\end{enumerate}
\end{proof}

\begin{remark}
In proposition \ref{proposition:PullbackPushout}, the implication $\eqref{enumerate:CommutativeSquare2Conflation} \Rightarrow \eqref{enumerate:CommutativeSquare2Pullback}$ does not use axiom \ref{R0*}.  In contrast, the converse $\eqref{enumerate:CommutativeSquare2Pullback} \Rightarrow \eqref{enumerate:CommutativeSquare2Conflation}$ needs axiom \ref{R0*}.  To see this, let $C \in \CC$ be any object.  By proposition \ref{Lemma:BasicProperties}, the diagram
\[\xymatrix{
		C\ar@{=}[r] \ar[d] & C \ar[d]\\
		0 \ar@{=}[r] & 0
	}\]
	satifies the properties in \eqref{enumerate:CommutativeSquare2Pullback}.  The implication $\eqref{enumerate:CommutativeSquare2Pullback} \Rightarrow \eqref{enumerate:CommutativeSquare2Conflation}$ implies that $C \to 0$ is a deflation.  This implies that axiom \ref{R0*} holds.
\end{remark}

\begin{proposition}\label{proposition:FactorizationOfConflationMorphism}
Let $\CC$ be a deflation-exact category.  Every morphism $(f,g,h)$ between conflations $X \stackrel{i}{\rightarrowtail} Y \stackrel{p}{\twoheadrightarrow} Z$ and  $X' \stackrel{i'}{\rightarrowtail} Y' \stackrel{p'}{\twoheadrightarrow} Z'$ factors through some conflation $X' {\rightarrowtail} P {\twoheadrightarrow} Z$:
\[\xymatrix{
X\ar@{>->}[r]^i\ar[d]^f & Y\ar@{->>}[r]^p\ar[d] & Z\ar@{=}[d] \\
X'\ar@{=}[d]\ar@{>->}[r] & P \ar@{->>}[r]\ar[d] & Z \ar[d]^h \\
X'\ar@{>->}[r]^{i'} & Y'\ar@{->>}[r]^{p'} & Z'}\]
such that the upper-left and lower-right squares are both pullbacks and pushouts.
\end{proposition}

\begin{proof}
The factorization property is \cite[proposition 5.2]{BazzoniCrivei13}. The statements about the pushouts and pullbacks follow from proposition \ref{proposition:PushoutIfCokernel}.
\end{proof}

\begin{proposition}\label{proposition:PullbackOfInflation}
Let $\CC$ be a deflation-exact category.  The pullback of an inflation $f$ along a deflation is an inflation $f'$.
\end{proposition}

\begin{proof}
Let $f\colon X \rightarrowtail Z$ be an inflation and $g\colon Y \twoheadrightarrow Z$ be a deflation.  Consider the commutative diagram
\[\xymatrix{
P \ar@{.>}[r]^{f'} \ar@{.>}[d]^{g'} & Y \ar@{->>}[d]^{g} \\ X \ar@{>->}[r]^{f} & Z \ar@{->>}[r]^-{h} & {\coker(f)}
}\]
where the square is a pullback diagram and the bottom row is a conflation.  It follows from proposition \ref{proposition:MitchellPullbackPushout}(\ref{enumerate:MitchellPullback}) that $f'$ is the kernel of the composition $Y \twoheadrightarrow Z \twoheadrightarrow \coker(f)$ and hence an inflation by axiom \ref{R1}.
\end{proof}

The following proposition provides a sufficient condition for a subcategory of a deflation-exact category to be deflation-exact.

\begin{proposition}\label{proposition:DeflationClosed}
Let $\CC$ be a deflation-exact category.  Let $\DD \subseteq \CC$ be a full subcategory. The conflation structure of $\CC$ induces a deflation-exact structure on $\DD$ if for every deflation $f\colon Y\to Z$ in $\CC$ with $Y,Z\in \DD$, one has that $\ker(f)\in \DD$.
	\begin{enumerate}
		\item If $\CC$ satisfies axiom \ref{R0*}, then so does $\DD$.
		\item If $\CC$ satisfies axiom \ref{R3}, then so does $\DD$.
	\end{enumerate}
\end{proposition}

\begin{proof}
	The only non-trivial part is to show that $\DD$ inherits axiom \ref{R3} from $\CC$. To that end, let $K_D\to Y\xrightarrow{p}Z$ be a sequence in $\DD$ such that $K_D\to Y$ is the kernel of $p$ and let $i\colon X\to Y$ be a map in $\DD$ such that $p\circ i$ is a deflation. Write $P$ for the pullback of $pi$ along $p$ in $\CC$. By proposition \ref{proposition:PushoutIfCokernel}, we obtain a conflation $P\inflation Y\oplus X\deflation Z$. It follows that $P\in \DD$.  As the kernel of $\begin{pmatrix}p & pi \end{pmatrix} \colon Y \oplus X \deflation Z$ in $\DD$ is $K_D \oplus X$, one finds that $P\cong K_D\oplus X$. It follows that $K_D\to Y\to Z$ is a conflation by \cite[proposition~5.9]{BazzoniCrivei13} in $\CC$. In particular $K_D$ is the kernel of $Y\to Z$ in $\CC$. This completes the proof.
\end{proof}

%% file: PercolatingSubcategories.tex
\section{Percolating subcategories}\label{Section:PercolatingSubcategories}

Let $\CC$ be a one-sided exact category.  In this section, we define the notion of a percolating subcategory of $\CC$. To place this notion in context: if the category $\CC$ is abelian, a subcategory $\AA \subseteq \CC$ is percolating if and only if it is a Serre subcategory; if $\CC$ is an exact category, then the notion of a percolating subcategory is weaker than the notion of a right s-filtering subcategory in \cite{Schlichting04} (this will be verified in proposition \ref{proposition:RecoveredSchlichtingsFramework}).

Starting from a percolating subcategory $\AA \subseteq \CC$, we define a set of weak isomorphisms $S_\AA$.  This set is a left or right multiplicative set (see proposition \ref{proposition:MinimalConditionsRMS}).  We will proceed to establish some techincal results which will help to understand the localization $\CC[S_\AA^{-1}],$ chief among them lemma \ref{lemma:LiftingConflations} and proposition \ref{proposition:InterpretationOfP4}.

\subsection{Definitions and basic properties}
We start by defining percolating subcategories. As this definition does not refer to the deflation-exact structure of $\CC$, we formulate the definition for a more general conflation category.

\begin{definition}\label{Definition:GeneralPercolatingSubcategory}
	Let $\CC$ be a conflation category. A non-empty full subcategory $\AA$ of $\CC$ is called a \emph{right percolating subcategory} or a \emph{deflation-percolating subcategory} of $\CC$ if the following axioms are satisfied:
	\begin{enumerate}[label=\textbf{P\arabic*},start=1]
		\item\label{P1} $\AA$ is a \emph{Serre subcategory}, meaning:
		\[\mbox{ If } A'\rightarrowtail A \twoheadrightarrow A'' \mbox{ is a conflation in $\CC$, then } A\in \Ob(\AA) \mbox{ if and only if } A',A''\in \Ob(\AA).\]
		\item\label{P2} For all morphisms $C\rightarrow A$ with $C \in \Ob(\CC)$ and $A\in \Ob(\AA)$, there exists a commutative diagram
		\[\xymatrix{
		A'\ar[rd] & \\
		C \ar@{->>}[u]\ar[r]& A\\
				}\] with $A'\in \Ob(\AA)$ and where $C \twoheadrightarrow A'$ is a deflation.
		\item\label{P3} For any composition $\xymatrix{X\ar@{>->}[r]^i & Y\ar[r]^t & T}$ which factors through $\AA$, there exists a commutative diagram 
		\[\xymatrix{
			X\ar@{>->}[r]^i\ar@{->>}[d]^f & Y\ar@{->>}[d]^{f'}\ar@/^/[rdd]^t &\\
			A\ar@{>->}[r]^{i'}\ar@/_/[rrd] & P\ar@{.>}[rd] &\\
			&& T
		}\] with $A \in \Ob(\AA)$ and such that the square $XYAP$ is a pushout square.	
		\item\label{P4} For all maps $X\stackrel{f}{\rightarrow} Y$ that factor through $\AA$ and for all inflations $A\stackrel{i}{\inflation} X$ (with $A \in \Ob(\AA)$) such that $f\circ i=0$, the induced map $\coker(i)\to Y$ factors through $\AA$.
	\end{enumerate}
	
	By dualizing the above axioms one obtains a similar notion of a \emph{left percolating subcategory} or an \emph{inflation-percolating subcategory}.
\end{definition}

\begin{remark}
In axiom \ref{P3}, we start with a composition $t \circ i$ factoring through an object $B \in \Ob(\AA).$  We do not require any compatibility between this object $B \in \Ob(\AA)$ and the object $A \in \Ob(\AA)$ in the diagram occurring in the statement of axiom \ref{P3}.
\end{remark}

\begin{definition}\makeatletter
\hyper@anchor{\@currentHref}%
\makeatother\label{definition:AdditionalPercolatingDefinitions}
	  \begin{enumerate}
			\item Following the conventions by \cite{Schlichting04}, a non-empty full subcategory $\AA$ of a conflation category $\CC$ satisfying axioms \ref{P1} and \ref{P2} is called \emph{right filtering}.
			\item If $\AA$ is a right filtering subcategory of $\CC$ such that the map $A'\rightarrow A$ in axiom \ref{P2} can be chosen as a monic map, we will call $\AA$ a \emph{strongly right filtering subcategory}.
			\item A right percolating subcategory which is also strongly right filtering will be abbreviated to a \emph{strongly right percolating subcategory} or \emph{strongly deflation-percolating subcategory}.
	  \end{enumerate}
	The notions of a \emph{left filtering}, \emph{strongly left filtering}, and \emph{strongly left percolating subcategory} are defined dually.
\end{definition}

\begin{remark}\label{remark:PercolatingRemarks}
	\begin{enumerate}
			\item Any deflation-exact category is a deflation-percolating subcategory of itself.
			\item If $\CC$ is an exact category, then any subcategory $\AA$ satisfying axiom \ref{P2}  automatically satisfies axiom \ref{P3} (see \cite[proposition~2.15]{Buhler10}).
			\item As a deflation-percolating subcategory $\AA$ of $\CC$ is closed under extensions in $\CC$ (by axiom \ref{P1}), the conflations of $\CC$ induce a deflation-exact structure on $\AA$.
	\end{enumerate}
\end{remark}

\subsection{Weak isomorphisms}\label{subsection:WeakIsomorphisms}

Let $F\colon \CC \to \DD$ be an exact functor between conflation categories.  Let $\AA \subseteq \CC$ be a full subcategory and assume that $F(\AA) = 0$.  It is clear that, for any conflation $X\stackrel{f}{\rightarrowtail} Y \stackrel{g}{\twoheadrightarrow} Z$, we have that $X \in \Ob(\AA)$ implies that $F(g)$ is an isomorphism.  Likewise, $Z \in \Ob(\AA)$ implies that $F(f)$ is an isomorphism. This observation motivates the following definition (the terminology is based on \cite{Cardenas98,Schlichting04}).

\begin{definition}\label{Definition:WeakIsomorphisms}
	Let $\CC$ be a conflation category and let $\AA$ be a non-empty full subcategory of $\CC$. 	
	\begin{enumerate}
		\item	An inflation $f\colon X\rightarrowtail Y$ in $\CC$ is called an \emph{$\AA^{-1}$-inflation} if its cokernel belongs to $\AA$.
		\item A deflation $f\colon X\twoheadrightarrow Y$ in $\CC$ is called a \emph{$\AA^{-1}$-deflation} if its kernel belongs to $\AA$.
		\item A morphism $f\colon X\rightarrow Y$ is called a \emph{weak $\AA^{-1}$-isomorphism} (or simply a \emph{weak isomorphism} if $\AA$ is implied) if it is a finite composition of $\AA^{-1}$-inflations and $\AA^{-1}$-deflations. We often endow weak isomorphisms with ``$\sim$''.
	\end{enumerate}
	The set of weak isomorphisms is denoted by $S_{\AA}$. Given a weak isomorphism $f$, the \emph{composition length} of $f$ is defined as the smallest natural number $n$ such that $f$ can be written as a composition of $n$ $\AA^{-1}$-inflations or $\AA^{-1}$-deflations. 
\end{definition}

The following proposition is a straightforward strengthening of \cite[lemma~1.13]{Schlichting04}.

\begin{proposition}\label{proposition:MinimalConditionsRMS}
	Let $\CC$ be a deflation-exact category and let $\AA$ be a right filtering subcategory. The set $S_{\AA}$ of weak isomorphisms is a right multiplicative system. Moreover, every solid diagram 
	\[\xymatrix{
		X\ar@{.>}[r]^g\ar@{.>}[d]_{\rotatebox{90}{$\sim$}}^t & Y\ar[d]_{\rotatebox{90}{$\sim$}}^s\\
		Z\ar[r]^f & W
	}\] with $s\in S_{\AA}$ can be completed to a commutative square such that $t\in S_{\AA}$ and the composition length of $t$ is at most the composition length of $s$.\\
	If $\AA$ is a strongly right filtering subcategory, then the square in axiom \ref{RMS2} can be chosen as a pullback-square.
\end{proposition}

\begin{proof}
Axiom \ref{RMS1} is trivial.  For axiom \ref{RMS2} (using notation as in definition \ref{definition:RMS}), we can easily reduce to the case where $s\colon Y \stackrel{\sim}{\rightarrow} W$ is either a deflation or an inflation.  In the former case, we can complete the diagram by taking a pullback.  In the latter case, we can use axiom \ref{P2} to factor the composition $Z \stackrel{f}{\rightarrow} W \to \coker s$ as $Z \stackrel{\alpha}{\deflation} A \stackrel{\beta}{\rightarrow} \coker s$; the morphism $t\colon \ker \alpha \stackrel{\sim}{\inflation} Z$ then completes the diagram.

If $\AA \subseteq \CC$ is a strongly right filtering subcategory, then we can choose $\beta$ to be a monomorphism, and the square we obtained before is a pullback as well (this follows from proposition \ref{proposition:MitchellPullbackPushout}).

For axiom \ref{RMS3}, consider a composition $X \stackrel{f}{\rightarrow} Y \stackrel{s}{\rightarrow} Z$ with $s \in S$.  Assume that $s\circ f=0$.  We need to show that there is a $t \in S$ such that the composition $W \stackrel{t}{\rightarrow}X \stackrel{f}{\rightarrow} Y$ is zero.  Again, we can easily reduce to the case where $s$ is either an inflation or a deflation.

If $s$ is an inflation, then $f = 0$ so that we can choose $t = 1_X$.  If $s$ is a deflation, then $f \circ s = 0$ shows that $f$ factors as $X \stackrel{\alpha}{\rightarrow} \ker(s) \stackrel{\beta}{\rightarrow} Y$.  As $\ker(s) \in \AA$, it follows from axiom \ref{P2} that there is an inflation $t\colon W \stackrel{\sim}{\inflation} X$ such that $f \circ t = 0.$
\end{proof}

\begin{proposition}\label{proposition:WeakIsoPullback}
Let $f\colon X \deflation Y$ be a deflation in a deflation-exact category $\CC$.  For any weak isomorphism $s\colon Z \stackrel{\sim}{\rightarrow} Y$, the pullback along $f$ is a weak isomorphism.
\end{proposition}

\begin{proof}
This follows from propositions \ref{proposition:PushoutIfCokernel} and \ref{proposition:PullbackOfInflation}, and the pullback lemma.
\end{proof}

\begin{lemma}\label{lemma:CompositionOfAAdeflations}
	Let $\CC$ be a deflation-exact category and let $\AA\subseteq \CC$ be a non-empty full subcategory. If $\AA$ satisfies axiom \ref{P1}, the composition of two $\AA^{-1}$-deflations is again an $\AA^{-1}$-deflation.
\end{lemma}

\begin{proof}
Let $U \stackrel{a}{\rightarrow} V \stackrel{b}{\rightarrow}W$ be $\AA^{-1}$-deflations.  Axiom \ref{R1} shows that $ba\colon U \to W$ is a deflation.  Propositions \ref{proposition:MitchellPullback} and \ref{proposition:WhenPullback} now yield the following commutative diagram:
		\[\xymatrix{
		\ker(a')\ar@{>->}[d]^{k_a'}\ar@{=}[r] & \ker(a)\ar@{>->}[d]^{k_a} &\\
		P\ar@{>->}[r]^{k_{ab}}\ar@{->>}[d]^{a'} & U\ar@{->>}[d]^a \ar@{->>}[r]^{ba} & W \ar@{=}[d] \\
		\ker(b)\ar@{>->}[r]^{k_b} & V\ar@{->>}[r]^b &  W\\
		}\]  where the rows and columns are conflations, and the lower-left square is a pullback.  As $\ker(a),\allowbreak \ker(b) \in \Ob(\AA)$, axiom \ref{P1} implies that $P \in \Ob(\AA)$. Proposition \ref{proposition:MitchellPullbackPushout}\eqref{enumerate:MitchellPullback} implies that $P=\ker(ba)$. It follows that $ba \in S_{\AA}$, as required.
\end{proof}

\subsection{The lifting lemma}

The following crucial lemma allows one to lift conflations $X \inflation Y \deflation$ over a weak isomorphism $Y' \stackrel{\sim}{\rightarrow} Y$. 

\begin{lemma}[Lifting lemma]\label{lemma:LiftingConflations}
	Let $\CC$ be a deflation-exact category and let $\AA\subseteq \CC$ be a deflation-percolating subcategory. Given a conflation $X\stackrel{i}{\inflation} Y \stackrel{p}{\deflation} Z$ and a weak isomorphism $s\colon Y''\stackrel{\sim}{\rightarrow} Y$, there exists a weak isomorphism $t\colon \overline{Y} \to Y$, factoring through $s$, such that there is a commutative diagram
	\[\xymatrix{
		\overline{X}\ar@{>->}[r]^{\overline{i}}\ar[d]^{\rotatebox{90}{$\sim$}} & \overline{Y}\ar@{->>}[r]^{\overline{p}}\ar[d]^{\rotatebox{90}{$\sim$}}_t & \overline{Z}\ar[d]^{\rotatebox{90}{$\sim$}}\\
		X\ar@{>->}[r]^i & Y\ar@{->>}[r]^p & Z
	}\] where the rows are conflations and the vertical maps are weak isomorphisms.
\end{lemma}

\begin{proof}
	We first consider two cases. 
	\begin{enumerate}
		\item[Case I] Assume that $s$ factors as $\xymatrix{Y''\ar[r]^{\sim}_{s''} & Y'\ar@{->>}[r]^{\sim}_{s'} & Y}$. By axiom \ref{R2}, the pullback of $i$ along $s'$ exists.  By propositions \ref{proposition:WhenPullback} and \ref{proposition:PullbackOfInflation}, we obtain the following commutative diagram:
		\[\xymatrix{
		X'\ar@{>->}[r]\ar@{->>}[d]^{\rotatebox{90}{$\sim$}} & Y'\ar@{->>}[d]^{\rotatebox{90}{$\sim$}}_{s'}\ar@{->>}[r] & Z\ar@{=}[d]\\
		X\ar@{>->}[r]^i & Y\ar@{->>}[r]^p & Z
		}\] Thus, taking the pullback along $s'$, one can lift the conflation $(i,p)$ over $s'$.
		
		\item[Case II] Assume that $s$ factors as $\xymatrix{Y''\ar[r]^{\sim}_{s''} & Y'\ar@{>->}[r]^{\sim}_{s'} & Y}$.  Write $g\colon Y \deflation A'$ for the cokernel of $s'$. Applying axiom \ref{P3} to the composition $X\stackrel{i}{\inflation}Y\stackrel{g}{\deflation}A'$ yields a commutative diagram
		\[\xymatrix{
			\overline{X}\ar@{=}[r]\ar@{>->}[d]^{\rotatebox{90}{$\sim$}} & \overline{X}\ar@{>->}[d] & \\
			X\ar@{>->}[r]^i\ar@{->>}[d] & Y\ar@{->>}[r]^p\ar@{->>}[d] & Z\ar@{=}[d]\\
			A\ar@{>->}[r] & P\ar@{->>}[r]^{\sim}\ar[d] & Z\\
			& A' & 
		}\] such that the lower-left square is bicartesian and the composition $Y\deflation P\to A'$ equals $g$. By axiom \ref{P2}, the map $P\to A'$ factors as $P\deflation B\to A'$ with $B\in \Ob(\AA)$. Write $\overline{Z}\stackrel{\sim}{\inflation}P$ for the kernel of $P\deflation B$. Taking the pullback of $\overline{Z}\stackrel{\sim}{\inflation}P$ along $Y\deflation P$ yields the following commutative diagram:
		\[\xymatrix{
			\overline{X}\ar@{>->}[r]\ar@{=}[d] & \overline{Y}\ar@{->>}[r]\ar@{>->}[d]^{\rotatebox{90}{$\sim$}}_t & \overline{Z}\ar@{>->}[d]^{\rotatebox{90}{$\sim$}} \\
			\overline{X}\ar@{>->}[r]\ar@{>->}[d]^{\rotatebox{90}{$\sim$}} & Y\ar@{->>}[r]\ar@{=}[d] & P\ar@{->>}[d]^{\rotatebox{90}{$\sim$}}\\
			X\ar@{>->}[r]^i & Y\ar@{->>}[r]^p & Z
		}\] As the upper-right square is bicartesian, the map $t\colon\overline{Y}\stackrel{\sim}{\inflation}Y$ is indeed an $\AA^{-1}$-inflation. One readily verifies that the composition $\overline{Y} \inflation Y \deflation P \deflation Y \deflation B \to A'$ is zero, and hence that $t\colon\overline{Y}\stackrel{\sim}{\inflation}Y$ factors through $s' = \ker (Y \deflation A')$ via a map $u\colon \overline{Y}\to Y'$.
		
		By proposition \ref{proposition:MinimalConditionsRMS}, we obtain a commutative square
		\[\xymatrix{
			{\overline{Y}''}\ar@{..>}[r]^{\sim}_{t''}\ar@{..>}[d]_{u'} & \overline{Y}\ar[d]^u \ar@/^/[dr]^t\\
			{Y''}\ar[r]^{\sim}_{s''} & Y' \ar@{>->}[r]_{s'}&Y}\]
			such that the length of $t''$ is bounded by the length of $s''$. Thus, we have lifted the conflation $X\inflation Y \deflation Z$ over $s'$ to a conflation $\overline{X}\inflation \overline{Y}\deflation \overline{Z}$ and we have replaced $s''$ by $t''$.
		\end{enumerate}	
	The result follows by induction on the composition length of $s$.
\end{proof}

\subsection{Interpretation of axiom \ref{P4}}

Thus far, we have only used axioms \ref{P1} through \ref{P3} of a deflation-percolating subcategory.  The following proposition highlights the r{\^o}le of axiom \ref{P4}.

\begin{proposition}\label{proposition:InterpretationOfP4}
	Let $\AA\subseteq \CC$ be a deflation-percolating subcategory of a deflation-exact category $\CC$. Let $(f\colon X'\to Y,s\colon X'\to X)$ be a morphism in $S_{\AA}^{-1}\CC$. The following are equivalent:
		\begin{enumerate}
			\item\label{item:InterpretationOfP4A} $(f,s)=0$ in $S_{\AA}^{-1}\CC$,
			\item\label{item:InterpretationOfP4B} $f$ factors through $\AA$ in $\CC$,
			\item\label{item:InterpretationOfP4C} there exists an $\AA^{-1}$-inflation $t$ such that $f\circ t=0$ in $\CC$.			
		\end{enumerate}
\end{proposition}

\begin{proof}
	Assume that \eqref{item:InterpretationOfP4A} holds.  We will show that \eqref{item:InterpretationOfP4B} holds.  Clearly, $f\circ s^{-1}=0$ in $S^{-1}_\AA\CC$ and $s$ is an isomorphism in $S_{\AA}^{-1}\CC$. It follows that $Q(f)=0$.  By construction \ref{construction:Localization}, there is a weak isomorphism $u\colon M\to X'$ such that $f\circ u=0$ in $\CC$. As the zero object belongs to $\AA$, the composition $f\circ u$ factors through $\AA$.  If $u$ is an isomorphism in $\CC$, then \eqref{item:InterpretationOfP4B} holds.  Otherwise, the composition length of $u$ is at least one and one of the following cases hold.
	
		\begin{enumerate}
			\item[Case I] Assume that $u$ factors as $\xymatrix{M\ar@{->>}[r]^{\sim}_{u''} & M'\ar[r]^{\sim}_{u'} & X'}$. Note that $\ker(u'')\in \Ob(\AA)$, hence axiom \ref{P4} implies that $f\circ u'$ factors through $\AA$.
			\item[Case II] Assume that $u$ factors as $\xymatrix{M\ar@{>->}[r]^{\sim}_{u''} & M'\ar[r]^{\sim}_{u'} & X'}$. As $f\circ u$ factors through $\AA$, axiom \ref{P3} yields the following commutative diagram:
			\[\xymatrix{
				M\ar@{>->}[r]^{\sim}_{u''}\ar@{->>}[d] & M'\ar[r]^{\sim}_{u'}\ar@{->>}[d] & X'\ar[d]^f\\
				A\ar@{>->}[r]^{\sim} & Q\ar[r] & Y\\
			}\] where the left square is a pushout square.  By proposition \ref{proposition:PullbackPushout}, we know that $Q/A \cong M'/M \in \Ob(\AA)$, so that, by axiom \ref{P1}, $Q\in \Ob(\AA)$.  It follows that $f\circ u'$ factors through $\AA$.
		\end{enumerate}
		Iterating the previous two cases, we conclude that \eqref{item:InterpretationOfP4B} holds.
		
		Assume that \eqref{item:InterpretationOfP4B} holds. By axiom \ref{P2}, we may assume that $f\colon X' \to Y$ factors as $X' \deflation A \to Y$ with $A \in \AA$. The kernel of the deflation $X' \deflation A$ is the desired $\AA^{-1}$-inflation and, hence, \eqref{item:InterpretationOfP4C} holds. The implication $\eqref{item:InterpretationOfP4C}\Rightarrow\eqref{item:InterpretationOfP4A}$ is trivial. 
\end{proof}

We now give a useful criterion to verify that axiom \ref{P4} holds.

\begin{proposition}\label{proposition:P4Criterion}
Let $\EE$ be a deflation-exact category.  Let $\AA \subseteq \EE$ a nonempty full subcategory. If every morphism $a\colon A \to B$ in $\AA$ admits a cokernel in $\CC$ (with $\coker a \in \AA$), then $\AA$ satisfies axiom \ref{P4}.
\end{proposition}

\begin{proof}
Let $f\colon X\to Y$ be a map that factors as $X\stackrel{\rho}{\rightarrow}B\stackrel{h'}{\rightarrow}Y$ with $B\in \Ob(\AA)$ and let $A\stackrel{i}{\inflation}X \stackrel{p}{\deflation} Q$ be a conflation with $A\in \Ob(\AA)$ such that $fi=0$, as in the setup of axiom \ref{P4}.  The cokernel property of $p$ induces a unique map $h\colon Q\to Y$ such that $f=hp$.  We need to show that $h\colon Q \to Y$ factors through $\AA.$  By assumption, the composition $\rho i\colon A \to B$ admits a cokernel $p'\colon B \to C$ with $C\in \Ob(\AA)$.  We obtain the following commutative diagram:
	\[\xymatrix{
		A\ar@{>->}[r]^i\ar@{=}[d] & X\ar@{->>}[r]^p\ar[d]^{\rho} & Q\ar[r]^h\ar@{.>}[d]^{\rho'} & Y\ar@{=}[d]\\
		A\ar[r]^{\rho i} & B\ar[r]^{p'}\ar@/_2pc/[rr]_{h'} & C\ar@{.>}[r]^{\exists ! u} & Y
	}\] Here, the map $\rho'$ is induced by the cokernel property of $p$. As $h'\rho i=fi=0$, the cokernel property of $p'$ induces a unique map $u\colon C\to Y$ such that $h'=up'$.  It follows that $hp=f=h'\rho=up'\rho=u\rho'p$ and hence $h=u\rho'$ since $p$ is an epimorphism. We conclude that $h$ factors through $\AA$ as required.
\end{proof}

%% file: LocalizationOfOneSidedExactCategories.tex
\section{Quotients and localizations of one-sided exact categories}\label{Section:LocalizationOfOneSidedExactCategories}

Throughout this section, let $\CC$ denote a deflation-exact category and $\AA$ a deflation-percolating subcategory. We write $S_\AA$ for the corresponding set of weak isomorphisms (see definition \ref{Definition:WeakIsomorphisms}).  The aim of this section is to show that $S_{\AA}^{-1}\CC$ has a canonical deflation-exact structure such that the localization functor $Q\colon \CC\rightarrow S_{\AA}^{-1}\CC$ is exact.  Moreover, we show that $S_{\AA}^{-1}\CC$ is universal in the sense of the following definition.

\begin{definition}\label{definition:RightExactLocalization}
	Let $\CC$ be a deflation-exact category and $\AA$ a full deflation-exact subcategory.  We define the \emph{quotient} of $\CC$ by $\AA$ as a deflation-exact category $\CC/\AA$ together with an exact \emph{quotient functor} $Q\colon \CC\rightarrow \CC/\AA$ satisfying the following universal property: for any exact functor $F\colon \CC\rightarrow \mathcal{D}$ of deflation-exact categories such that $F(A)\cong 0$ for all $A\in\Ob(\AA)$ there exists a unique exact functor $G\colon \CC/\AA\rightarrow \mathcal{D}$ such that the following diagram commutes:
	\[\xymatrix{
	\AA\ar[d]\ar[rd]^0 & \\
	\CC\ar[r]^F\ar[d]_Q & \mathcal{D}\\
	\CC/\AA\ar@{.>}[ru]_{G}
	}\]
\end{definition}

\begin{remark}\label{remark:SAInverts}
The next two observations motivate the definitions of axiom \ref{P1} and weak isomorphisms.
\begin{enumerate}
  \item Let $A\rightarrowtail X\twoheadrightarrow Y$ be a conflation in $\CC$ with $A\in \Ob(\AA)$. Then $0\rightarrowtail Q(X)\twoheadrightarrow Q(Y)$ is a conflation in $\CC/\AA$. It follows $Q(X)\deflation Q(Y)$ is invertible in $\CC/\AA$.  Similarly, if $X\rightarrowtail Y\twoheadrightarrow A$ is a conflation in $\CC$ with $A\in \Ob(\AA)$, then $Q(X)\inflation Q(Y)$ is invertible. In particular, all weak isomorphisms become isomorphisms under $Q$.
  \item The kernel of any exact functor $F\colon \CC \to \DD$ is a Serre subcategory of $\CC$, i.e.~it satisfies \ref{P1}.
\end{enumerate}
\end{remark}

Let $\AA$ be a deflation-percolating subcategory of a deflation-exact category $\CC$. The main theorem (theorem \ref{theorem:Maintheorem} below) states that the localization functor $Q\colon\CC\to S_{\AA}^{-1}\CC$ is a quotient functor. The proof consists of two major steps: in the first step, we endow $S_\AA^{-1}\CC$ with the structure of a conflation category such that $Q\colon \CC \to S_\AA^{-1}\CC$ is exact; in the second step, we show that the conflation category $S_\AA^{-1}\CC$ is a deflation-exact category.

\subsection{The \texorpdfstring{category $S_\AA^{-1}\CC$}{localized category} is a conflation category}

The next proposition allows us to impose a conflation structure on $S_\AA^{-1}\CC$ for which $Q\colon \CC \to S_\AA^{-1} \CC$ is exact (see definition \ref{definition:LocalizationConflation} below).

\begin{proposition}\label{proposition:CokernelsDescend}
Let $\CC$ be a deflation-exact category and let $\AA$ be a deflation-percolating subcategory. The localization functor $Q\colon \CC \to S_\AA^{-1}\CC$ maps conflations to kernel-cokernel pairs.
\end{proposition}

\begin{proof}
Let $X \stackrel{i}{\inflation} Y \stackrel{p}{\deflation} Z$ be a conflation in $\CC$.  As $S_{\AA}$ is a right multiplicative system (see proposition \ref{proposition:MinimalConditionsRMS}), we know that $Q(i)$ is the kernel of $Q(p)$.  We only need to show that $Q(p)$ is the cokernel of $Q(i)$.  For this, we consider the following diagram:
	\[\xymatrix{
		& T& &\\
		&Y''\ar[d]_{\rotatebox{90}{$\sim$}}^s\ar[u]_f&&\\
		X\ar@{>->}[r]^i&Y\ar@{->>}[r]^p&Z
	}\] where the composition $(f,s)\circ(i,1)$ is zero in $S_{\AA}^{-1}\CC$.  We will show that $(p,1)$ is the cokernel of $(i,1)$ by showing that, in $S_\AA^{-1}\CC$, the morphism $(f,s)$ factors uniquely through $(p,1)$.  Using the lifting lemma (lemma \ref{lemma:LiftingConflations}), we find the following diagram
	\[\xymatrix{
	&T&\\
		X'\ar@{>->}[r]^{i'}\ar[d]_{\rotatebox{90}{$\sim$}}& Y'\ar@{->>}[r]^{p'}\ar[d]_{\rotatebox{90}{$\sim$}}^{s'}\ar[u]_{f'}& Z'\ar[d]_{\rotatebox{90}{$\sim$}}\\
		X\ar@{>->}[r]^i&Y\ar[r]^{p}&Z
	}\] where the rows are conflations and where $(f,s) = (f',s')$. 	As $(f,s)\circ (i,1)$ is zero in $S_{\AA}^{-1}\CC$, we infer that $Q(f'i') = 0$. By proposition \ref{proposition:InterpretationOfP4}, the composition $f'i'$ factors through $\AA$.
	
	Applying axiom \ref{P3} to the composition $f'i'$ yields a commutative diagram
	\[\xymatrix{
		\overline{X}\ar@{=}[r]\ar@{>->}[d]^{\iota}_{\rotatebox{90}{$\sim$}} & \overline{X}\ar@{>->}[d]^{\iota'} &\\
		X'\ar@{>->}[r]^{i'}\ar@{->>}[d]^{\rho} & Y'\ar@{->>}[r]^{p'}\ar@{->>}[d]^{\rho'} & Z'\ar@{=}[d]\\
		A\ar@{>->}[r]_j & \overline{Z}\ar@{->>}[r]_{k}^{\sim}\ar[d]^{h'} & Z'\\
		& T & 
	}\] such that the rows and columns are conflations, where $A\in \Ob(\AA)$ and $h'\rho'=f'$. Thus we obtain the following commutative diagram:
	\[\xymatrix{
		& T &\\
		\overline{X}\ar@{>->}[r]^{\iota'}\ar@{>->}[d]^{\rotatebox{90}{$\sim$}}_{\iota} & Y'\ar@{->>}[r]^{\rho'}\ar@{=}[d]\ar[u]^{f'} & \overline{Z}\ar@{->>}[d]^{\rotatebox{90}{$\sim$}}_k\ar@{.>}@/_/[ul]_{m}\\
		X'\ar@{>->}[r]^{i'} & Y'\ar@{->>}[r]^{p'} & Z'
	}\] As $f'\iota'=0$, there is an induced map $m\colon \overline{Z}\to T$ such that $m\rho'=f'$. It follows that $(f,s)$ factors through $(p,1_Y)$ in $S_{\AA}^{-1}\CC$ as required.
	
	It remains to show that such a factorization is unique. It suffices to show that $Q(p)$ is an epimorphism in $S_{\AA}^{-1}\CC$. So let $(g,t)$ be map such that $(g,t)\circ Q(p)=0$. By proposition \ref{proposition:WeakIsoPullback}, we find the following diagram
\[\xymatrix{
X\ar@{>->}[r]^i\ar@{=}[d] &Y \ar@{->>}[r]^p & Z \\
X\ar@{>->}[r]^{i'} & Y' \ar@{..>}[u]_{\rotatebox{90}{$\sim$}}^{t'} \ar@{..>>}[r]^{p'}& Z'\ar[u]_{\rotatebox{90}{$\sim$}}^t \ar[r]^{g} & T
}\]
where the right square is a pullback and the vertical arrows are weak isomorphisms. Note that $Q(gp')=0$ and thus proposition \ref{proposition:InterpretationOfP4}.\eqref{item:InterpretationOfP4C} yields an $\AA^{-1}$-inflation $k\colon K \stackrel{\sim}{\inflation}Y'$ such that $gp'k=0$ in $\CC$. By the lifting lemma (lemma \ref{lemma:LiftingConflations}) we obtain a commutative diagram 
\[\xymatrix{
	& \overline{K}\ar@{->>}[r]^{\overline{p}}\ar[d]^{\rotatebox{90}{$\sim$}}_{\overline{k}} & \overline{Z}\ar[d]^{\rotatebox{90}{$\sim$}}_s & \\
	X\ar@{>->}[r]^{i'} & Y'\ar@{->>}[r]^{p'} & Z'\ar[r]^g & T
}\] where $\overline{k}$ factors through $k$. It follows that the composition $gs\overline{p}=0$ in $\CC$. As $\overline{p}$ is a deflation, $\overline{p}$ is epic and thus $gs=0$ in $\CC$. It follows that $Q(g)\circ Q(s)=Q(g\circ s)=0$ and since $Q(s)$ is an isomorphism, we find $Q(g)=0$ as required. This completes the proof.
\end{proof}

\begin{definition}\label{definition:LocalizationConflation}
Let $\AA$ be a deflation-percolating subcategory of a deflation-exact category $\CC$.  We say that a sequence $X \to Y \to Z$ is a conflation in $S_\AA^{-1}\CC$ if it is isomorphic (in $S_\AA^{-1}\CC$) to the image of a conflation under the localization functor $Q\colon \CC \to S_\AA^{-1}\CC$, i.e.~there is a conflation $\overline{X} \rightarrowtail \overline{Y} \twoheadrightarrow \overline{Z}$ in $\CC$ and a commutative diagram
\[\xymatrix{
Q(\overline{X}) \ar[r] \ar[d] & Q(\overline{Y}) \ar[r] \ar[d] & Q(\overline{Z}) \ar[d] \\
X \ar[r] & Y \ar[r] & Z
}\]
in $S_\AA^{-1}\CC$ where the vertical arrows are isomorphisms.
\end{definition}

\begin{remark}
	It follows from proposition \ref{proposition:CokernelsDescend} that definition \ref{definition:LocalizationConflation} endows $S_{\AA}^{-1}\CC$ with a conflation structure.	With this choice, the localization functor $Q\colon \CC \to S^{-1}_\AA \CC$ is conflation-exact. 
\end{remark}

\begin{proposition}\label{proposition:QuotientInConflationCategories}
Let $\CC$ be a deflation-exact category and let $\AA$ be a deflation-percolating subcategory.  When we endow $\CC[S^{-1}_\AA]$ with the conflation structure from definition \ref{definition:LocalizationConflation}, the localization functor $Q\colon \CC \to \CC[S^{-1}_\AA]$ satifies the universal property of a quotient $\CC \to \CC / \AA$ in the category of (small) conflation categories, meaning that $Q$ is exact and every exact functor $F\colon \CC \to \DD$ (with $\DD$ a conflation category) for which $F(\AA) = 0$ factors uniquely through $Q$.
\end{proposition}

\begin{proof}
It follows easily from definition \ref{definition:LocalizationConflation} that $Q$ is exact.  It remains to show that the localization functor $Q$ satisfies the universal property of the quotient $\CC/\AA$.  For this, consider an exact functor $F\colon \CC\rightarrow \mathcal{D}$ between deflation-exact categories such that $F(A)\cong 0$ for all $A\in \Ob(\AA)$.

By remark \ref{remark:SAInverts}, we know that $F(s)$ is an isomorphism for all $\AA^{-1}$-inflations and all $\AA^{-1}$-deflations $s$ in $S_{\AA}$ (and hence for any $s \in S_{\AA}$).  By the universal property of $Q$, there exists a unique functor $G\colon S_{\AA}^{-1}\CC\rightarrow \mathcal{D}$ such that $F=G\circ Q$.

It remains to show that $G$ is exact.  It follows from definition \ref{definition:LocalizationConflation} that any conflation in $S_{\AA}^{-1}\CC$ lifts to a conflation in $\CC$.  Since $F$ is exact, this lift is mapped to a conflation in $\DD$.  Since $F=G\circ Q$, we know that $G$ maps conflations to conflations, i.e.~$G$ is exact.
\end{proof}

\subsection{The \texorpdfstring{category $S_\AA^{-1}\CC$}{localized category} is a deflation-exact category}

We are now in a position to prove the main theorem, namely that the conflation category $S_\AA^{-1}\CC$ from definition \ref{definition:LocalizationConflation} is deflation-exact.

\begin{theorem}\label{theorem:Maintheorem}
Let $\CC$ be a deflation-exact category and let $\AA$ be a deflation-percolating subcategory.  The conflation category $S_{\AA}^{-1}\CC$ (see definition \ref{definition:LocalizationConflation}) is a deflation-exact category.  Moreover, if $\CC$ satisfies axiom \ref{R0*}, so does $S_{\AA}^{-1}\CC$.
\end{theorem}

\begin{proof}
It is easy to see that axiom \ref{R0} (respectively axiom \ref{R0*}) descends to $S_{\AA}^{-1}\CC$.  We now check that $S_{\AA}^{-1}\CC$ satisfies axioms \ref{R1} and \ref{R2}.
\begin{enumerate}
	\item[\ref{R1}] We consider two deflations $X\rightarrow Y$ and $Y\rightarrow Z$ in $S_{\AA}^{-1}\CC$.  By definition \ref{definition:LocalizationConflation}, this means that there are deflations $\overline{X}\rightarrow \overline{Y}$ and $\overline{\overline{Y}}\rightarrow \overline{Z}$ and a diagram
	\[\xymatrix@!@C=0.5em@R=0.5em{
		  & & Z &\ar[l]\ar[r]^{\sim} &\overline{Z}\\
		  & &  \ar[u]\ar[d]^{\rotatebox{90}{$\sim$}} & &\\		
		X &\ar[l]_{\sim}\ar[r] & Y &\ar[l]Y''\ar[r]^{\sim} & \overline{\overline{Y}}\ar@{->>}[uu]\\
		 \ar[u]^{\rotatebox{90}{$\sim$}}\ar[d] & &  \ar[u]^{\rotatebox{90}{$\sim$}}Y'\ar[d] & P\ar@{.>}[l]\ar@{.>}[u]^{\rotatebox{90}{$\sim$}} & \\
		\overline{X}\ar@{->>}[rr]& & \overline{Y}
	}\] which descends to a commutative diagram in $S_{\AA}^{-1}\CC$. Here, we chose the direction of the isomorphisms in $S_{\AA}^{-1}\CC$ in such a way to get the particular arrangement of arrows in $S_{\AA}$. The first step of the proof is to find a better representation in $\CC$ of this composition of deflations in $S_{\AA}^{-1}\CC$.
	
	Using proposition \ref{proposition:MinimalConditionsRMS}, we obtain the dotted arrows by axiom \ref{RMS2}. Note that the induced map $P\rightarrow Y'$ descends to an isomorphism in $S_{\AA}^{-1}\CC$. It follows that we can represent the outer edge by the solid part of the following commutative diagram:
	\[\xymatrix@!{
		& & \overline{Z}\\
		R\ar@{.>}[d]\ar@{.>>}[r]& P \ar[d]  \ar[r]^{\sim} & \overline{\overline{Y}}\ar@{->>}[u]\\
		\overline{X} \ar@{->>}[r] & \overline{Y} &\\
	}\]By axiom \ref{RMS1}, the composition $P\xrightarrow{\sim}Y''\xrightarrow{\sim} \overline{\overline{Y}}$ belongs to $S_{\AA}$. Axiom \ref{R2} yields the pullback square $RP\overline{Y}\overline{X}$ in $\CC$. As $P\rightarrow \overline{Y}$ descends to an isomorphism and $Q$ commutes with pullbacks (see remark \ref{remark:PreservationOfKernelsAndPullbacks}), the map $R\rightarrow \overline{X}$ descends to an isomorphism as well. It follows that the original composition of deflations in $S_{\AA}^{-1}\CC$ can be represented by the composition $R\twoheadrightarrow P\xrightarrow{\sim}\overline{\overline{Y}}\twoheadrightarrow \overline{Z}$.
	
	Write $K\inflation \overline{\overline{Y}}$ for the kernel of the deflation $\overline{\overline{Y}}\deflation \overline{Z}$. By lemma \ref{lemma:LiftingConflations}, we can lift the conflation over $P\stackrel{\sim}{\rightarrow} \overline{\overline{Y}}$ to $P'$ and obtain the following commutative diagram:
	\[\xymatrix{
		& Z'\ar[r]^{\sim} & \overline{Z}\\
		R'\ar@{.>>}[r]\ar@{.>}[d] & P'\ar@{->>}[u]\ar[rd]^{\rotatebox{135}{$\sim$}}\ar[d] & \\
		R\ar@{->>}[r] & P\ar[r]^{\sim} & \overline{\overline{Y}}\ar@{->>}[uu]
	}\] The lower left square is a pullback square obtained by axiom \ref{R2} in $\CC$. Note that the map $P'\to P$ descends to an isomorphism in $S_{\AA}^{-1}\CC$ and the map $R'\to R$ descends to an isomorphism as well (this follows from remark \ref{remark:PreservationOfKernelsAndPullbacks} and the fact that pullbacks of isomorphisms are isomorphisms). By axiom \ref{R1}, the composition $R'\deflation P' \deflation Z'$ is a deflation in $\CC$. Moreover, this composition descends to the composition $R\twoheadrightarrow P\xrightarrow{\sim}\overline{\overline{Y}}\twoheadrightarrow \overline{Z}$ up to isomorphism. This establishes that axiom \ref{R1} holds.
	
	\item[\ref{R2}]
	We now show that the pullback along a deflation exists and yields a deflation in $S_{\AA}^{-1}\CC$. For this, consider a co-span $X \twoheadrightarrow Y \leftarrow Z$ in $S_\AA^{-1} \CC$. The co-span can be represented by the following diagram in $\CC$:
		\[\xymatrix@d@!{
		&& Z\\
		&& Z'\ar[u]^{\rotatebox{0}{$\sim$}}\ar[d]\\
		X & X'\ar[l]^{\rotatebox{90}{$\sim$}}\ar[r] &Y & P\ar@{.>}[lu]_{\rotatebox{45}{$\sim$}} \ar@{.>}[ld]\\
		X''\ar[u]^{\rotatebox{0}{$\sim$}}\ar[d] & & Y'\ar[u]^{\rotatebox{0}{$\sim$}}\ar[d]\\
		\overline{X}\ar@{->>}[rr] & & \overline{Y}\\
	}\]
	 The dotted arrows are obtained by applying axiom \ref{RMS2} to $Y' \stackrel{\sim}{\twoheadrightarrow} Y \leftarrow Z'$.  In this way, we obtain a co-span $\overline{X} \twoheadrightarrow \overline{Y} \leftarrow P$ in $\CC$, which is isomorphic, in $S_{\AA}^{-1}\CC$, to the original co-span $X \twoheadrightarrow Y \leftarrow Z$.  As $Q$ preserves pullbacks, we are done by axiom \ref{R2} in $\CC$. \qedhere
\end{enumerate}
\end{proof}

\subsection{Exact quotients of exact categories}

The definition of the quotient in definition \ref{definition:RightExactLocalization} is taken in the category of conflation categories.  Even if one starts with an exact category $\CC$, the quotient category $\CC / \AA = \CC[S_\AA^{-1}]$ needs only be deflation-exact.  In this subsection, we show how one can obtain a quotient in the category of exact categories with respect to left or right percolating subcategories.

Our approach is based on the following proposition (see \cite[proposition I.7.5]{Rosenberg11}).

\begin{proposition}
	Let $\CC$ be a deflation-exact category. There exists an exact category $\overline{\CC}$ and a fully faithful exact functor $\gamma\colon\CC\rightarrow \overline{\CC}$ which is 2-universal in the following sense: for any exact category $\DD$, the functor $-\circ Q\colon \Hom_{\text{ex}}(\overline{\CC}, \DD) \to \Hom_{\text{ex}}(\CC, \DD)$ is an equivalence.  The category $\overline{\CC}$ is called the \emph{exact hull} of $\CC$.
\end{proposition}

For the benefit of the reader, we recall the construction given in \cite{Rosenberg11}.  As $\CC$ is deflation-exact, there is a Grothendieck pretopology where the covers are the deflations.  The exact hull $\overline{\CC}$ is obtained as the smallest fully exact subcategory of the category of sheaves on $\CC$ containing the representable sheaves.

The next corollary is an immediate application of the previous proposition and theorem \ref{theorem:Maintheorem}.

\begin{corollary}\label{corollary:Rosenberg}
	Let $\CC$ be an exact category and let $\AA$ be a deflation-percolating subcategory of $\CC$. The composition of the exact quotient functor $Q\colon \CC\rightarrow \CC/\AA$ and the embedding $\gamma\colon \CC/\AA\rightarrow \CC \dq \AA \coloneqq \overline{\CC/\AA}$ is an exact functor between exact categories satisfying the following 2-universal property: for any exact category $\DD$, the functor $- \circ \gamma\colon \Hom_{\text{ex}}(\CC \dq \AA, \DD) \to \Hom_{\text{ex}}(\CC, \DD)$ is a fully faithful functor whose essential image consists of those conflation-exact functors $F\colon \CC \to \DD$ for which $F(\AA) = 0$.
\end{corollary}

\subsection{On the role of axioms \ref{P1}-\ref{P4}}

In this section, we expand upon the inclusion of each of the axioms \ref{P1}-\ref{P4}.

As noted in \S\ref{subsection:WeakIsomorphisms}, the kernel of an exact functor is a Serre subcategory.  For $\AA$ to be the kernel of $Q$, axiom \ref{P1} needs to be satisfied.

Axiom \ref{P2} is used to bound the composition length of the reflected weak isomorphism in axiom \ref{RMS2} (see proposition \ref{proposition:MinimalConditionsRMS}). In particular, the proof of the lifting lemma and proposition \ref{proposition:InterpretationOfP4} rely on this bound. In fact, axiom \ref{P2} is equivalent to requiring that axiom \ref{RMS2} reflects the structure of weak isomorphisms. Indeed, let $f\colon X\to A$ be any morphism (with $A \in \AA$). Applying axiom \ref{RMS2}, we obtain the following commutative diagram:
\[\xymatrix{
	P\ar@{>->}[r]^{\sim}\ar[d] & X\ar@{->>}[r]\ar[d]^f & {X/P}\ar@{..>}[d]\\
	0\ar@{>->}[r]^{\sim} & A\ar@{->>}[r] & A
}\] Requiring that $P \to X$ is an inflation (since $0 \inflation X$ is), we see that axiom \ref{P2} holds.

In section \S\ref{Subsection:GliderExample} we construct an example of an exact category $\EE$ and a deflation-percolating subcategory $\AA$ such that $\EE/\AA$ is not inflation-exact. One can verify explicitly that $\AA$ satisfies the duals of axioms \ref{P1}, \ref{P3} and \ref{P4} but not the dual of \ref{P2}. As the quotient $\EE/\AA$ is not inflation-exact one sees that one cannot simply omit (the dual) of axiom \ref{P2}.

Example \ref{Example:P3Requirement} below gives a deflation-exact category $\CC$ and a full subcategory $\AA$ of $\CC$ satisfying axioms \ref{P1}, \ref{P2}, and \ref{P4}, but not satisfying axiom \ref{P3}.  We show explicitly that $\CC[S_\AA^{-1}]$ fails to satisfy theorem \ref{theorem:Maintheorem}.  This justifies the requirement of axiom \ref{P3}.

Similarly, example \ref{Example:P4Requirement} below gives a deflation-exact category $\CC$ and a full subcategory $\AA$ of $\CC$ satisfying axiom \ref{P1}, \ref{P2}, and \ref{P3}, but not satisfying axiom \ref{P4}. Again, we explicitly show that $\CC[S_\AA^{-1}]$ fails to satisfy theorem \ref{theorem:Maintheorem}. This justifies the requirement of axiom \ref{P4}.

\begin{example}\label{Example:P3Requirement}
Consider the quiver $A_4: 1\leftarrow 2\leftarrow 3\leftarrow 4$. Let $k$ be a field and write $\rep_k(A_4)$ for the category of finite-dimensional $k$-representations of $A_4$. We write $S_j$ for the simple representation associated to the vertex $j$, $P_j$ for its projective cover and $I_j$ for its injective envelope. The Auslander-Reiten quiver of $\rep_k(A_4)$ is give by:
	\[\xymatrix@!C=3pt@R=3pt{
		&&&P_4\ar[dr] &&&\\
&&P_3\ar[ur]\ar[dr] && I_2\ar[dr]&&\\
& P_2\ar[ur]\ar[dr] && \tau^{-1}P_2\ar[ur]\ar[dr] && I_3\ar[dr]&\\ 
S_1\ar[ur] && S_2\ar[ur] && S_3\ar[ur] && S_4
	}\] where $\tau$ is the Auslander-Reiten translate.
	
Let $\CC=\text{add}\{S_1,P_2,P_3,P_4,S_2,S_3,I_2\}$ be the full additive subcategory of $\rep_k(A_4)$ generated by the objects  $S_1,P_2,P_3,P_4,S_2,S_3$ and $I_2$ (thus, consisting of all objects which do not have direct summands isomorphic to $S_4, I_3,$ and $\tau^{-1} P_2$). We claim that $\CC$, with the conflations given by those in $\rep_k(A_4),$ is a deflation-exact category.  Let $\mathcal{D}$ be the full additive subcategory of $\rep_k(A_4)$ generated by $\CC$ and $\tau^{-1}P_2$.  As $\mathcal{D}$ is an extension-closed subcategory of an abelian category, it is exact.  It now follows from proposition \ref{proposition:DeflationClosed} that $\CC$ is deflation-exact.

Let $\AA$ be the full additive subcategory of $\CC$ generated by $S_2$. Clearly, $\AA$ is an abelian subcategory satisfying axioms \ref{P1} and \ref{P2}. In fact, axiom \ref{A2} (see definition \ref{Definition:AbelianPercolating} below) holds and thus axiom \ref{P4} holds by the proof of proposition \ref{proposition:AIsAbelian}.  Furthermore, one can verify that $\AA$ does not satisfy axiom \ref{P3} (indeed, \ref{P3} fails on the composition $P_2\inflation P_3\to I_2$).

Consider the commutative diagram
\[\xymatrix@!@C=0.5em@R=0.5em{
	P_2\ar@{>->}[rr]^i\ar[rd] && P_3\ar@{->>}[rr]\ar[dd] && S_3\\
	& S_2\ar[rd] &&&\\
	&&I_2&&	
}\] in $\CC$. If $\CC/\AA$ were a conflation category and the localization functor $Q\colon \CC\rightarrow \CC/\AA$ is exact, $Q(S_3)$ is the cokernel of  $Q(i)$. The above diagram yields an induced map $Q(S_3)\rightarrow Q(I_2)$ in $\CC/\AA$. On the other hand, one can verify explicitly that such a morphism cannot be obtained by localizing with respect to the weak isomorphisms. It follows that $P_2 \to P_3 \to S_3$ is not a kernel-cokernel pair in $S_\AA^{-1}\CC$.
\end{example}

\begin{example}\label{Example:P4Requirement}
	Consider the quiver $1\stackrel{\gamma}{\leftarrow}2\stackrel{\beta}{\leftarrow}3\stackrel{\alpha}{\leftarrow}4$ with relation $\gamma\beta\alpha=0$.
	The category $\UU$ of finite-dimensional representations of this quiver can be represented by its Auslander-Reiten quiver:
	\[\xymatrix@!C=3pt@R=3pt{
&&P_3\ar[dr] && I_2\ar[dr]&&\\
& P_2\ar[ur]\ar[dr] && \tau^{-1}P_2\ar[ur]\ar[dr] && I_3\ar[dr]&\\ 
S_1\ar[ur] && S_2\ar[ur] && S_3\ar[ur] && S_4
	}\] 
	We consider on $\UU$ the exact structure induced by all Auslander-Reiten sequences but $S_2\to \tau^{-1}P_2\to S_3$ (see \cite[corollary~3.10]{Enomoto2018} or \cite[theorem~5.7]{BrustleHassounLangfordRoy20} to see that this is a well-defined exact structure).  Let $\EE$ now be the full Karoubi subcategory of $\UU$ given by those objects who do not have $S_1$ as a direct summand.  As $\EE$ is extension-closed in $\UU$, we know that $\EE$ is an exact category.  As $\EE$ satisfies axiom \ref{R3}, it is straightforward to see that $P_2\to P_3\to S_3$ is not a conflation in $\EE$. 
	
	Let $\AA$ be the full additive subcategory of $\EE$ generated by $P_2$ and $P_3$. As $P_2\to P_3\to S_3$ is not a conflation, it is easy to verify axiom \ref{P1}. Axiom \ref{P2} is straightforward to verify and axiom \ref{P3} is automatic as $\EE$ is exact.
	
	We now claim that $\AA$ does not satisfy axiom \ref{P4}.  Note that the deflation $S_2\oplus P_3\deflation \tau^{-1}P_2$ is an $\AA^{-1}$-deflation (with kernel $P_2$), and that the split embedding $S_2\inflation S_2\oplus P_3$ is an $\AA^{-1}$-inflation (with cokernel $P_3$). As the composition $S_2\stackrel{\sim}{\inflation}S_2\oplus P_3\stackrel{\sim}{\deflation}\tau^{-1}P_2 \to S_3$ is zero, the map $\tau^{-1}P_2\to S_3$ descends to zero in $S_{\AA}^{-1}\EE$.  As the (irreducible) morphism $\tau^{-1}P_2\to S_3$ does not factor through $\AA$, it follows from proposition \ref{proposition:InterpretationOfP4} that $\AA$ does not satisfy axiom \ref{P4}.
	
	We now show that conflations in $\EE$ need not descend to kernel-cokernel pairs in $S_{\AA}^{-1}\EE$.  Consider the conflation $\tau^{-1}P_2\inflation I_2\deflation S_4$ in $\EE$ and note that the composition $\tau^{-1}P_2\inflation I_2\to I_3$ descends to zero in $S_{\AA}^{-1}\EE$ (indeed, the map $\tau^{-1}P_2\to I_3$ factors through $\tau^{-1}P_2\to S_3$ which descends to zero).
	
	If $\tau^{-1}P_2\inflation I_2\deflation S_4$ were a kernel-cokernel pair in $S_{\AA}^{-1}\EE$, there should be a map $f\colon S_3\to I_3$ in $S_{\AA}^{-1}\EE$ such that $P_3\to I_3$ factors through $f$. Note that a map $f\colon S_3\to I_3$ in $S_{\AA}^{-1}\EE$ can be represented by a diagram $S_3 \stackrel{\sim}{\leftarrow}X\to I_3$ in $\EE$.  Writing $X\stackrel{\sim}{\rightarrow}S_3$ as a composition of $\AA^{-1}$-inflations and $\AA^{-1}$-deflations, one finds that $X\cong S_3\oplus P_2^{\oplus n_2}\oplus P_3^{\oplus n_3}$ (for some $n_2,n_3$).  This shows that there are no non-zero maps $S_3\to I_3$ in $S_{\AA}^{-1}\EE$.  Hence, $S_{\AA}^{-1}\EE$ is not a conflation category.
\end{example}

%% file: AbelianPercolatingSubcategories.tex
\section{Admissibly percolating subcategories}\label{section:AbelianPercolating}

In the previous sections, we considered quotients of deflation-exact categories by deflation-percolating subcategories.  In theorem \ref{theorem:Maintheorem}, we showed that such a quotient can be obtained by localization with respect to a right multiplicative system. In practice, it might be difficult to show that a given morphism $f$ in $\CC$ is a weak isomorphism or that $Q(f)$ is invertible.

In this section, we address these difficulties by considering admissibly percolating subcategories of one-sided exact categories.  Admissibly (inflation- or deflation-) percolating subcategories satisfy stronger axioms than regular percolating subcategories.  We show that every weak isomorphism is an admissible morphism with kernel and cokernel in $\AA$ (see theorem \ref{theorem:WeakIsomorphismsEqualAAInverseIsomorphisms}). This last property motivates the terminology of an admissibly percolating subcategory.  

In addition, we show that the class of weak isomorphisms (with respect to an admissibly percolating subcategory) satisfies the $2$-out-of-$3$ property (proposition \ref{proposition:2OutOf3}) and is saturated (proposition \ref{proposition:Saturation}).

\subsection{Basic definitions and results}

We begin with the definition of an admissibly percolating subcategory.

\begin{definition}\label{Definition:AbelianPercolating}
Let $\CC$ be a conflation category and let $\AA$ be a non-empty full subcategory of $\CC$.  We call $\AA$ an \emph{admissibly deflation-percolating subcategory} or a \emph{strictly deflation-percolating subcategory} of $\CC$ if the following three properties are satisfied:
	\begin{enumerate}[label=\textbf{A\arabic*},start=1]
		\item\label{A1} $\AA$ is a Serre subcategory, meaning:
		\[\mbox{ If } A'\rightarrowtail A \twoheadrightarrow A'' \mbox{ is a conflation in $\CC$, then } A\in \Ob(\AA) \mbox{ if and only if } A',A''\in \Ob(\AA).\]
		\item\label{A2} For all morphisms $C\rightarrow A$ with $C \in \Ob(\CC)$ and $A\in \Ob(\AA)$, there exists a commutative diagram
		\[\xymatrix{
		A'\ar@{>->}[rd] & \\
		C \ar@{->>}[u]\ar[r]& A\\
				}\] with $A'\in \Ob(\AA)$, and where $C \twoheadrightarrow A'$ is a deflation and $A\rightarrowtail A'$ is an inflation.
		\item\label{A3}If $a\colon C\rightarrowtail D$ is an inflation and $b\colon C\twoheadrightarrow A$ is a deflation with $A\in \Ob(\AA)$, then the pushout of $a$ along $b$ exists and yields an inflation and a deflation, i.e.
		\[\xymatrix{
			C \ar@{>->}[r]^{a}\ar@{->>}[d]^b & D\ar@{.>>}[d]\\
			A\ar@{>.>}[r] & P
		}\]
	\end{enumerate}
\end{definition}

\begin{remark}\makeatletter
\hyper@anchor{\@currentHref}%
\makeatother\label{remark:AbelianPercolatingAboutAxioms}
\begin{enumerate}
	\item A conflation category with an admissibly deflation-percolating subcategory satisfies \ref{R0*}.  Indeed, it follows from \ref{A1} that $0 \in \AA$ and from \ref{A2} that any morphism $X \to 0$ is a deflation.
	\item It follows from proposition \ref{proposition:WhenPushout} that the pushout square in axiom \ref{A3} is a pullback square as well.
	\item Conditions \ref{A1} and \ref{A2} are also required by \cite[definition 4.0.35]{Cardenas98}. 
	\item Given a deflation-exact category $\CC$ and an admissibly deflation-percolating subcategory $\AA$, axiom \ref{A2} implies that $\AA$ is strongly right filtering. By proposition \ref{proposition:MinimalConditionsRMS}, pullbacks along weak isomorphisms exist and weak isomorphisms are stable under pullbacks.
	\item If $\CC$ is an exact category, axiom \ref{A3} is automatically satisfied. See for example the dual of \cite[proposition~2.15]{Buhler10}.
\end{enumerate}
\end{remark}

\begin{lemma}\label{Lemma:EpiToA}
Let $\AA$ be an admissibly deflation-percolating subcategory of a conflation category $\CC$.  Let $f\colon C \to A$ be a morphism in $\CC$ with $A \in \AA$.  If $f$ is a monomorphism (epimorphism), then $f$ is an inflation (deflation).  In particular, a morphism $X \to 0$ is a deflation.
\end{lemma}

\begin{proof}
	This is an immediate application of axioms \ref{A1} and \ref{A2}.
\end{proof}

\begin{proposition}\label{proposition:AIsAbelian}
	Let $\CC$ be a deflation-exact category and let $\AA$ be an admissibly deflation-percolating subcategory. Then $\AA$ is a (strongly) deflation-percolating subcategory of $\CC$ and $\AA$ is an abelian subcategory.
\end{proposition}

\begin{proof}
	Axiom \ref{P1} and \ref{P2} hold by axioms \ref{A1} and \ref{A2}. Axiom \ref{P3} follows from axiom \ref{A2} and \ref{A3} in a straightforward way. To show show axiom \ref{P4}, it suffices to see that the conditions of proposition \ref{proposition:P4Criterion} holds.  For this, we note that axiom \ref{A2} shows that every every morphism $f\colon A \to B$ in $\AA$ is admissible and hence admits a cokernel $B \deflation \coker(f).$  By axiom \ref{A1}, we know that $\coker(f) \in \AA.$  We conclude that axiom \ref{P4} holds, and hence $\AA$ is a (strongly) deflation-percolating subcategory of $\CC$.
	
	It remains to show that $\AA$ is abelian. Let $f\colon A\to B$ be a morphism in $\AA$. By axiom \ref{A2}, $f$ is admissible. By axiom \ref{A1}, $\ker(f)$ and $\coker(f)$ belong to $\AA$, it follows that $\AA$ is abelian.
\end{proof}

	\begin{remark}\makeatletter
\hyper@anchor{\@currentHref}%
\makeatother
		\begin{enumerate}
			\item Let $\CC$ be a deflation-exact category and let $\AA$ be a deflation-percolating subcategory.  We claim that, if all weak isomorphisms are admissible, then $\AA$ satisfies axiom \ref{A2}.

Let $A,B \in \Ob(\AA)$.  Any morphism $f\colon A \to B$ is the composition of $\begin{psmallmatrix} 1 & f \end{psmallmatrix}\colon A \stackrel{\sim}{\inflation} A \oplus B$ and $\begin{psmallmatrix} 0 \\ 1 \end{psmallmatrix}\colon A \oplus B \stackrel{\sim}{\deflation} B$, and thus a weak isomorphism.

Hence, if the weak isomorphisms are admissible, then any map between objects of $\AA$ is admissible.  This observation, combined with axiom \ref{P2}, yields axiom \ref{A2}.

			\item As in the proof of proposition \ref{proposition:AIsAbelian}, axioms \ref{A1} and \ref{A2} implies axiom \ref{P4}.  We will see in example \ref{Example:IsbellCategory} that axioms \ref{A1}, \ref{A2}, and \ref{P3} (and thus also \ref{P4}) alone are not sufficient for the set $S_\AA$ to be admissible.  This motivates strengthening axiom \ref{P3} to axiom \ref{A3} (see theorem \ref{theorem:WeakIsomorphismsEqualAAInverseIsomorphisms}).
		\end{enumerate}
	\end{remark}

\begin{proposition}\label{proposition:PercolatingOfPercolating}
Let $\CC$ be a deflation-exact category and let $\AA \subseteq \CC$ be an admissibly percolating subcategory.  If $\BB \subseteq \AA$ is a Serre subcategory, then $\BB \subseteq \CC$ is an admissibly percolating subcategory.
\end{proposition}

\begin{proof}
It is easy to see that $\BB \subseteq \CC$ satisfies axioms \ref{A1} and \ref{A3}.  Using that $\AA$ is an abelian category (proposition \ref{proposition:AIsAbelian}) and that the embedding $\AA \subseteq \CC$ satisfies axiom \ref{A2}, it is easy to verify that $\BB \subseteq \CC$ satisfies axiom \ref{A2}.
\end{proof}

\subsection{Homological consequences of axiom \ref{A3}}\label{subsection:HomologicalConsequencesOfAxiomA3}

Throughout this subsection, let $\CC$ be a deflation-exact category and $\AA$ a non-empty full subcategory of $\CC$ satisfying axiom \ref{A3}.  We show that the existence of such a subcategory yields a weak version of axiom \ref{R3} (see proposition \ref{proposition:WeakR3} below), and we show that a weak version of the $3\times 3$-lemma holds (see proposition \ref{proposition:WeakNineLemma} below).  If $\CC$ is a strongly deflation-exact category, these two properties are automatically satisfied (see \cite{BazzoniCrivei13}).

\begin{proposition}\label{proposition:WeakR3}
	Let $g\colon Y\rightarrow Z$ be a map such that $g$ has a kernel belonging to $\AA$ and such that there exists a deflation $f\colon X\twoheadrightarrow Y$ such that $gf$ is also a deflation. Then $g$ is a deflation.
\end{proposition}

\begin{proof}
	Proposition \ref{proposition:MitchellPullback} yields the following commutative diagram
	\[\xymatrix{
	\ker(gf)\ar@{.>}[d]^{f'} \ar@{>->}[r]^{k'} & X\ar@{->>}[r]\ar@{->>}[d]^f & Z\ar@{=}[d]\\
	\ker(g)\ar[r]^k & Y\ar[r]^g & Z
	}\]
	where the left-hand square is a pullback.  Axiom \ref{R2} implies that $f'$ is a deflation.
	Proposition \ref{proposition:WhenPullback} yields the existence of the following commutative diagram
	\[\xymatrix{
	K\ar@{>->}[d]^{l}\ar@{=}[r] & K\ar@{>->}[d]^{l'}\\
	\ker(gf)\ar@{->>}[d]^{f'} \ar@{>->}[r]^{k'} & X\ar@{->>}[d]^f \\
	\ker(g)\ar[r]^k & Y
	}\] 
	where the columns are conflations.  By proposition \ref{proposition:MitchellPullbackPushout}\eqref{enumerate:MitchellPushout}, we know that the lower square is a pushout.  Since $\ker(g)\in \Ob(\AA)$, axiom \ref{A3} implies that $k\colon \ker(g)\rightarrowtail Y$ is an inflation.  Proposition \ref{proposition:WhenPushout} implies that $g$ is the cokernel of the inflation $k$, and hence $g$ is a deflation.
\end{proof}

The next proposition is a version of the $3\times 3$-lemma as given in \cite[proposition 5.11]{BazzoniCrivei13}.

\begin{proposition}\label{proposition:WeakNineLemma}
	Consider a commutative diagram
	\[\xymatrix{
		X\ar@{>->}[r]\ar@{->>}[d]  & Y\ar@{->>}[r]\ar@{->>}[d]  & Z\ar@{->>}[d] \\
		X'\ar@{>->}[r] & Y'\ar@{->>}[r]  & Z'
	}\]
	where the rows are conflations and the vertical arrows are deflations.  If $X'\in \Ob(\AA)$, then the above diagram can be completed to a commutative diagram
	\[\xymatrix{
	  X''\ar@{>->}[r]\ar@{>->}[d] & Y''\ar@{>->}[d]\ar@{->>}[r] & Z''\ar@{>->}[d]\\
		X\ar@{>->}[r]\ar@{->>}[d]  & Y\ar@{->>}[r]\ar@{->>}[d]  & Z\ar@{->>}[d] \\
		X'\ar@{>->}[r] & Y'\ar@{->>}[r]  & Z'
	}\]
	where the rows and the columns are conflations.  Moreover, the upper left square is a pullback and the lower right square is a pushout.
\end{proposition}

\begin{proof}
	By proposition \ref{proposition:FactorizationOfConflationMorphism}, the diagram can be extended to a commutative diagram 
	\[\xymatrix{
		X\ar@{>->}[r]^i\ar@{->>}[d]_{f} \ar@{}[rd] | {A} & Y\ar@{->>}[r]^p\ar[d] \ar@{}[rd] | {B} & Z\ar@{=}[d]\\
		X'\ar@{>->}[r]\ar@{=}[d] \ar@{}[rd] | {C} & P\ar@{->>}[r]\ar[d] \ar@{}[rd] | {D} & Z\ar@{->>}[d]^h\\
		X'\ar@{>->}[r]_{i'} & Y'\ar@{->>}[r]_{p'} & Z'
	}\]  such that the square {D} is a pullback and square {A} is both a pullback and a pushout. By axioms \ref{A3} and \ref{R2} the maps $Y\rightarrow P$ and $P\rightarrow Y'$ are deflations.  Applying proposition \ref{proposition:WhenPullback} yields the following commutative diagrams:
	\begin{center}
	\begin{tabular}{l p{.2\textwidth} r}
	$\xymatrix{
		X''\ar@{=}[r]\ar@{>->}[d]\ar@{}[rd] | {E} & X'' \ar@{>->}[d]&\\
		X\ar@{->>}[d]\ar@{}[rd] | {A}\ar@{>->}[r] & Y\ar@{->>}[d]\ar@{->>}[r]\ar@{}[rd] | {B} & Z\ar@{=}[d]\\
		X'\ar@{>->}[r] & P\ar@{->>}[r] & Z
	}$
	& &
	$\xymatrix{
		& Z''\ar@{=}[r]\ar@{>->}[d]\ar@{}[rd] | {F} & Z''\ar@{>->}[d]\\
		X'\ar@{=}[d]\ar@{>->}[r]\ar@{}[rd] | {C} & P\ar@{->>}[d]\ar@{->>}[r]\ar@{}[rd] | {D} & Z\ar@{->>}[d]\\
		X'\ar@{>->}[r] & Y'\ar@{->>}[r] & Z'
	}$
	\end{tabular}
	\end{center}
	where the rows and colums are conflations.  Starting from the conflations $X'' \rightarrowtail Y \twoheadrightarrow P$ and $Z'' \rightarrowtail P \twoheadrightarrow Y'$, we construct the commutative diagram
	\[\xymatrix{
	X''\ar@{=}[d]\ar@{.>}[r]\ar@{}[rd] | {G} & Y''\ar@{>->}[d]\ar@{.>}[r]\ar@{}[rd] | {H} & Z''\ar@{>->}[d] \\
	X'' \ar@{>->}[r]& Y \ar@{->>}[d]\ar@{->>}[r]\ar@{}[rd] | {I}& P\ar@{->>}[d]\\
	 & Y'\ar@{=}[r] & Y'	
	}\]  
where the rows and columns are conflations. Here, the dotted morphism $\xymatrix@1{Y'' \ar@{..>}[r] & Z''}$ is chosen such that the square {H} is a pullback (see proposition \ref{proposition:MitchellPullback}, the chosen morphism is automatically a deflation by \ref{R2}).  The square {G} is given by proposition \ref{proposition:WhenPullback}.

Putting the commutative squares together, we obtain the commutative diagram:
	\[\xymatrix{
	X''\ar@{=}[d]\ar@{>.>}[r]\ar@{}[rd] | {G} & Y''\ar@{>->}[d]\ar@{.>>}[r]\ar@{}[rd] | {H} & Z''\ar@{>->}[d] \\
	X'' \ar@{>->}[r] \ar@{>->}[d] \ar@{}[rd] | {E} & Y \ar@{=}[d]\ar@{->>}[r]\ar@{}[rd] | {B} & P\ar@{->>}[d]\\
	X \ar@{>->}[r] & Y\ar@{->>}[r] & Z	
	}\]
	where the right-most column composes to the morphism $Z'' \rightarrowtail Z$ by the square {F}.  As $X'' \rightarrowtail X$, $Y'' \rightarrowtail Y$, and $Z'' \rightarrowtail Z$ have been constructed as kernels of $X \twoheadrightarrow X'$, $Y \twoheadrightarrow Y'$, and $Z \twoheadrightarrow Z'$, respectively, we obtain the $3 \times 3$-diagram in the statement of the proposition.

Finally, consider the $3 \times 3$-diagram in the statement of the proposition.  It follows from proposition \ref{proposition:MitchellPullbackPushout} that the upper left square is a pullback and the lower right square is a pushout.
\end{proof}

\subsection{Weak isomorphisms are admissible}\label{subsection:AbelianPercolatingAdmissible}

Throughout this section, $\CC$ denotes a deflation-exact category and $\AA$ denotes an admissibly deflation-percolating subcategory. Consider the set $\widehat{S_{\AA}} \coloneqq S_{\AA}\cap \Adm(\CC)$ of admissible weak isomorphisms. The aim of this section is to show that $\widehat{S_{\AA}}=S_{\AA}$.

\begin{remark}\makeatletter
\hyper@anchor{\@currentHref}%
\makeatother\label{remark:DefinitionAdmissibleWeakIso}
	\begin{enumerate}
		\item A morphism $f\colon X\rightarrow Y$ in $\CC$ belongs to $\widehat{S_{\AA}}$ if and only if $f$ is admissible and $\ker(f),\coker(f)\in \AA$. 
		\item For any admissible morphism $f$, one automatically has that $\coim(f)\cong \im(f)$ and $f$ factors as deflation-inflation through $\im(f)$.
	  \item Admissible weak isomorphisms are called $\AA^{-1}$-isomorphisms in \cite[definition~4.0.36]{Cardenas98}.
		\item Any morphism $\alpha\colon A\rightarrow B$ in an admissibly deflation-percolating subcategory $\AA\subseteq \CC$ belongs to $\widehat{S_{\AA}}$. Indeed, this is an immediate corollary of proposition \ref{proposition:AIsAbelian}.
	\end{enumerate}
\end{remark}

We show two additional homological properties which are consequences of axiom \ref{A2}. The first is a strengthening of the lifting lemma (lemma \ref{lemma:LiftingConflations})

\begin{corollary}\label{corollary:WeakNineLemma}
Let $X \rightarrowtail Y \twoheadrightarrow Z$ be a conflation and $f\colon Y \twoheadrightarrow B$ be a deflation.  If $B \in \Ob(\AA)$, then there is a commutative diagram
	\[\xymatrix{
	  X''\ar@{>->}[r]\ar@{>->}[d]^{\rotatebox{90}{$\sim$}} & Y''\ar@{>->}[d]^{\rotatebox{90}{$\sim$}}\ar@{->>}[r] & Z''\ar@{>->}[d]^{\rotatebox{90}{$\sim$}}\\
		X\ar@{>->}[r]\ar@{->>}[d]  & Y\ar@{->>}[r]\ar@{->>}[d]^{f}  & Z\ar@{->>}[d] \\
		A\ar@{>->}[r] & B\ar@{->>}[r]  & C
	}\]
	where the rows and the columns are conflations, and where the bottom row lies in $\AA$.  Moreover, the upper left square is a pullback and the lower right square is a pushout.
\end{corollary}

\begin{proof}
It follows from \ref{A2} that the composition $X \to Y \to B$ factors as $X \twoheadrightarrow A \rightarrowtail B$ with $A \in \Ob(\AA)$.  This gives the following commutative diagram:
	\[\xymatrix{
		X\ar@{>->}[r]\ar@{->>}[d]  & Y\ar@{->>}[r]\ar@{->>}[d]^{f}  & Z\ar@{.>>}[d] \\
		A\ar@{>->}[r] & B\ar@{->>}[r]  & C
	}\]
	with exact rows (the bottom rows lies in $\AA$ by \ref{A1}).  The dotted arrow is induced by the universal property of the cokernel $Y\twoheadrightarrow Z$.  One easily verifies that the dotted arrow is an epimorphism and, thus, by lemma \ref{Lemma:EpiToA}, a deflation.  The statement now follows from proposition \ref{proposition:WeakNineLemma}.
\end{proof}

\begin{proposition}\label{proposition:PushoutOfAInverseAlongMapToA}
	Let $f\colon X\rightarrow Y$ belong to $\widehat{S_{\AA}}$ and let $g\colon X\rightarrow A$ be any morphism.  If $A \in \Ob(\AA)$, then the pushout of $f$ along $g$ exists and the induced map belongs to $\widehat{S_{\AA}}$.
\end{proposition}

\begin{proof}
	By definition, $f$ is an admissible map. By axiom \ref{A2}, $g$ is admissible as well. Since $A' \in \AA$, corollary \ref{corollary:WeakNineLemma} yields a commutative diagram
	\[\xymatrix{
		Y && \\
		X'\ar@{>->}[u]\ar@{.>>}[r] & P& \\
		X\ar@{->>}[r] \ar@{->>}[u]& A' \ar@{.>>}[u]\ar@{>->}[r]&A\\
		\ker(f) \ar@{.>>}[r]\ar@{>->}[u]& A''\ar@{>.>}[u]
	}\]
	with $P \in \AA$ and such that the square $X'PA'X$ is a pushout. Applying axiom \ref{A3} twice yields a commutative diagram
		\[\xymatrix{
		Y\ar@{.>>}[r] & Q& \\
		X'\ar@{>->}[u]\ar@{->>}[r]\ar@{>.>}[u] & P\ar@{>.>}[u]\ar@{>.>}[r]&R \\
		X\ar@{->>}[r] \ar@{->>}[u]& A' \ar@{->>}[u]\ar@{>->}[r]&A\ar@{.>>}[u]
	}\]

	Since $Q,R,P\in \Ob(\AA)$ and $\AA$ is an abelian category by proposition \ref{proposition:AIsAbelian}, we can complete $Q,R,P$ to a pushout square. Hence we obtain the commutative diagram
	\[\xymatrix{
		Y\ar@{->>}[r] & Q\ar@{>.>}[r]& S\\
		X'\ar@{>->}[u]\ar@{->>}[r]\ar@{>->}[u] & P\ar@{>->}[u]\ar@{>->}[r]&R\ar@{>.>}[u] \\
		X\ar@{->>}[r] \ar@{->>}[u]& A' \ar@{->>}[u]\ar@{>->}[r]&A\ar@{->>}[u]\\
	}\] where all squares are pushout squares.  It follows from the pushout lemma that the square $YSAX$ is a pushout square as well. Since the map $A\rightarrow S$ belongs to $\AA$, remark \ref{remark:DefinitionAdmissibleWeakIso} yields that it is an admissible weak isomorphism.
\end{proof}

\begin{lemma}\label{lemma:CompositionOfAAInflations}
	Let $\CC$ be a deflation-exact category and let $\AA$ be a full subcategory of $\CC$. If $\AA$ satisfies axioms \ref{A1} and \ref{A3}, the composition of $\AA^{-1}$-inflations is again an $\AA^{-1}$-inflation.
\end{lemma}

\begin{proof}
	Let $\xymatrix{U\ar@{>->}[r]^{\sim}_a & V}$ and $\xymatrix{V\ar@{>->}[r]^{\sim}_b & W}$ be two $\AA^{-1}$-conflations. Consider the following commutative diagram:
	\[\xymatrix{
		U\ar@{>->}[r]^{\sim}_a\ar@{=}[d] & V\ar@{->>}[r]\ar@{>->}[d]^{\rotatebox{90}{$\sim$}}_b &\coker(a)\ar@{>->}[d]\\
		U\ar[r]_{ba} & W\ar@{->>}[r]\ar@{->>}[d] &P\ar@{->>}[d]\\
		  & \coker(b)\ar@{=}[r] & \coker(b)
	}\] The top right square is a pushout square which exists by axiom \ref{A3}. Note that the top right square is also a pullback, it follows that the compostion $ba$ is the kernel of the deflation $W\deflation P$. Since $\coker(a),\coker(b)\in \AA$, axiom \ref{A1} imlies that $P\in \AA$ as well. It follows that $ba$ is an $\AA^{-1}$-inflation.
\end{proof}

The next lemma is crucial in showing that $S_{\AA}=\widehat{S_{\AA}}$, i.e.~that the weak isomorphisms are automatically admissible.

\begin{lemma}\label{lemma:SwitchAInflationsAndADeflations}
	Let $\CC$ be a deflation-exact category and let $\AA\subseteq \CC$ be an admissibly deflation-percolating subcategory. Let $a\colon U \stackrel{\sim}{\inflation} V$ and $b\colon V\stackrel{\sim}{\deflation} W$ be an $\AA^{-1}$-inflation and $\AA^{-1}$-deflation, respectively. The composition $b\circ a$ is an admissible weak isomorphism.
\end{lemma}

\begin{proof}
	Using corollary \ref{corollary:WeakNineLemma}, we find the commutative diagram
				\[\xymatrix{
			 \ker(c_ak_b)\ar@{>->}[d]^i\ar[r]^{k_b'}& U\ar@{->>}[r]^{c_b'} \ar@{>->}[d]^a & \ker(c_a')\ar@{>->}[d]^{k_a'}\\
			\ker(b)\ar@{>->}[r]^{k_b}\ar@{->>}[d]^{p} & V\ar@{->>}[r]^{b}\ar@{->>}[d]^{c_a} & W\ar@{->>}[d]^{c_a'}\\
			\im(c_ak_b)\ar@{>->}[r]^{i'}& \coker(a)\ar@{->>}[r]^{p'} &\coker(c_a k_b)
		}\]
		such that the rows and columns are conflations.  By axiom \ref{A1}, the left column and lower row belong to $\AA$.  The upper-right square shows that $ba=k_a'c_b'$ and thus $ba$ is admissible. Clearly, $\ker(ba)=\ker(c_ak_b)$ and $\coker(ba)=\coker(c_a k_b)$, both belonging to $\AA$.
\end{proof}

\begin{theorem}\label{theorem:WeakIsomorphismsEqualAAInverseIsomorphisms}
Let $\CC$ be a deflation-exact category. If $\AA\subseteq \CC$ an admissibly deflation-percolating subcategory, then $S_{\AA}=\widehat{S_{\AA}}$, in particular all weak isomorphisms are admissible. Moreover, $S_{\AA}$ is a right multiplicative system such that the square in axiom \ref{RMS2} can be chosen as a pullback square, in particular, one can take pullbacks along weak isomorphisms.
\end{theorem}

\begin{proof}
	The proof is a straightforward application of proposition \ref{proposition:MinimalConditionsRMS} and  lemmas \ref{lemma:CompositionOfAAInflations} and \ref{lemma:SwitchAInflationsAndADeflations}. 
\end{proof}

\subsection{The 2-out-of-3 property}\label{subsection:2oo3}

Throughout this section $\CC$ is a deflation-exact category and $\AA$ is an admissibly deflation-percolating subcategory. We now show that the right multiplicative system $S_{\AA}$ of admissible weak isomorphisms satisfies the 2-out-of-3 property.  We first establish some preliminary results.

\begin{lemma}\label{Lemma:Bazzoni5.10}\label{Lemma:DeflationInflationFiveLemma}
Consider a commutative diagram
\[\xymatrix{
		X\ar@{>->}[r]\ar[d]^f & Y\ar@{->>}[r]\ar[d]^g & Z\ar@{=}[d]\\
		X'\ar@{>->}[r] & Y'\ar@{->>}[r] & Z
	}\]
with exact rows.
	\begin{enumerate}
	\item If $f$ is an $\AA^{-1}$-inflation, then $g$ is an $\AA^{-1}$-inflation.
	\item If $f$ is an $\AA^{-1}$-deflation, then $g$ is an $\AA^{-1}$-deflation.
	\end{enumerate}
\end{lemma}

\begin{proof}
	\begin{enumerate}
		\item From proposition \ref{proposition:PushoutIfCokernel}, we obtain the following commutative diagram with exact rows and columns (where the left square is a pushout):
		\[\xymatrix{
		X\ar@{>->}[r]\ar@{>->}[d]^f_{\rotatebox{90}{$\sim$}} & Y\ar@{->>}[r]\ar[d]^g & Z\ar@{=}[d]\\
		X'\ar@{>->}[r]\ar@{->>}[d] & Y'\ar@{->>}[r]\ar[d]^{c_g} & Z\\
		\coker(f)\ar@{=}[r] & \coker(g) & 
		}\]  As $c_g\colon Y'\rightarrow \coker(g)$ is an epimorphism and $\coker(f)\in \Ob(\AA)$, lemma \ref{Lemma:EpiToA} yields that $c_g$ is a deflation. Denote the kernel of $c_g$ by $K\rightarrowtail Y'$.  By corollary \ref{corollary:WeakNineLemma} we obtain a commutative diagram:
		\[\xymatrix{
		X\ar@{>->}[r]\ar@{=}[d] & Y\ar@{->>}[r]\ar@{.>}[d] & Z\ar@{=}[d]\\
		X\ar@{>->}[r]\ar@{>->}[d] & K\ar@{->>}[r]\ar@{>->}[d] & Z\ar@{=}[d]\\
		X'\ar@{>->}[r]\ar@{->>}[d]& Y'\ar@{->>}[r]\ar@{->>}[d]& Z\ar@{->>}[d]\\
		\coker(f)\ar@{=}[r] & \coker(g)\ar@{->>}[r] &0
		}\] The dotted arrow is obtained by factoring $g$ through the kernel of its cokernel. By the short five lemma (\cite[lemma 5.3]{BazzoniCrivei13}), the induced map $Y\rightarrow K$ is an isomorphism. It follows that $g$ is an inflation.  Since $\coker(g)\in\Ob(\AA)$, we find that $g \in S_{\AA}$.
		
		\item By proposition \ref{proposition:PushoutIfCokernel}, we know that the left square is a pushout and a pullback, and we obtain the following commutative diagram (where the columns are conflations)
		\[\xymatrix{
		\ker(f)\ar@{>->}[d]^k\ar@{=}[r] & \ker(g)\ar[d]\\
		X\ar@{>->}[r]^i\ar@{->>}[d]^f & Y\ar[d]^g \\
		X'\ar@{>->}[r]^{i'} & Y' 
	}\] with $\ker(f)\in \Ob(\AA)$ and such that $ik$ is the kernel of $g$.  By proposition \ref{proposition:PullbackPushout}, the map $\begin{psmallmatrix}
	i' & g
	\end{psmallmatrix}\colon X'\oplus Y\twoheadrightarrow Y'$ is a deflation. By \cite[lemma 5.1]{BazzoniCrivei13}, the map $\begin{psmallmatrix}
	f & 0\\0&1
	\end{psmallmatrix}\colon X\oplus Y\rightarrow X'\oplus Y$ is a deflation as well. Axiom \ref{R1} yields that 
	\[\begin{pmatrix}
	i' & g
	\end{pmatrix}\begin{pmatrix}
	f & 0\\0&1
	\end{pmatrix}=\begin{pmatrix}
	i'f & g
	\end{pmatrix}=\begin{pmatrix}
	gi & g
	\end{pmatrix}=g\begin{pmatrix}
	i & 1
	\end{pmatrix}\] is a deflation. As $\begin{psmallmatrix} i& 1 \end{psmallmatrix}\colon X \oplus Y \to Y$ is a retraction (and hence a deflation), proposition \ref{proposition:WeakR3} shows that $g$ is a deflation. \qedhere
	\end{enumerate}
\end{proof}

\begin{lemma}\label{lemma:CompositionOfAAInflationAndInflation}
	Let $\CC$ be a deflation-exact category and let $\AA$ be a non-empty full subcategory satisfying axiom \ref{A3}. Consider two composable inflations $f\colon X\inflation Y$ and $g\colon Y\inflation Z$. If $f$ is an $\AA^{-1}$-inflation, $g\circ f$ is an inflation.
\end{lemma}

\begin{proof}
	Let $q\colon Y\deflation \coker(f)$ be the cokernel of $f$. By axiom \ref{A3}, we obtain the following commutative diagram such that the lower-right square is bicartesian:
	\[\xymatrix{
		X\ar@{=}[r]\ar@{>->}[d]^{\rotatebox{90}{$\sim$}}_f & X\ar[d]_{gf}\\
		Y\ar@{>->}[r]\ar@{->>}[d] & Z\ar@{.>>}[d]\\
		\coker(f)\ar@{>.>}[r] & P
	}\] By proposition \ref{proposition:MitchellPullback}, $g\circ f$ is the kernel of the deflation $Z\to P$ and thus $g\circ f$ is an inflation.
\end{proof}

\begin{proposition}\makeatletter
\hyper@anchor{\@currentHref}%
\makeatother\label{proposition:DeflationInflationFiveLemma}
\begin{enumerate}

	\item Consider the following commutative diagram in $\CC$
\[\xymatrix{
X\ar@{>->}[r]\ar@{>->}[d]^f & Y\ar@{->>}[r]\ar[d]^g & Z\ar@{>->}[d]^h \\
X'\ar@{>->}[r] & Y'\ar@{->>}[r] & Z'}\]
where the rows are conflations.
	  \begin{enumerate}
		  \item If $f$ is an $\AA^{-1}$-inflation, then $g$ is an inflation. 
			\item If additionally $h$ is an $\AA^{-1}$-inflation, then $g$ is an $\AA^{-1}$-inflation.
		\end{enumerate}

	\item Consider the following commutative diagram in $\CC$
\[\xymatrix{
X\ar@{>->}[r]\ar@{->>}[d]^f & Y\ar@{->>}[r]\ar[d]^g & Z\ar@{->>}[d]^h \\
X'\ar@{>->}[r] & Y'\ar@{->>}[r] & Z'}\]
where the rows are conflations.
	  \begin{enumerate}
		  \item If $f$ is an $\AA^{-1}$-deflation, then $g$ is a deflation. 
			\item If additionally $h$ is an $\AA^{-1}$-deflation, then $g$ is an $\AA^{-1}$-deflation.
		\end{enumerate}
\end{enumerate}
\end{proposition}

\begin{proof}

 Following proposition \ref{proposition:FactorizationOfConflationMorphism}, we consider the following factorization of both diagrams in the statement of the proposition:
\[\xymatrix{
X\ar@{>->}[r] \ar@{}[rd] | {A} \ar[d]_f & Y\ar@{->>}[r]\ar[d]^{g_1} \ar@{}[rd] | {B} & Z\ar@{=}[d] \\
X'\ar@{=}[d]\ar@{>->}[r] \ar@{}[rd] | {C} & P \ar@{->>}[r]\ar[d]^{g_2}  \ar@{}[rd] | {D} & Z \ar[d]^h \\
X'\ar@{>->}[r] & Y'\ar@{->>}[r] & Z'}\]

\begin{enumerate}
\item If $f$ is an $\AA^{-1}$-inflation, then lemma \ref{Lemma:DeflationInflationFiveLemma} yields that $g_1$ is an $\AA^{-1}$-inflation. As $D$ is a pullback square and $h$ an inflation, proposition \ref{proposition:PullbackOfInflation} yields that $g_2$ is an inflation. It now follows from lemma \ref{lemma:CompositionOfAAInflationAndInflation} that $g=g_2\circ g_1$ is an inflation. Moreover, if $h$ is an $\AA^{-1}$-inflation, its cokernel belongs to $\AA$. As $D$ is also a pushout square, $\coker(g_2)=\coker(h)$. Hence $g_2$ is an $\AA^{-1}$-inflation. By lemma \ref{lemma:CompositionOfAAInflations} we conclude that $g=g_2\circ g_1$ is an $\AA^{-1}$-inflation.

\item If $f$ is an $\AA^{-1}$-deflation,  lemma \ref{Lemma:DeflationInflationFiveLemma} yields that $g_1$ is an $\AA^{-1}$-deflation. As $D$ is a pullback square and $h$ a deflation, axiom \ref{R2} yields that $g_2$ is a deflation. By axiom $\ref{R1}$, $g=g_2\circ g_1$ is a deflation.
 Moreover, if $h$ is an $\AA^{-1}$-deflation, then $g_2$ is a deflation (by axiom \ref{R3}) with $\ker g = \ker h \in \AA$.  It then follows from theorem \ref{theorem:WeakIsomorphismsEqualAAInverseIsomorphisms} yields that $g = g_2 \circ g_1$ is an $\AA^{-1}$-deflation.\qedhere
\end{enumerate}
\end{proof}

\begin{proposition}\label{proposition:2OutOf3}
	The two-out-of-three property holds, i.e.~if $f,g$ are composable morphisms, then if two of the three maps $f,g$, and $gf$ belong to $S_{\AA}$, so does the third. 
\end{proposition}

\begin{proof}
	As we already showed that $S_{\AA}$ is a right multiplicative set, we know that $f,g\in S_{\AA}$ implies that $gf \in S_\AA$. We will first show that $g,gf\in S_{\AA}$ implies that $f \in S_\AA$. In step 1 we prove this statement assuming that $g$ is an inflation. In step 2 we prove the statement assuming that $g$ is a deflation. These two steps suffice as we know that $g$ has a deflation-inflation factorization where both parts are morphisms in $S_{\AA}$.\\
	
	\begin{enumerate}
	\item[Step 1: ] We now show that if $g,gf\in S_{\AA}$ and $g$ is an inflation, then $f\in S_{\AA}$.  As $g$ is a monomorphism, we find that $\ker(f)=\ker(gf)\in \Ob(\AA)$. It follows that $\coim(gf)\cong\coim(f)$. Hence we obtain the diagram: 
	\[\xymatrix{
		& Y\ar@{>->}[rd]^g & \\
		X\ar@{->>}[r]\ar[ru]^f\ar@{->>}[rd] & \coim(f)\ar[d]^{\cong} \ar@{.>}[u]& Z\\
		& \coim(gf)\ar@{>->}[ru] & 
	}\] Clearly the left-hand side of the diagram is commutative. Since $X\twoheadrightarrow \coim(f)$ is epic, the right side is commutative as well.  Since $gf\in S_{\AA}$ is admissible, we have that $\coim(gf)\cong \im(gf)$.  Since the right-hand side commutes, the map $\im(gf)\rightarrowtail Z \twoheadrightarrow \coker(g)$ is zero. By corollary \ref{corollary:WeakNineLemma}, we obtain a commutative diagram:
		\[\xymatrix{
	\coim(f)\ar@{>->}[d]\ar[r]^{\cong} & \im(gf)\ar@{->>}[r] \ar@{>->}[d]& 0\ar@{>->}[d]\\
	Y\ar@{>->}[r]^g\ar@{->>}[d] & Z\ar@{->>}[r]\ar@{->>}[d] & \coker(g)\ar@{=}[d]\\
	 K \ar@{>->}[r]& \coker(gf)\ar@{->>}[r] & \coker(g)
	}\] Using that $g$ is monic, one readily verifies that $\coim(f)\rightarrowtail Y$ coincides with the dotted arrow from the previous diagram. It follows that $f\colon X \deflation X/\ker(f) = \coim(f) \inflation Y$ is admissible and that$\coker(f)=K\in \AA$. We conclude that $f\in S_{\AA}$.

	\item[Step 2: ] We now show that if $g,gf\in S_{\AA}$ and $g$ is a deflation, then $f\in S_{\AA}$.  Consider the commutative diagram: 
	\[\xymatrix{
	\ker(gf)\ar@{.>}[d]^{\phi} \ar@{>->}[r]& X\ar[d]^f\ar@{->>}[r] & \im(gf)\ar@{>->}[d]\\
	\ker(g)\ar@{>->}[r] & Y\ar@{->>}[r]^g &Z
	}\] As $\phi$ is a map in $\AA$ and $\AA$ is abelian (see proposition \ref{proposition:AIsAbelian}), we know that $\phi$ factors as a deflation-inflation through its image $\im(\phi)$. Moreover, proposition \ref{proposition:MitchellPullback} implies that the left square is a pullback. As pullbacks preserve kernels, $\ker(f)=\ker(\phi)\in \Ob(\AA)$.
	
	Axiom \ref{A3} yields the following commutative diagram:
	\[\xymatrix{
		\ker(gf)\ar@{>->}[r]\ar@{->>}[d]^{\rotatebox{90}{$\sim$}} & X\ar@{->>}[r]\ar@{->>}[d] & \im(gf)\ar@{=}[d]\\
		\im(\phi)\ar@{>->}[r]\ar@{>->}[d]^{\rotatebox{90}{$\sim$}}& P\ar@{->>}[r]\ar@{.>}[d] & \im(gf)\ar@{>->}[d]^{\rotatebox{90}{$\sim$}}\\
		\ker(g)\ar@{>->}[r]  & Y\ar@{->>}[r] & Z
	}\] Indeed, the upper-left square is a pushout square constructed by axiom \ref{A3} and the lower-right square commutes as $X\twoheadrightarrow P$ is epic. By proposition \ref{proposition:DeflationInflationFiveLemma} we conclude that $P\rightarrow Y$ is an $\AA^{-1}$-inflation. It follows that $f\in S_{\AA}$. This concludes step $2$.
	\end{enumerate}
	
	Next, we will show that if $f,gf\in S_{\AA}$ then $g\in S_{\AA}$. Since $f$ has a deflation-inflation factorization, it suffices to prove the statement separately assuming that $f$ is a deflation and assuming $f$ is an inflation. This will be done in step 1' and step 2'.
	
	\begin{enumerate}
	\item[Step 1': ] Assume that $f$ is a deflation. Then $\coker(gf)=\coker(g)\in \Ob(\AA)$. Hence we get the diagram 
	\[\xymatrix{
		 & Y\ar[rd]^g\ar@{.>}[d]^h & \\
		X\ar@{->>}[ru]^f\ar@{->>}[rd] & \im(g)\ar@{>->}[r] & Z\\
		& \im(gf)\ar@{>->}[ru]\ar[u]^{\cong} & 
	}\] Using that $\im(g)\rightarrowtail Z$ is monic, we see that this diagram is commutative.
	
	As the composition $\ker(f)\rightarrow X\rightarrow \im(gf)$ is zero, one easily obtains the following commutative diagram:
	\[\xymatrix{
		\ker(f)\ar@{>.>}[r]\ar@{=}[d] & \ker(gf)\ar@{->>}[r]\ar@{>->}[d] & C\ar@{.>}[d]\\
		\ker(f)\ar@{>->}[r]\ar@{->>}[d] & X\ar@{->>}[r]^f\ar@{->>}[d] & Y\ar@{.>}[d]^{h'}\\
		0\ar@{>->}[r] & \im(gf)\ar@{=}[r] & \im(gf)
	}\] Here the induced map $\ker(f)\rightarrow \ker(gf)$ is a monomorphism between objects in$\AA$.  As $\AA$ is abelian (see propositio \ref{proposition:AIsAbelian}),  it has a cokernel $C  \in \AA$. By proposition \ref{proposition:MitchellPullbackPushout} the upper-right square is a pushout, so that by axiom \ref{A3} the map $C\rightarrow Y$ is an inflation and by proposition \ref{proposition:PushoutIfCokernel} the map $Y\rightarrow \im(gf)$ is a deflation. As $f$ is epic one sees that $h=h'$. It follows that $g\colon Y \deflation \im(gf) \inflation Z$ has a deflation-inflation factorization. Since $\ker(h)=C\in \Ob(\AA)$ and $\ker(h)=\ker(g)$ we conclude that $g$ is an admissible weak isomorphism.
	
	\item[Step 2': ] Let $f$ be an inflation. We obtain a commutative diagram:
	\[\xymatrix{
		X\ar@{>->}[r]^f\ar@{->>}[d] & Y\ar[d]^g\ar@{->>}[r] & \coker(f)\ar@{.>}[d]^{\phi}\\
		\im(gf)\ar@{>->}[r] & Z\ar@{->>}[r]\ar@{.>>}[d] & \coker(gf)\ar@{->>}[d]\\
		& \coker(\phi)\ar@{=}[r] & \coker(\phi)
	}\] Here we used that the induced map $\phi$ lies in $\AA$ and hence has a cokernel $\coker(\phi) \in \AA$ (as $\AA$ is abelian, see proposition \ref{proposition:AIsAbelian}), the map $Z\twoheadrightarrow \coker(\phi)$ is a deflation by axiom \ref{R1}. The upper-right square is a pushout by proposition \ref{proposition:MitchellPullbackPushout}. It follows that $Z\twoheadrightarrow \coker(\phi)$ is the cokernel of $g$.
	
	As $\phi$ is a morphism in the abelian category $\AA$, it admits a deflation-inflation factorization $\coker(f)\twoheadrightarrow \im(\phi)\rightarrowtail \coker(gf)$.  Taking the pullback of $\im(\phi)\rightarrowtail \coker(gf)$ along $Z\twoheadrightarrow \coker(gf)$ we obtain the commutative diagram:
	\[\xymatrix{
		X\ar@{>->}[r]^f\ar@{->>}[d]_{\rotatebox{90}{$\sim$}} & Y\ar@{->>}[r]\ar@{.>}[d] & \coker(f)\ar@{->>}[d]_{\rotatebox{90}{$\sim$}}\\
		\im(gf)\ar@{>->}[r]\ar@{=}[d] & P\ar@{->>}[r]\ar[d] & \im(\phi)\ar@{>->}[d]\\
		\im(gf)\ar@{>->}[r] & Z\ar@{->>}[r] & \coker(gf)		
	}\]
	By proposition \ref{proposition:PullbackOfInflation}, $P\rightarrowtail Z$ is an inflation. By proposition \ref{proposition:DeflationInflationFiveLemma}, the map $Y\rightarrow P$ is a deflation whose kernel belongs to $\AA$. It follows that $g$ is an admissible weak isomorphism. \qedhere
	\end{enumerate}
\end{proof} 

\subsection{The saturation property}

We now show that the right multiplicative set $S_\AA$ of weak isomorphisms satisfies the saturation property.  Our proof is based on the two-out-of-three-property (proposition \ref{proposition:2OutOf3}).

\begin{lemma}\label{lemma:IdempotentsInKernel}
Let $\CC$ be a deflation-exact category and let $\AA$ be an idempotent complete deflation-percolating subcategory. Let $e\colon X \to X$ be any idempotent in $\CC$.  If $Q(e) = 0$, then $X \cong \ker(e) \oplus A$ for some $A \in \AA$.
\end{lemma}

\begin{proof}
This follows from the first part of the proof of \cite[lemma 1.17.6]{Schlichting04} and proposition \ref{proposition:InterpretationOfP4}.
\end{proof}

\begin{proposition}\label{proposition:Saturation}
	Let $\CC$ be a deflation-exact category and let $\AA$ be an admissibly deflation-percolating subcategory. The set $S_{\AA}$ of weak isomorphisms is saturated.
\end{proposition}

\begin{proof}
	Let $f\colon Y\rightarrow Z$ be map that descends to an isomorphism in $S_{\AA}^{-1}\CC$. By definition, there exists a map $(g,s)\in \Mor(S_{\AA}^{-1}\CC)$ such $(f,1) \circ (g,s) \sim (1,1)$, i.e.~there exists a commutative diagram:
	\[\xymatrix{
		& Z'\ar[ld]_{s}\ar[rd]^{fg}&\\
		Z\ar@{=}[rd] & M\ar[l]_{\sim}\ar[r]^{\sim}\ar[d]_{\rotatebox{90}{$\sim$}}\ar[u]^h & Z\\
		& Z\ar@{=}[ru] &
	}\] Since $fgh\in S_{\AA}$, we can take the pullback of $f$ along $fgh$ (see theorem \ref{theorem:WeakIsomorphismsEqualAAInverseIsomorphisms}). We obtain the following commutative diagram
		\[\xymatrix{
		& Y\ar[rr]^f&& Z\\
		&&&\\
		&P\ar[uu]^{\beta}_{\rotatebox{90}{$\sim$}}\ar[rr]^{\alpha}&&M\ar[uu]^{\rotatebox{90}{$\sim$}}_{f(gh)}\ar[lluu]_{gh}\\
		M\ar@/^2pc/[uuur]^{gh}\ar@/_2pc/[rrru]_{1_M}\ar@{.>}[ru]^{\gamma}&&&
	}\] where the square is a pullback and $\gamma\colon M \to P$ is induced by the pullback property.  Clearly, $\alpha$ is a retraction and $\gamma$ is a section. Since $Q(f)$ is invertible in $S_{\AA}^{-1} \CC$ and the localization functor commutes with pullbacks (see proposition \ref{proposition:BasicPropertiesOfLocalization}), we know that $Q(\alpha)$ is invertible in $S_{\AA}^{-1} \CC$ and that $Q(\alpha)^{-1} = Q(\gamma)$. It follows that the kernel of $\alpha$ is zero in $S_{\AA}^{-1}\CC$. This implies that $Q(1_P - \gamma \circ \alpha) = 0$.  Lemma \ref{lemma:IdempotentsInKernel} shows that $P \cong K \oplus A$ with $A \in \AA$.  We infer that $\ker(\gamma \circ \alpha) = \ker(\alpha) = A.$  As $\CC$ satisfies \ref{R0*} (see remark \ref{remark:AbelianPercolatingAboutAxioms}) and $\alpha$ is a retraction, we may infer that $\alpha$ is a deflation in $\CC$.  It follows that $\alpha\in S_{\AA}$.  From the 2-out-of-3 property, it follows that $gh \in S_\AA$ and subsequently that $f \in S_{\AA}$.
\end{proof}

%% file: StabilityOfQuillensObscureAxiom.tex
\section{Quillen's obscure axiom under localizations}\label{section:StabilityQuillen'sObscureAxiom}

Let $\AA$ be a deflation-percolating subcategory of a strongly deflation-exact category $\CC$.  It is shown in theorem \ref{theorem:Maintheorem} that the localization $\CC / \AA = S_{\AA}^{-1}\CC$ is again a deflation-exact category.  As will be illustrated in example \ref{example:MissingR3} below, even when $\AA$ is an admissibly deflation-percolating subcategory, the category $\CC / \AA$ does not need to inherit axiom \ref{R3} from $\CC$.

In this section, we show that if $\CC$ is \emph{weakly idempotent complete} (thus, if every retraction $X \to Y$ admits a kernel) and satisfies axiom \ref{R3}, then the same holds for $\CC / \AA$ whenever $\AA$ is a strongly deflation-percolating subcategory (see definition \ref{definition:AdditionalPercolatingDefinitions}).

We start by reformulating the conditions on $\CC$.

\begin{proposition}\label{proposition:wicR3}
Let $\CC$ be a deflation-exact category.  The following are equivalent:
\begin{enumerate}
	\item $\CC$ is weakly idempotent complete and satisfies axiom \ref{R3},
	\item If $f\colon X \to Y$ and $g\colon Y \to Z$ are morphisms in $\CC$ such that $gf$ is a deflation, then $g$ is a deflation.
\end{enumerate}
\end{proposition}

\begin{proof}
This follows easily from \cite[proposition 6.4]{BazzoniCrivei13}.
\end{proof}

\begin{theorem}\label{theorem:StabilityOfStrongCondition}
	Let $\CC$ be a weakly idempotent complete strongly deflation-exact category (thus, satisfying the equivalent conditions in proposition \ref{proposition:wicR3}) and let $\AA$ be a strongly deflation-percolating subcategory. The localization $\CC[S_\AA^{-1}]$ is also a weakly idempotent complete strongly deflation-exact category.
\end{theorem}

\begin{proof}
	By theorem \ref{theorem:Maintheorem}, we know that $S_{\AA}^{-1}\CC$ is a deflation-exact category. We now show that if $f\colon X\rightarrow Y$ and $g\colon Y\rightarrow Z$ are two maps in $S_{\AA}^{-1}\CC$ such that $gf$ is a deflation, then $g$ is a deflation.  Consider a commutative diagram
	\[\xymatrix{
	X \ar@{=}[r]\ar[d]^f &X\ar@{->>}[d]\\
	Y\ar[r]^g & Z
	}\] in $S_{\AA}^{-1}\CC$. The diagram lifts to a diagram 
		\[\xymatrix@!@C=0.3em@R=0.3em{
	X\ar@{=}[rr] & & X&\ar[l]X'''\ar[r]^{\sim}&\overline{X}\ar@{->>}[dd]\\
	X' \ar[u]^{\rotatebox{90}{$\sim$}}\ar[d]& & X''\ar[u]^{\rotatebox{90}{$\sim$}}\ar[d]&&\\
	Y & Y'\ar[l]_{\sim}\ar[r] & Z&\ar[l]_{\sim}Z'\ar[r]&\overline{Z}
	}\] in $\CC$. We claim that we can choose a lift 
	\[\xymatrix{
	\widetilde{X} \ar[r]^{\sim}\ar[d]^{f'} &\overline{X}\ar@{->>}[d]^h\\
	\widetilde{Y}\ar[r]^{g'} & \overline{Z}
	}\] in $\CC$ such that this diagram descends to $gf$ under the localization functor $Q$.	Indeed, applying axiom \ref{RMS2} four times we obtain the diagram
	\[\xymatrix@!@C=0.2em@R=0.2em{
	X\ar@{=}[rr] & & X & X'''\ar[l]\ar[rr]^{\sim} & & \overline{X}\ar@{->>}[dddd]\\
	& & & \ar@{.>}[u]^{\rotatebox{90}{$\sim$}}\ar@{.>}[ld] & &\\
	X'\ar[uu]^{\rotatebox{90}{$\sim$}}\ar[dd]& \ar@{.>}[l]X_1\ar@{.>}[r]^{\sim} & X''\ar[uu]^{\rotatebox{90}{$\sim$}}\ar[dd] & &\ar@{.>}[lu]_{\rotatebox{135}{$\sim$}}\ar@{.>}[ld]X_2 &\\
	&& &\ar@{.>}[lu]_{\rotatebox{135}{$\sim$}}\ar@{.>}[d] & &\\
	Y & Y'\ar[l]_{\sim}\ar[r] & Z & Z'\ar[l]_{\sim}\ar[rr] && \overline{Z}
	}\] which descends to a commutative diagram in $S_{\AA}^{-1}\CC$. Rearranging the diagram and applying axiom \ref{RMS2} twice we obtain the diagram: 
	\[\xymatrix@!@C=0.2em@R=0.5em{
		X'' &&&&X_2\ar[d]_{\rotatebox{90}{$\sim$}}\ar[llll]\\
		X_1\ar[u]^{\rotatebox{90}{$\sim$}}\ar[d] & & X_3 \ar@{.>}[ll]\ar@{.>}[rru]^{\rotatebox{26.5}{$\sim$}}&&\overline{X}\ar@{->>}[d]\\
		Y & Y'\ar[l]_{\sim}\ar[r] & Z & Z'\ar[l]_{\sim}\ar[r] & \overline{Z}\\
		&& \widetilde{Y}\ar@{.>}[lu]_{\rotatebox{135}{$\sim$}}\ar@{.>}[ru] &&
	}\]	Applying axiom \ref{RMS2} to $X_3\rightarrow Y$ along $\widetilde{Y}\xrightarrow{\sim} Y$ we obtain the desired lift:
	\[\xymatrix{
	\widetilde{X} \ar[r]^{\sim}\ar[d]^{f'} &\overline{X}\ar@{->>}[d]^{h}\\
	\widetilde{Y}\ar[r]^{g'} & \overline{Z}
	}\] This shows the claim.	 Hence, it suffices to show that, given a commutative diagram in $\CC$: 
	\[\xymatrix{
		X\ar[r]^{\sim}_s\ar[d]^f & X'\ar@{->>}[d]^h\\
		Y\ar[r]^g&Z
	}\] the morphism $g$ descends to a deflation in $S_{\AA}^{-1}\CC$.  
	
	Write $k\colon K\inflation X'$ for the kernel of $h$. By the lifting lemma \ref{lemma:LiftingConflations}, there exists a commutative diagram
	\[\xymatrix{
	\overline{K}\ar[rr]^{\sim}\ar@{>->}[d]_{\overline{k}} && K\ar@{>->}[d]^k\\
	\overline{X}\ar@{->>}[d]_{\overline{h}}\ar[r]_t\ar@/^2pc/[rr]^{\sim} & X\ar[r]^{\sim}_s\ar[rd]_{gf} & X'\ar@{->>}[d]^{h}\\
	\overline{Z}\ar[rr]^{\sim}_u && Z	
	}\] Since $\AA$ is strongly right filtering in $\CC$, proposition \ref{proposition:MinimalConditionsRMS} yields that pullbacks along weak isomorphisms exist. Hence we obtain the following commutative diagram:
	\[\xymatrix{
		\overline{X}\ar@/^2pc/[rrd]^t\ar@{.>}[rd]^{t'}\ar@/_2pc/@{->>}[rddd]_{\overline{h}} &&\\
		& P_2\ar[r]^{\sim}_{u''}\ar[d]_{f'} & X\ar[d]^f\\
		& P_1\ar[r]^{\sim}_{u'}\ar[d]_{g'} & Y\ar[d]^g\\
		&\overline{Z}\ar[r]^{\sim}_{u} & Z
	}\] Here the two squares are pullback squares and thus the rectangle $P_2X\overline{Z}Z$ is a pullback by the pullback lemma. The map $t'$ is induced by the pullback property. Note that $g'f't'=\overline{h}$ is a deflation in $\CC$. By axiom \ref{R3} and weak idempotent completeness, $g'$ is a deflation in $\CC$. Clearly $Q(g')=Q(g)$ in $S_{\AA}^{-1}\CC$ and thus $g$ descends to a deflation as required.
	
	Note that the above property automatically implies that $S_{\AA}^{-1}\CC$ is weakly idempotent complete (see proposition \ref{proposition:wicR3}), this completes the proof.
\end{proof}

The following example illustrates that the condition of $\CC$ being weakly idempotent complete cannot be dropped.

\begin{example}\label{example:MissingR3}
Let $Q$ be the quiver $\xymatrix@1{c \ar[r]^\alpha & b\ar[r]^\beta & a}$ with relation $\beta\alpha = 0$.  Let $k$ be a field and write $\rep_k Q$ for the category of finite-dimensional $k$-representations of $Q$.  We write $S_a, S_b,$ and $S_c$ for the simple representations associated to the corresponding vertices, and $P_a, P_b,$ and $P_c$ for their projective covers (note that $P_a \cong S_a$).  The Auslander-Reiten of $\rep_k Q$ quiver is given by
\[\xymatrix{
&P_b \ar[rd]^{g} && P_c \ar[rd]^{l} \\
S_a \ar[ru]^{f} && S_b \ar[ru]^{h} && S_c}\]
Let $\CC$ be the full subcategory of $\rep_k Q$ given by all objects not isomorphic to $S_a^{\oplus n} \oplus S_b$ (for any $n \geq 0$).  As $\CC$ is an extension-closed full subcategory of an abelian category, it is endowed with a natural exact structure.

Let $\AA = \add\{S_a\}$ be the additive closure of $S_a$ in $\CC$.  We have that $\AA \subseteq \CC$ is an admissibly deflation-percolating subcategory.  Following theorem \ref{theorem:Maintheorem}, we can consider the quotient $Q\colon \CC \to \CC/\AA$.  Note, as an additive category, $\CC / \AA$ is generated by $Q(P_b), Q(P_c),$ and $Q(S_c)$.  More specifically, the category $\CC / \AA$ is equivalent (as an additive category) to $\rep_k A_2$.

We claim that the quotient $\CC / \AA$ does not satisfy axiom \ref{R3}.  Indeed, consider the sequences
\[\mbox{$\xymatrix@1{Q(P_b) \ar[r]^{Q(hg)} & Q(P_c) \ar[r]^{Q(l)} & Q(S_c)}$ and $\xymatrix@1{Q(P_b) \ar@{=}[r] & Q(P_b) \ar[r]& 0}$}\]
in $\CC/\AA$.  One can verify that the first sequence is not a conflation (in particular, the map $Q(l)\colon Q(P_c) \to Q(S_c)$ is not a deflation).  However, the direct sum of these two sequence is a conflation:
\[\xymatrix{Q(P_b \oplus P_b) \ar@{>->}[rr]^{Q\begin{psmallmatrix} 1 & 0 \\ 0 & hg \end{psmallmatrix}} && Q(P_b \oplus P_c) \ar@{->>}[rr]^{Q \begin{psmallmatrix} 0 & l \end{psmallmatrix}} && Q(S_c)}\]
(this uses that $Q(P_b \oplus P_b) \cong (P_b \oplus S_b)$).  It follows from \cite[proposition 5.9]{BazzoniCrivei13} that $\CC / \AA$ does not satisfy axiom \ref{R3}.

Note that $\CC$ satisfies axiom \ref{R3} (as $\CC$ is exact) and that $\AA$ is strongly (even admissibly) deflation-percolating in $\CC$, but that $\CC$ is not weakly idempotent complete (as the retraction $S_b \oplus P_b \to P_b$ has no kernel).
\end{example}

%% file: ExamplesAndApplications.tex
\section{Examples and applications}\label{section:Examples}

In this section, we give examples of the localizations studied in this paper.  We start with a comparison with \cite{Cardenas98, Schlichting04}.  Next, we show that a deflation-exact category with \ref{R0*} is a category with fibrations (thus, in particular, a coWaldhausen category).  In this way, we have a natural definition of the $K$-theory of a deflation-exact category.  We show that the quotient behaves as expected on the level of the Grothendieck groups.

We then proceed by considering some more specific examples of percolating subcategories.  In \S\ref{subsection:Torsion}, we consider torsion theories in exact categories and give sufficient conditions for the torsion-free part to be a deflation-percolating subcategory or a right special filtering subcategory.  In \S\ref{subsection:QuasiAbelian}, we consider the case of a one-sided quasi-abelian category (also called an almost abelian category) and show that the axioms of a percolating subcategory simplify in this setting.

Finally, we consider two more explicit examples.  The first example (\S\ref{subsection:LCA}) concerns the category $R-\LC$ locally compact modules over a discrete ring $R$.  It was shown in \cite{Braunling20} that the subcategory $R-\LC_D$ of discrete modules is, in general, neither a left nor a right special filtering subcategory.  We show that $R-\LC_D$ is a right percolating subcategory so that we can consider the quotient category $R-\LC / R-\LC_D$.  Furthermore, we show that the subcategory of totally disconnected locally abelian groups is a deflation-percolating subcategory.  In the second example (\S\ref{Subsection:GliderExample}), we give an example coming from the theory of glider representations.  Here, we show explicitly that the quotient category is not inflation-exact (and hence not exact).

\subsection{Comparison to localization theories of exact categories}

Localizations of exact categories have been considered with an eye on $K$-theoretic applications in \cite{Cardenas98, Schlichting04}. We now compare these notions with the notions introduced in this paper.  The following figure provides an overview.
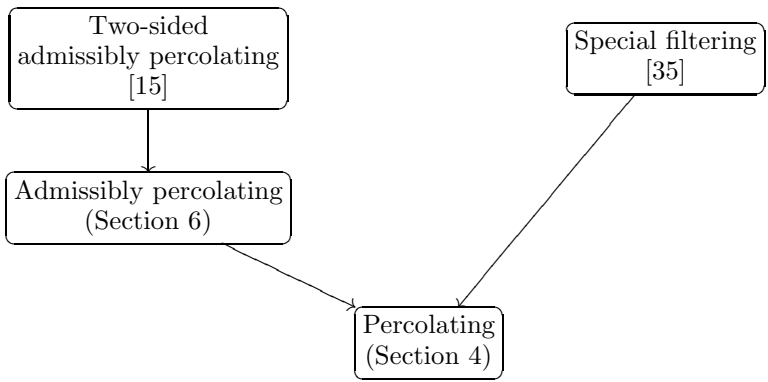
\begin{figure}[H]
		\[\xymatrix{
		*+[F-:<3pt>]{\txt{Two-sided \\ admissibly percolating \\ \cite{Cardenas98}}} \ar[d] &&*+[F-:<3pt>]{\txt{Special filtering \\ \cite{Schlichting04}}} \ar[ddl]\\
		*+[F-:<3pt>]{\txt{Admissibly percolating \\ (Section \ref{section:AbelianPercolating})}} \ar[dr] \\
		&*+[F-:<3pt>]{\txt{Percolating \\ (Section \ref{Section:PercolatingSubcategories})}}
		}\]
	\caption{Different types of subcategories of an exact category}
	\label{figure:ExactOverview}
\end{figure}

\subsubsection{Cardenas' localization theory}

The localization theory of exact categories developed by Cardenas in \cite{Cardenas98} is recovered completely by the framework of localizations with respect to percolating subcategories.  We recover the following result from \cite{Cardenas98}: 

\begin{theorem}
Let $\CC$ be an exact category and let $\AA$ be a full subcategory satisfying axioms \ref{A1}, \ref{A2} and the dual of \ref{A2}. There exists an exact category $\CC/\AA$ and an exact functor $Q\colon \CC\rightarrow \CC/\AA$ satisfying the following universal property: for any exact category $\CC$ and exact functor $F\colon \CC\rightarrow \CC$ such that $F(A)\cong 0$ for all $A\in \AA$, there exists a unique exact functor $G\colon \CC/\AA\rightarrow \CC$ such that $F=G\circ Q$.

The set $S_\AA$ of weak equivalences is a (left and right) multiplicative set, and the quotient $\CC / \AA$ is equivalent to the category localization $\CC[S_\AA^{-1}].$
\end{theorem}

\begin{proof}
	Since $\CC$ is exact, the subcategory $\AA$ automatically satisfies axiom \ref{A3}. Hence, $\AA$ is both an admissibly inflation- and deflation-percolating subcategory. By theorem \ref{theorem:WeakIsomorphismsEqualAAInverseIsomorphisms} and its dual, the set $\widehat{S_{\AA}}$ is a multiplicative system. By theorem \ref{theorem:Maintheorem} and its dual, as well as proposition \ref{proposition:QuotientInConflationCategories}, the category $\CC[S_{\AA}^{-1}]$ is both inflation-exact and deflation-exact and the canonical localization functor $Q\colon \CC\rightarrow \CC[S_{\AA}^{-1}]$ is exact.  Moreover, $\CC[S_{\AA}^{-1}]\cong \CC/\AA$ and $Q$ satisfies the desired universal property. 
\end{proof}

\subsubsection{Schlichting's localization theory}\label{subsubsection}

We recall the notion of an s-filtering subcategory of an exact category introduced by Schlichting (see \cite[definition~1.5]{Schlichting04}). We use the reformulation given in \cite[definition~2.12 and proposition~A.2]{BraunlingGroechenigWolfson16}.

\begin{definition}\label{definition:SpecialFiltering}
	Let $\CC$ be an exact category and let $\AA$ be a full subcategory. The subcategory $\AA$ is called \emph{right special} if for every inflation $A\rightarrowtail X$ with $A\in \AA$ there exists a commutative diagram 
	\[\xymatrix{
		A\ar@{>->}[r]\ar@{=}[d] & X\ar@{->>}[r]\ar[d] & Y\ar[d]\\
		A\ar@{>->}[r] & B\ar@{->>}[r] & C
	}\] such that the rows are conflations in $\CC$ and the lower row belongs to $\AA$. Dually, $\AA$ is called \emph{left special} if $\AA^{op}$ is right special in $\CC^{op}$.\\
	The subcategory $\AA$ is called \emph{right s-filtering} if it is both right filtering, i.e.~satisfies axioms \ref{P1} and \ref{P2}, and right special in $\CC$.
\end{definition}

The following results about Schlichting's localization theory can be found in \cite[propositions~1.16 and 2.6]{Schlichting04}):

\begin{theorem}\label{theorem:Schlichting'sMainTheorem}
	Let $\CC$ be an exact category and let $\AA$ be a right s-filtering subcategory. The localization functor $Q\colon \CC\rightarrow S_{\AA}^{-1}\CC$ endows $S_{\AA}^{-1}\CC$ with the structure of an exact category. The functor $Q$ is universal among exact functors from $\CC$ to exact categories that vanish on $\AA$, i.e.~$\CC/\AA \cong S_{\AA}^{-1}\CC$.\\
	Moreover, if $\AA$ is idempotent complete, the sequence 
	\[\Db(\AA)\rightarrow \Db(\CC)\rightarrow \Db(\CC/\AA)\]
	is a Verdier localization sequence.
\end{theorem}

The localization theory developed in \cite{Schlichting04} is compatible with the localization theory with respect to percolating subcategories in the following sense.

\begin{proposition}\label{proposition:RecoveredSchlichtingsFramework}
Let $\CC$ be an exact category and $\AA\subseteq \CC$ a full subcategory. If $\AA$ is a right s-filtering subcategory, then $\AA$ is a deflation-percolating subcategory.
\end{proposition}

\begin{proof}
	As $\AA$ is a right filtering subcategory of $\CC$, axioms \ref{P1} and \ref{P2} are satisfied. Since $\CC$ is exact, axiom \ref{A3} is satisfied (and thus axiom \ref{P3} is satisfied as well). It remains to verify axiom \ref{P4}.
	
	To that end, consider a commutative diagram:
	\[\xymatrix{
		A\ar@{>->}[rd]^i\ar@/^2pc/[rrrd]^0 &&&\\
		 & X\ar[rr]^f\ar@{->>}[rd]_p^{\rotatebox{135}{$\sim$}}\ar[dd]_{\rho}&& Y\\
		&& Q\ar@{.>}[ru]^h & \\
		& B\ar@/_2pc/[rruu]_g&&
	}\] We need to show that $h$ factors through $\AA$.
	
	By axiom \ref{P2}, we may, without loss of generality, assume that $\rho$ is a deflation with kernel $\iota\colon K\stackrel{\sim}{\inflation} X$. Since $p\circ \iota\in S_{\AA}$, \cite[lemma~1.17.(3)]{Schlichting04} implies that there exists an $\AA^{-1}$-inflation $t\colon U\stackrel{\sim}{\inflation}K$ such that $p\iota t$ is an $\AA^{-1}$-inflation as well. Lemma \ref{lemma:CompositionOfAAInflations} shows that the composition $\iota \circ t$ is an $\AA^{-1}$-inflation. As $\CC$ is an exact category, the $3\times 3$-lemma (see \cite[corollary 3.6]{Buhler10}) yields the following commutative diagram:
	\[\xymatrix{
	0\ar@{>->}[r]\ar@{>->}[d]^{\rotatebox{90}{$\sim$}} & U\ar@{=}[r]\ar@{>->}[d]^{\rotatebox{90}{$\sim$}}_{\iota t} & U\ar@{>->}[d]^{\rotatebox{90}{$\sim$}} & \\
	A\ar@{>->}[r]_i\ar@{=}[d] & X\ar@{->>}[r]_p^{\sim}\ar@{->>}[d]_{q'} & Q\ar@{->>}[d]_q\ar[r]_h & Y\ar@{=}[dd]\\
	A\ar@{>->}[r]^{i'} & B'\ar@{->>}[r]^{p'}\ar@{.>}[d]_l& C\ar@{.>}[rd]^{\exists ! u} &\\
	& B\ar[rr]^g && Y\\
	}\] Here, $q'$ is the cokernel of $\iota t$ and $l$ is the induced map such that $lq'=\rho$. Note that $(gl)i'=glq'i=g\rho i=fi=0$, hence $gl$ factors through $p'$ via a unique induced map $u\colon C\to Y$ satisfying $up'=gl$.  Moreover, as $p$ is epic and $hp=f=g\rho=q'lg=up'q'=(uq)p$, we conclude that $h=uq$. This shows that $h$ factors through $C\in \Ob(\AA)$, as required.
\end{proof}

\begin{remark}
Let $\CC$ be an exact category and let $\AA$ be a right s-filtering subcategory.  By theorem \ref{theorem:Schlichting'sMainTheorem}, the category $S_{\AA}^{-1}\CC$ is an exact category.  Using the above proposition and theorem \ref{theorem:Maintheorem}, we may only conclude that $S_{\AA}^{-1}\CC$ is a deflation-exact category. In section \ref{Subsection:GliderExample} we will show that the localization of an exact category with respect to a one-sided percolating subcategory need not be exact.
\end{remark}

\begin{remark}
In subsequent work \cite{HenrardvanRoosmalen19b}, we show that, given a deflation-exact category $\CC$ and a deflation-percolating subcategory $\AA$, one still obtains a Verdier localization sequence
		\[\DAb(\CC)\rightarrow \Db(\CC)\rightarrow \Db(\CC/\AA).\]
	Here $\DAb(\CC)$ denotes the thick triangulated subcategory of $\Db(\CC)$ generated by $\AA$ (the derived category of a deflation-exact category is defined in \cite{BazzoniCrivei13}).
\end{remark}

\subsection{Waldhausen categories and the Grothendieck group}

Given a deflation-exact category $\CC$ and an admissibly deflation-percolating subcategory $\AA$, one can encode the localization $\CC/\AA$ into a \emph{coWaldhausen category}. In this way, one can study the K-theory of $\CC/\AA$. In particular one obtains an immediate description of the Grothendieck group of $\CC/\AA$ (see proposition \ref{proposition:WaldhausenGrothendieck}). We refer the reader to \cite{Weibel13} for more details.

\begin{definition}
	Let $\CC$ be a category and let $\cofib(\CC)$ be a set of morphisms in $\CC$ called \emph{cofibrations} (indicated by arrows $\rightarrowtail$). The pair $(\CC,\cofib(\CC))$ is called a \emph{category with cofibrations} if the following axioms are satisfied:
	\begin{enumerate}[label=\textbf{W\arabic*},start=0]
		\item\label{W0} Every isomorphism is a cofibration and cofibrations are closed under composition.
		\item\label{W1} The category $\CC$ has a zero object $0$ and for each $X\in \CC$ the unique map $0\rightarrowtail X$ is a cofibration.
		\item\label{W2} Pushouts along cofibrations exist and cofibrations are stable under pushouts.
	\end{enumerate}
	Axioms \ref{W1} and \ref{W2} yield the existence of cokernels of cofibrations, thus for every cofibration $X\rightarrowtail Y$ there is a canonical \emph{cofibration sequence} $X\rightarrowtail Y\twoheadrightarrow Z$.
	A \emph{category with fibrations} is defined dually. A fibration is depicted by $\twoheadrightarrow$ and the set of fibrations is denoted by $\fib(\CC)$.
\end{definition}

\begin{remark}\label{remark:EquivalenceRightExactAndCategoryWithFibrations}
An inflation-exact category with \ref{L0*} is a category with cofibrations and, dually, a deflation-exact category with \ref{R0*} is a category with fibrations.
\end{remark}

\begin{definition}
	Let $(\CC,\cofib(\CC))$ be a category with fibrations and let $w\CC$ be a set of morphisms in $\CC$ called \emph{weak equivalences} (indicated by arrows endowed with $\sim$). The triple $(\CC,\cofib(\CC),w\CC)$ is called a \emph{Waldhausen category} if $w\CC$ contains all isomorphisms and is closed under composition and the following axiom (called the \emph{gluing axiom}) is satisfied:
	\begin{enumerate}[label=\textbf{W\arabic*},start=3]
		\item\label{W3}  For any commutative diagram of the form 
		\[\xymatrix{
			Z\ar[d]_{\rotatebox{90}{$\sim$}} & X\ar[d]_{\rotatebox{90}{$\sim$}}\ar@{>->}[r]\ar[l]  & Y\ar[d]_{\rotatebox{90}{$\sim$}}\\
			Z'&X'\ar@{>->}[r]\ar[l]&Y'
		}\] the induced map $Z\cup_X Y\rightarrow Z'\cup_{X'}Y'$ between the pushouts is a weak equivalence.
	\end{enumerate}
A \emph{coWaldhausen category} is defined dually.
\end{definition}

\begin{definition}
	Let $(\CC,\cofib(\CC),w\CC)$ be a Waldhausen category. The Grothendieck group $K_0(\CC)$ (often denoted as $K_0(w\CC)$) is defined as the free abelian group generated by the isomorphism classes of objects in $\CC$ modulo the relations:
	\begin{enumerate}
		\item $[X]=[Y]$ if there is a weak equivalence $X\xrightarrow{\sim}Y$.
		\item $[Z]=[X]+[Y]$ for every cofibration sequence $X\rightarrowtail Z\twoheadrightarrow Y$.
	\end{enumerate}
	The Grothendieck group of a coWaldhausen category is defined dually.
\end{definition}

\begin{proposition}
Let $\CC$ be a deflation-exact category satisfying axiom \ref{R0*} and let $\AA$ be a deflation-percolating subcategory. Let $\fib(\CC)$ be the set of deflations in $\CC$ and let $w\CC$ be the saturated closure of the set of weak isomorphisms with respect to the subcategory $\AA$.  The triple $(\CC,\fib(\CC),w\CC)$ is a coWaldhausen category.
\end{proposition}

\begin{proof}
	By assumption the category $\CC$ satisfies axiom \ref{R0*}.  By remark \ref{remark:EquivalenceRightExactAndCategoryWithFibrations}, the pair $(\CC,\fib(\CC))$ is a category with fibrations. We now show that $w\CC$ satisfies the gluing axiom. Consider a commutative diagram:
	\[\xymatrix{
		Z\ar[r]\ar[d]_{\rotatebox{90}{$\sim$}} & X\ar[d]_{\rotatebox{90}{$\sim$}} & \ar@{->>}[l]\ar[d]_{\rotatebox{90}{$\sim$}} Y\\
		Z'\ar[r] & X'&\ar@{->>}[l] Y'
	}\] Here, arrows endowed with $\sim$ are weak equivalences.
	
	By axiom \ref{R0*} and the dual of \cite[proposition~5.7]{BazzoniCrivei13} we obtain a commutative diagram:
	\[\xymatrix{
		Z\cap_{X}Y \ar@{>->}[r]\ar@{.>}[d]& Z\oplus Y\ar@{->>}[r]\ar[d]_{\rotatebox{90}{$\sim$}} & X\ar[d]_{\rotatebox{90}{$\sim$}}\\
		Z'\cap_{X'}Y'\ar@{>->}[r] & Z'\oplus Y'\ar@{->>}[r] & X'
	}\] As the localization $Q\colon \CC \to \CC / \AA$ commutes with kernels, the induced map $Z\cap_{X}Y\rightarrow Z'\cap_{X'}Y'$ descends to an isomorphism. It follows that the triple $(\CC,\fib(\CC),w\CC)$ is a coWaldhausen category.
\end{proof}

\begin{proposition}\label{proposition:WaldhausenGrothendieck}
Let $\CC$ be a deflation-exact category satisfying axiom \ref{R0*} and let $\AA$ be a deflation-percolating subcategory. Let $\fib(\CC)$ be the set of deflations in $\CC$ and let $w\CC$ be the saturated closure of the set of weak isomorphisms with respect to the subcategory $\AA$.  The quotient functor $Q\colon \CC \to \CC / \AA$ induces an isomorphism $K_0(w\CC)\cong K_0(\CC / \AA)$, where $K_0(\CC/\AA)$ is defined in the usual manner (the weak equivalences on $\CC / \AA$ are the isomorphisms).
\end{proposition}

\begin{proof}
By theorem \ref{theorem:Maintheorem}, the quotient category $\CC/\AA$ is a deflation-exact category. Note that $\Ob(\CC)=\Ob(\CC/\AA)$. One readily verifies that the map $f\colon K_0(w\CC)\rightarrow K_0(\CC/\AA)$, which is the identity on objects, is a group morphism.

Let $F(\CC / \AA)$ and $F(w\CC)$ be the free abelian groups generated by the objects of $\CC / \AA$ and $\CC$, respectively.  As $\Ob(\CC / \AA) = \Ob(\CC)$, we can consider the identity map $\bar{g}\colon F(\CC/\AA)\rightarrow F(w\CC)$.  Let $X\rightarrowtail Y\twoheadrightarrow Z$ be a conflation in $\CC/\AA$.  By definition \ref{definition:LocalizationConflation}, this means that there is a diagram
	\[\xymatrix{
		\overline{X}\ar@{>->}[rr] && \overline{Y}\ar@{->>}[rr] && \overline{Z}\\
		\ar[u]^{\rotatebox{90}{$\sim$}}\ar[d]_{\rotatebox{90}{$\sim$}}&&\ar[u]^{\rotatebox{90}{$\sim$}}\ar[d]_{\rotatebox{90}{$\sim$}}&&\ar[u]^{\rotatebox{90}{$\sim$}}\ar[d]_{\rotatebox{90}{$\sim$}}\\
		X &\ar[l]_{\sim}\ar[r]& Y &\ar[l]_{\sim}\ar[r]& Z
	}\] in $\CC$ which descends to a commutative diagram in $\CC/\AA$ and such that the vertical arrows descend to isomorphisms. Hence $[X]=[\overline{X}]$, $[Y]=[\overline{Y}]$ and $[Z]=[\overline{Z}]$ in $K_0(w\CC)$. As $\overline{X}\rightarrowtail \overline{Y}\twoheadrightarrow \overline{Z}$ is a conflation in $\CC$, we obtain $[Y]=[X]+[Z]$ in $K_0(w\CC)$. It follows that $\bar{g}$ induces a group homomorphism $K_0(\CC/\AA)\rightarrow K_0(w\CC)$. Clearly, $f$ and $g$ are inverse to each other. We conclude that $K_0(\CC/\AA)\cong K_0(w\CC)$.
\end{proof}

\subsection{Torsion theories in one-sided exact categories}\label{subsection:Torsion}

In \cite{BournGran06}, a definition of a torsion theory is given for general homological categories.  We restrict ourselves to the context of deflation-exact categories and relate them to deflation-percolating subcategories.

\begin{definition}
	Let $\CC$ be a one-sided exact category.  A torsion theory in $\CC$ is a pair of full replete (=closed under isomorphisms) subcategories $(\TT,\FF)$ such that 
	\begin{enumerate}
		\item $\Hom(T,F)=0$ for all $T\in \TT$ and all $F\in \FF$,
		\item for any object $M\in \CC$ there exists a conflation 
		\[T \rightarrowtail M \twoheadrightarrow F\]
		in $\CC$ with $T\in \TT$ and $F\in \FF$.
	\end{enumerate}
	We refer to the category $\FF$ as the torsion-free subcategory and to $\TT$ as the torsion category.
	
	A torsion theory $(\TT,\FF)$ is called \emph{hereditary} if $\TT$ is a Serre subcategory of $\CC$ and \emph{cohereditary} if $\FF$ is a Serre subcategory.
\end{definition}

\begin{lemma}\label{lemma:TorsionSequence}
	Let $\CC$ be a one-sided exact category and let $(\TT,\FF)$ be a torsion theory. For any object $M\in \CC$, there is a conflation
	\[\xymatrix{M_T\ar@{>->}[r] & M\ar@{->>}[r] & M_F,}\]
	with $ M_T \in \TT$ and $M_F\in \FF$, which is unique up to isomorphism.
\end{lemma}

\begin{proof}
The proof from \cite[lemma 4.2]{BournGran06} carries over to this setting.
\end{proof}

\begin{proposition}\label{proposition:TorsionPairAdjoints}
Let $\CC$ be a conflation category, let $(\TT,\FF)$ be a torsion theory in $\CC$.
\begin{enumerate}
	\item The inclusion functor $i\colon \TT\rightarrow \CC$ has a right adjoint $R$ and the inclusion functor $j\colon \FF\rightarrow \CC$ has a left adjoint $L$.
	\item $\TT = \{C \in \Ob(\CC) \mid \Hom(C, \FF) = 0\},$
	\item $\FF = \{C \in \Ob(\CC) \mid \Hom(\TT, C) = 0\}.$
\end{enumerate} 
\end{proposition}

\begin{proof}
The right adjoint to the embedding $i\colon \TT \to \CC$ is given by $M \mapsto M_T$.  The left adjoint to the embedding $j\colon \FF\rightarrow \CC$ is given $M \mapsto M_F$.

It follows from the definition of a torsion theory that $\TT \subseteq \{C \in \Ob(\CC) \mid \Hom(C, \FF) = 0\}$.  For the other inclusion, let $C \in \Ob(\CC)$ such that $\Hom(C, \FF)=0$.  As there is a deflation $C \deflation C_F$, we find $C_F = 0$ so that $C \cong C_T \in \TT$, as required.

The proof that $\FF = \{C \in \Ob(\CC) \mid \Hom(\TT, C) = 0\}$ is similar.
\end{proof}

\begin{corollary}\label{corollary:TorsionTheories}
Let $\CC$ be a conflation category, let $(\TT,\FF)$ be a torsion theory in $\CC$.  The subcategory $\FF$ is closed under extensions and subobjects in $\CC$.  Dually, $\TT$ is closed under extensions and (epimorphic) quotient objects in $\CC$.
\end{corollary}

\begin{remark}
Let $\CC$ be a deflation-exact category.  The conflation structure of $\CC$ induces a conflation structure on the extension-closed subcategories $\TT$ and $\FF$.  Moreover, both $\TT$ and $\FF$ are deflation-exact.
\end{remark}

The following proposition uses notation from proposition \ref{proposition:TorsionPairAdjoints}. 

\begin{proposition}\label{proposition:TorsionfreeIsPercolating}
Let $\CC$ be an exact category. Let $(\TT, \FF)$ be a cohereditary torsion pair in $\CC$.  
\begin{enumerate}
  \item The subcategory $\FF \subseteq \CC$ is a right filtering subcategory.
	\item\label{enumerate:TorsionfreeIsPercolating} If for every inflation $f\colon A \inflation X$ (with $A \in \FF$), the morphism $jL(f)\colon A \to X_F$ has a cokernel which lies in $\FF$, then $\FF$ is a deflation-percolating subcategory of $\CC$.
	\item If $L\colon \CC \to \FF$ (or, equivalently, $jL\colon \CC \to \CC$) is a conflation-exact functor, then $\FF$ is a right s-filtering subcategory.
\end{enumerate}
\end{proposition}

\begin{proof} \begin{enumerate}
	\item  As $(\TT, \FF)$ is a cohereditary torsion pair in $\CC$, we know that $\FF$ satisfies axiom \ref{P1}.  We show that axiom \ref{P2} holds.  Let $f\colon X\rightarrow F$ be a morphism with $F\in \FF$.  By lemma \ref{lemma:TorsionSequence} we obtain the commutative diagram:
		\[\xymatrix{
			X_T\ar@{>->}[d]\ar[rd]^0 & \\
			X\ar[r]^{f}\ar@{->>}[d] & F\\
			X_F\ar@{.>}[ru] &
		}\]
		The composition $X_T\rightarrowtail X\xrightarrow{\alpha} F$ is zero since $\Hom(\TT,\FF)=0$. It follows that $f$ factors through the deflation $X\deflation X_F$. This shows that axiom \ref{P2} is satisfied.

\item By the above, $\FF$ is right filtering. As $\CC$ is exact, axiom \ref{P3} is automatically satisfied. Axiom \ref{P4} follows immediately from proposition \ref{proposition:P4Criterion}.
\item We only need to verify that $\FF \subseteq \CC$ is right special.  Let $f\colon A \inflation X$ be any inflation with $A \in \FF$.  As $L$ is conflation-exact, we know that the composition $L(f)\colon A \inflation X \deflation X_F$ is an inflation.  This establishes that $\FF \subseteq \CC$ is right special.\qedhere
\end{enumerate}
\end{proof}

\begin{definition}\label{definition:RightConflationExact}
Let $F\colon \CC \to \DD$ be a functor between conflation categories.  We say that $F$ is \emph{right conflation-exact} if, for any conflation $X \stackrel{f}{\inflation} Y \stackrel{g}{\deflation} Z$, we have that $F(g)\colon F(Y) \to F(Z)$ is a deflation and $F(g) = \coker F(f).$  A \emph{left conflation-exact functor} is defined dually.
\end{definition}

\begin{remark}
The notions of a right `exact' functor \cite[\S1.3]{Rosenberg11} and a sequentially right exact functor \cite[definition 3.1]{PeschkeVanderLinden16} are special cases of right conflation-exact functors. 
\end{remark}

\begin{corollary}\label{corollary:RightConflationExact}
Let $\CC$ be an exact category. Let $(\TT, \FF)$ be a cohereditary torsion pair in $\CC$.  If $jL\colon \CC \to \CC$ is right conflation-exact, then $\FF$ is a deflation-percolating subcategory of $\CC$.
\end{corollary}

\begin{proof}
Let $X \stackrel{f}{\inflation} Y \stackrel{g}{\deflation} Z$ be a conflation in $\CC$.  By assumption, $jL(f)$ has a cokernel (namely $jL(Z)$), which lies in $\FF$ since $\FF$ is a Serre subcategory of $\CC$.  The result follows from proposition \ref{proposition:TorsionfreeIsPercolating}.
\end{proof}

\begin{example}\label{example:IndiscreteAndHausdorff}
Consider the quasi-abelian category $\mathsf{TAb}$ of topological abelian groups. Let $\TT$ be the full subcategory of of topological abelian groups with the indiscrete topology and let $\FF$ be the full subcategory of Hausdorff abelian groups. Every abelian group fits into a conflation $\overline{\left\{e_G\right\}}\inflation G \deflation G/\overline{\left\{e_G\right\}}$ with $\overline{\left\{e_G\right\}}\in \Ob(\TT)$ and $G/\overline{\left\{e_G\right\}}\in \Ob(\FF)$. Moreover, $\Hom(\TT,\FF)=0$. Hence $(\TT,\FF)$ is a torsion theory. 

It is easy to check that $(\TT, \FF)$ is hereditary.  In the notation of proposition \ref{proposition:TorsionPairAdjoints}, the functor $iR\colon \mathsf{TAb} \to \mathsf{TAb}$ is left conflation-exact.  By the dual of corollary \ref{corollary:RightConflationExact}, $\TT$ is an inflation-percolating subcategory of $\mathsf{TAb}$.  Moreover, using that $\TT$ is abelian and that $\mathsf{TAb}$ is an exact category, we find that the subcategory $\TT$ is an admissibly inflation-percolating subcategory of $\mathsf{TAb}$.

Note that the natural functor $\Phi\colon \FF\to \mathsf{TAb}/\TT$ is essentially surjective and faithful. However, $\Phi$ is not full. Indeed, consider the conflation $\mathbb{Q}\inflation \mathbb{R}\deflation \mathbb{R}/\mathbb{Q}$ where $\mathbb{R}$ has the usual Euclidean topology. Note that $\mathbb{Q}$ and $\mathbb{R}$ are Hausdorff groups and $\mathbb{R}/\mathbb{Q}$ has the indiscrete topology. It follows that $\mathbb{Q}\cong \mathbb{R}$ in $\mathsf{TAb}/\TT$. As there is no non-zero morphism $\mathbb{R}\to \mathbb{Q}$ in $\mathsf{TAb}$, $\Phi$ cannot be full.
\end{example}

\begin{example}
	Let $\rep_k(A_4)$ be the category of finite-dimensional representations of the quiver $A_4$.  The category $\rep_k(A_4)$ can be visualized by its Auslander-Reiten quiver:
	\[\xymatrix@!@C=0.5em@R=0.5em{
	&&&P_4\ar[rd]&&&\\
	 &&P_3\ar[ru]\ar[rd] && I_2\ar[rd]&&\\
		& P_2\ar[ru]\ar[rd] && X\ar[ru]\ar[rd] && I_3\ar[rd] &\\
		S_1\ar[ru] && S_2\ar[ru] && S_3\ar[ru] && S_4
	}\]
	Let $\CC$ be the full additive subcategory of $\UU$ generated by $S_1,P_2,P_3,P_4, S_2, X, I_2$ and $S_4$ . Clearly, $\CC$ is exact as it is an extension-closed subcategory of $\UU$. Let $\TT$ be the full additive subcategory of $\CC$ generated by $S_1, P_4, I_2$ and $S_4$ and let $\FF$ be the full additive subcategory of $\CC$ generated by $S_2$ and $X$. One readily verifies that $(\TT,\FF)$ is a cohereditary torsion pair in $\CC$. The functor $jL\colon \CC \to \CC$ is right deflation-exact.  By corollary \ref{corollary:RightConflationExact}, $\FF$ is a deflation-percolating subcategory of $\CC$.  However, $\FF$ is not right special in $\CC$ (this can be seen by considering the inflation $X \inflation I_2$).
\end{example}

\begin{remark}
In example \ref{Example:P4Requirement}, the subcategory $\AA$ is a torsion-free class of a torsion theory $(\TT, \AA)$.  This torsion theory is split in the sense that every object in $\CC$ is a direct sum of an object in $\AA$ and an object in $\TT.$   However, the corresponding localization is not deflation-exact.  This shows that the conditions on the functor $jL$ in proposition \ref{proposition:TorsionfreeIsPercolating}.\ref{enumerate:TorsionfreeIsPercolating} cannot be removed.

Applying $jL$ to the conflation $P_2 \inflation P_3 \deflation S_3$ yields $P_2 \to P_3 \deflation 0.$  This shows that $jL$ does not commute with cokernels.
\end{remark}

\subsection{(One-sided) quasi-abelian categories}\label{subsection:QuasiAbelian}

Many interesting examples of localizations with respect to percolating subcategories arise in the context of \emph{(one-sided) quasi-abelian} categories.  We recall the following definition from \cite{Rump01}:

\begin{definition}
	An additive category $\CC$ is called \emph{pre-abelian} if every morphism $f\colon A\rightarrow B$ in $\CC$ has a kernel and cokernel.
	
	A pre-abelian category $\CC$ is called \emph{left quasi-abelian} if cokernels are stable under pullbacks and it is called \emph{right quasi-abelian} if kernels are stable under pushouts.  A pre-abelian category is called \emph{quasi-abelian} if it is both left and right quasi-abelian. 
\end{definition}

\begin{remark}\makeatletter
\hyper@anchor{\@currentHref}%
\makeatother\label{remark:QuasiAbelianremarks}
	\begin{enumerate}
		\item Left quasi-abelian categories have a natural strongly deflation-exact structure (all cokernels are deflations) and right quasi-abelian categories have natural strongly inflation-exact structure (all kernels are inflations), for this reason we will prefer the terminology of \emph{deflation quasi-abelian} and \emph{inflation quasi-abelian} categories over the left and right versions. We refer the reader to \cite[section~4]{BazzoniCrivei13} for useful results on one-sided quasi-abelian categories in the context of one-sided exact categories. Quasi-abelian categories inherit a natural exact structure (see also \cite{Schneiders99}).
		\item Left or right quasi-abelian categories are called \emph{left or right almost abelian categories} in \cite{Rump01}. 
		\item Let $f\colon X \to Y$ be a morphism.  In a preabelian category, a morphism $f\colon X \to Y$ admits a diagram
		\[\xymatrix{
		{\ker f} \ar[r]& X \ar[rrr]^f \ar[dr] &&& Y \ar[r] & {\coker f} \\
		&& {\coim f} \ar[r]_{\widehat{f}} & {\im f} \ar[ru]}\]
		where $\coim f = \coker(\ker f)$ and $\im f = \ker(\coker f)$.  In a deflation quasi-abelian category, the canonical map $\widehat{f}\colon \coim f \to \im f$ is a monomorphism; dually, in an inflation quasi-abelian category the map $\widehat{f}\colon \coim f \to \im f$ is an epimorphism (see \cite[proposition~1]{Rump01}).  Hence, for a quasi-abelian category, the canonical morphism $\coim(f)\rightarrow \im(f)$ is both a monomorphism and an epimorphism (see also \cite[corollary~1.1.5]{Schneiders99}). 
		\item By \cite[proposition~B.3]{BondalVandenBergh03}, every quasi-abelian category $\CC$ can be realized as the torsion-free part of a cotilting torsion pair $(\TT,\CC)$ in an abelian category. Dually, $\CC$ can be realized as the torsion part of a tilting torsion pair $(\CC,\FF)$ in an abelian category.
		\item Following \cite[theorem 4.17]{BrustleHassounTattar20} (extending \cite{HassounShahWegner20}), a pre-abelian exact category $\CC$ is quasi-abelian if and only if it satisfies the admissible intersection property (in the sense of \cite[definition 4.3]{HassounRoy19}).
	\end{enumerate}
\end{remark}

The next lemma yields an easy characterization of axiom \ref{P2} for deflation quasi-abelian categories.

\begin{lemma}\label{lemma:QuasiAbelianP2}
	Let $\CC$ be a deflation quasi-abelian category and let $\AA\subseteq \CC$ be a full subcategory. The subcategory $\AA$ satisfies axiom \ref{P2} if and only if $\AA$ is closed under subobjects.
\end{lemma}

\begin{proof}
	Assume that $\AA$ satisfies axiom \ref{P2}. Let $f\colon X\hookrightarrow A$ be a monomorphism such that $A\in \Ob(\AA)$. By axiom \ref{P2}, $f$ factors as 
	\[\xymatrix{X\ar@{->>}[r] & B\ar[r] & A}\] with $B\in \Ob(\AA)$. Since $f$ is monic, the deflation $X\twoheadrightarrow B$ is an isomorphism. Hence $X\cong B\in \AA$.
	
	Conversely, assume that $\AA$ is closed under subobjects. Let $f\colon X\rightarrow A$ be a morphism in $\CC$ with $A\in \Ob(\AA)$. Since $\CC$ is deflation quasi-abelian, $f$ factors as 
	\[\xymatrix{X\ar@{->>}[r] & \coim(f)\ar[r] & A}\]
	As in remark \ref{remark:QuasiAbelianremarks}, the map $\coim(f)\rightarrow A$ is monic. Since $\AA$ is closed under subobjects, $\coim(f)\in \Ob(\AA)$.  Hence, $f$ has the desired factorization and axiom \ref{P2} holds. 
\end{proof}

The next proposition yields an easy characterization of percolating subcategories in a (one-sided) quasi-abelian setting.

\begin{proposition}\label{proposition:PercolatingInQA}
	Let $\CC$ be a deflation quasi-abelian category and let $\AA\subseteq \CC$ be a full subcategory. The subcategory $\AA$ is deflation-filtering (i.e. axioms \ref{P1} and \ref{P2} are satisfied) if and only if $\AA$ is strongly deflation-percolating.
\end{proposition}

\begin{proof}
If $\AA$ is strongly deflation-percolating, then $\AA$ is deflation-filtering by definition.  For the other implication, assume that $\AA$ is deflation-filtering in $\EE$.  By definition, axioms \ref{P1} and \ref{P2} are satisfied.  We show that $\AA$ is strongly right filtering.  Let $f\colon X\to A$ be a morphism in $\CC$ with $A\in \Ob(\AA)$.  Note that $f=f_m\circ \widehat{f}\circ f_e$ and that $f_m$ is monic as it is a kernel. By remark \ref{remark:QuasiAbelianremarks}, $\widehat{f}$ is monic and hence $f_m\circ \widehat{f}$ is monic.  By lemma \ref{lemma:QuasiAbelianP2}, we conclude that $\coim(f)\in \Ob(\AA)$ and hence $\AA$ is strongly deflation-filtering.
	
	We now show axiom \ref{P3}. Let $i\colon X\inflation Y$ be an inflation and $t\colon Y\to T$ a map such that $t\circ f$ factors through $\AA$. Without loss of generality we may assume there is a deflation $f\colon X\deflation A$ with $A\in \Ob(\AA)$ and a map $t'\colon A\to T$ such that $t\circ i=t'\circ f$. As $\CC$ is pre-abelian, pushouts exist and we obtain the following commutative diagram:
	\[\xymatrix{
		X\ar@{>->}[rr]^{i}\ar@{->>}[d]^f && Y\ar@{->>}[r]^p\ar@{->>}[d]^{f'}\ar@/^2pc/[rr]^t& Z\ar@{=}[d]& T\ar@{=}[d]\\
		A\ar[rr]^{i'}\ar@{.>}[rd]^{\iota} &&P\ar@{->>}[r]^{p'}\ar@/_2pc/@{.>}[rr]_{t''} & Z & T\\
		& K'\ar@{>->}[ru]_{\iota'} &&	
	}\] Here the square $XYAP$ is a pushout square, $f'$ is a deflation since it is the cokernel of the composition $i\circ \ker f$ (see proposition \ref{proposition:MitchellPullbackPushout}), $p'$ is the cokernel of $i'$ and $\iota'$ is the kernel of $p'$ (see proposition \ref{proposition:PushoutsPreserveCokernels}). The map $t''$ is induced by the pushout property and satisfies $t''i'=t'$. Since $f$ is epic and $p'i'f=0$, $p'i'=0$ and hence $i$ factors through $\iota'=\ker(p')$ via $\iota$.  By proposition \ref{proposition:MitchellPullback}, the square $XYKP$ is a pullback square and thus axiom \ref{R2} implies that the composition $X\stackrel{f}{\deflation} A \stackrel{\iota}{\rightarrow} K$ is a deflation. As $\CC$ is a deflation-exact category (and hence satisfies the equivalent conditions in proposition \ref{proposition:wicR3}), we find that $\iota$ is a deflation.  Axiom \ref{P1} implies that $K\in \Ob(\AA)$.  Furthermore, proposition \ref{proposition:PushoutIfCokernel} implies that the square $XYKP$ is a pushout square as well.  This shows that axiom \ref{P3} is satisfied.
	
Finally, it follows from proposition \ref{proposition:P4Criterion} that axiom \ref{P4} is satisfied.
\end{proof}

Combining lemma \ref{lemma:QuasiAbelianP2} and proposition \ref{proposition:PercolatingInQA}, we find the characterization as given in proposition \ref{proposition:IntroductionRecognitionToolI}.

\begin{proposition}\label{proposition:TwoSidedQuotientOfQuasiAbelianIsQuasiAbelian}
Let $\CC$ be a quasi-abelian category and let $\AA \subseteq \CC$ be a full subcategory.  If $\AA$ is both inflation- and deflation-percolating in $\CC$, then the category $\CC / \AA$ is quasi-abelian.
\end{proposition}

\begin{proof}
As $\AA$ is both inflation- and deflation exact, the set $S_\AA$ of weak isomorphisms is a left and right multiplicative system.  The localization $Q\colon \CC \to \CC[S^{-1}_\AA] (=\CC/ \AA)$ commutes with finite limits and colimits and $\CC[S^{-1}_\AA]$ is exact.  We need to show that every kernel-cokernel pair in $\CC[S^{-1}_\AA]$ is a conflation.  Let $X\stackrel{f}{\to} Y \stackrel{g}{\rightarrow} Z$ be a kernel-cokernel pair in $\CC[S^{-1}_\AA]$.  As $S_\AA$ is a right multiplicative system, there is a roof $Y \stackrel{\sim}{\leftarrow} Y' \stackrel{g'}{\rightarrow} Z$ in $\CC$ representing $g$.  Since $Q$ commutes with finite limits, we know that $X \cong \ker(g')$ in $\CC[S^{-1}_\AA].$  As $\ker(g) \inflation Y'$ is an inflation in $\CC$, the map $f\colon X \to Y$ is an inflation in $\CC[S^{-1}_\AA]$ as well.  Hence, $X \to Y \to Z$ is a conflation in $\CC[S^{-1}_\AA].$ 
\end{proof}

\begin{example}
Let $k$ be a field and let $\mathsf{TVS}$ be the category of topological vector spaces over $k.$  Any of the following subcategories satisfy the conditions of proposition \ref{proposition:PercolatingInQA} and are strongly deflation-percolating subcategories of $\mathsf{TVS}:$
\begin{enumerate}
	\item the subcategory of vector spaces with the discrete topology,
	\item the subcategory of finite-dimensional vector spaces.
\end{enumerate}
\end{example}

\begin{example}
	Let $\LCAf\subset \LCA$ be the full subcategory of the locally compact (Hausdorff) groups consisting of finite abelian groups. Clearly $\LCAf\subset\LCA$ is a Serre subcategory which is closed under subojects and quotients. Hence proposition \ref{proposition:IntroductionRecognitionToolI} (and its dual) imply that $\LCAf$ is both a inflation- and deflation-percolating subcategory of $\LCA$. Moreover, proposition \ref{proposition:TwoSidedQuotientOfQuasiAbelianIsQuasiAbelian} then implies that $\LCA/\LCAf$ is a quasi-abelian category. This result extends to locally compact $R$-modules (see \S\ref{subsection:LCA}).
\end{example}

\begin{example}\label{Example:IsbellCategory}
	Let $\II$ be the Isbell category, that is, the full additive subcategory of $\Ab$ generated by the abelian groups containing no element of order $p^2$ for some fixed prime $p$. The Isbell category is a deflation quasi-abelian category which does not satisfy axioms \ref{L1}, \ref{L2} and \ref{L3} (see \cite[section~2]{Kelly69} and \cite[example~4.7]{BazzoniCrivei13}). Moreover, $\II$ is a reflective subcategory of $\Ab$, i.e.~the inclusion functor $\iota\colon \II\hookrightarrow \Ab$ has a left adjoint $L\colon \Ab\to \II$. The adjoint $L$ is determined by $L(G)=\{g\in G\mid p^2\nmid \text{ord}(g)\}$.
	
	Let $\AA\subseteq \II$ be the full subcategory generated by the $p$-groups. Clearly $\AA$ is a Serre subcategory of $\II$ which is closed under subobjects. Proposition \ref{proposition:IntroductionRecognitionToolI} implies that $\AA$ is a strongly deflation-percolating subcategory of $\II$. Write $\BB$ for the full subcategory of $\Ab$ generated by the $p$-groups. As $\BB\subset\Ab$ is a Serre subcategory, the quotient $\Ab/\BB = \Ab[S^{-1}_\BB]$ is an abelian category.	Using the adjunction $L\dashv \iota$ and the universal properties of quotients, one readily verifies that $\Ab/\BB\simeq \II/\AA$ are equivalent. In particular $\II/\AA$ is abelian.
	
	Note that lemma \ref{lemma:CompositionOfAAInflations} implies that $\AA\subseteq \II$ does not satisfy axiom \ref{A3} as the inflation $\mathbb{Z}\stackrel{\cdot p}{\inflation}\mathbb{Z}$ is a weak isomorphism, but the composition $\mathbb{Z}\stackrel{\cdot p}{\inflation}\mathbb{Z}\stackrel{\cdot p}{\inflation}\mathbb{Z}$ is not an inflation in $\II$. In particular, this shows that weak isomorphisms need not be admissible.
\end{example}

\subsection{Locally compact modules}\label{subsection:LCA}
Let $\LCA$ be the category of locally compact (and Hausdorff) abelian groups. It is shown in \cite[proposition~1.2]{HoffmannSpitzweck07} that $\LCA$ is a quasi-abelian category.

Let $R$ be a unital ring, endowed with the discrete topology.  We write $R-\LC$ for the category of locally compact (and Hausdorff) $R$-modules.  We furthermore write $R-\LC_{\mathsf{C}}$ or $R-\LC_{\mathsf{D}}$ for the full subcategories given by those $R$-modules whose topology is compact or discrete, respectively.  

\begin{proposition}\label{proposition:LCAIsQuasiAbelian}
Let $R$ be a unital ring.
\begin{enumerate}
\item The categories $R-\LC$ and $\LC-R$ are quasi-abelian.
\item There are quasi-inverse contravariant functors:
\[\mbox{$\bD\colon R-\LC \to \LC-R$ and $\bD'\colon \LC-R \to R-\LC$}\]
which interchange compact and discrete $R$-modules.
\end{enumerate}
\end{proposition}

\begin{proof}
The first part follows from \cite[proposition~1.2]{HoffmannSpitzweck07} (see also \cite[proposition~2.2]{Braunling20}).

The contravariant functors in the second statement are induced by the standard Pontryagin duality $\LCA \to \LCA$ (see \cite[theorem 1]{Levin73} or \cite[theorem 2.3]{Braunling20}).
\end{proof}

It follows from \cite{HoffmannSpitzweck07} that the canonical exact structure on $R-\LC$ is described as follows: a morphism $f\colon X \to Y$ is an inflation if and only if it is a closed injection; a morphism $f\colon X \to Y$ is a deflation if and only if it is an open surjection.

\begin{proposition}\makeatletter
\hyper@anchor{\@currentHref}%
\makeatother\label{proposition:PercolatingInLCA}
	\begin{enumerate}
		\item The category $R-\LC_{\mathsf{D}}$ is an admissibly deflation-percolating subcategory of $R-\LC$.  The set $S_{R-\LC_{\mathsf{D}}}$ of admissible weak isomorphisms is saturated.
		\item The category $R-\LC_{\mathsf{C}}$ is an admissibly inflation-percolating subcategory of $R-\LC$.  The set $S_{(R-\LC_{\mathsf{C}})}$ of admissible weak isomorphisms is saturated.
	\end{enumerate}
\end{proposition}

\begin{proof}
We first show that $R-\LC_{\mathsf{D}}$ satisfies axiom \ref{A1}. Let $A\stackrel{f}{\inflation}B\stackrel{g}{\deflation}C$ be a conflation in $R-\LC$. It is straightforward to show that if $B$ is discrete, then so are $A$ and $C$. Conversely, assume that $A$ and $C$ are discrete. Since the singleton $\left\{0_B\right\}$ is open in $A$ and $A$ has the subspace topology of $B$, there exists an open $U\subseteq B$ such that $\left\{0_B\right\}=U\cap A$. Since the singleton $\left\{0_C\right\}$ is open in $C$, $g^{-1}(\left\{0_C\right\})=\ker(g)=A$ is open in $B$. Hence $\left\{0_B\right\}$ is open $B$. It follows that $B$ has the discrete topology.

Axiom \ref{A2} follows from the observation that any map $f\colon X\rightarrow A$ with $A$ discrete induces an open surjective map $X\rightarrow \im(f)$ in $R-\LC$. Axiom \ref{A3} is automatic as $R-\LC$ is an exact category.
It follows from proposition \ref{proposition:Saturation} that the set $S_{(R-\LC_{\mathsf{C}})}$ is saturated and it follows from theorem \ref{theorem:WeakIsomorphismsEqualAAInverseIsomorphisms} that weak isomorphisms are admissible.

Pontryagin duality implies the corresponding statements about $R-\LC_{\mathsf{C}}$.
\end{proof}

\begin{remark}
\cite[example 4]{Braunling20} shows that $\LCA_{\mathsf{C}}$ is not left (or right) s-filtering in $\LCA$ in the sense of \cite{Schlichting04} (see definition \ref{definition:SpecialFiltering}).  On the other hand, putting $R=\mathbb{Z}$, the previous proposition implies that the category $\LCA_C$ is an admissibly inflation-percolating subcategory of $\LCA$.  It follows that $\LCA/\LCA_{\mathsf{C}}$ can be described as a localization with respect to the saturated left multiplicative system given by the weak $\LCA_{\mathsf{C}}^{-1}$-isomorphisms and the localization carries a natural inflation-exact structure (see theorem \ref{theorem:Maintheorem}).
\end{remark}

\begin{remark}
The category $\LCA_{\mathsf{D}}$ is not an inflation-percolating subcategory of $\LCA$. Indeed, the dual of axiom \ref{A2} fails for the map $1_{\mathbb{R}}\colon (\mathbb{R},\tau_{\text{discrete}})\rightarrow (\mathbb{R},\tau_{\text{trivial}})$.  Dually, the category $\LCA_{\mathsf{C}}$ is not a deflation-percolating subcategory of $\LCA$.
\end{remark}

\begin{remark}
It follows from proposition \ref{proposition:PercolatingOfPercolating} that Serre subcategories of $R-\LC_{\mathsf{D}}$ or $R-\LC_{\mathsf{C}}$ are themselves admissibly (deflation- or inflation-)percolating subcategories of $R-\LC$.
\end{remark}

Following \cite{Braunling19, Braunling20}, we write $R-\LC_{\bR \mathsf{C}}$ for the full subcategory of $R-\LC$ whose objects have a direct sum decomposition $\bR^n \oplus C$ (as topological groups) where $C$ is compact.  It is shown in \cite[corollary 9.4]{Braunling19} that $R-\LC_{\bR {\mathsf{C}}}$ is an idempotent complete fully exact subcategory of $R-\LC$.  We write $R-\LC_{\bR}$ for those objects of $R-\LC$ which are isomorphic to $\bR^n$ (with the standard topology).

As an application of proposition \ref{proposition:TorsionfreeIsPercolating}, we show that $R-\LC_{\mathsf{C}}$ is left s-filtering in $R-\LC_{\bR {\mathsf{C}}}$.  In this way, we recover \cite[proposition~9.8]{Braunling19}.

\begin{proposition}
\begin{enumerate}
	\item The pair $(R-\LC_{{\mathsf{C}}}, R-\LC_{\bR})$ is a torsion pair in $R-\LC_{\bR {\mathsf{C}}}$.
	\item $R-\LC_{{\mathsf{C}}}$ is left s-filtering in $R-\LC_{\bR {\mathsf{C}}}$.
	\item $R-\LC_{\bR}$ is right s-filtering in $R-\LC_{\bR {\mathsf{C}}}$.
\end{enumerate}
\end{proposition}

\begin{proof}
It is clear that $(R-\LC_{\mathsf{C}}, R-\LC_{\bR})$ is a torsion pair in $R-\LC_{\bR {\mathsf{C}}}$: the torsion of an object $\bR^n \oplus C$ is given by $t(\bR^n \oplus C) = C$ and the torsion-free part of an object is given by $f(\bR^n \oplus C) = \bR^n$.

As $R-\LC_{\bR {\mathsf{C}}}$ is a fully exact subcategory of $R-\LC$ and $R-\LC_{C}$ satisfies \ref{P1} in $R-\LC$, it follows that $R-\LC_{{\mathsf{C}}}$ satisfies \ref{P1} in $R-\LC_{\bR {\mathsf{C}}}$.  Moreover, as $\bR$ is injective in $\LCA$, it is clear that $R-\LC_{\bR}$ is closed under extensions.  Hence, we find that $R-\LC_{\bR}$ also satisfies \ref{P1} in $R-\LC_{\bR {\mathsf{C}}}$.

Lastly, given any conflation $\bR^{n_1} \oplus C_1 \inflation \bR^{n_2} \oplus C_2 \deflation \bR^{n_3} \oplus C_3$ in $R-\LC_{\bR {\mathsf{C}}}$, we find, by applying the functor $L\colon R-\LC_{\bR {\mathsf{C}}} \to R-\LC_{\bR}$, the conflation $\bR^{n_1}\inflation \bR^{n_2}\deflation \bR^{n_3}$.  The $3 \times 3$-lemma shows that the torsion part $C_1 \inflation  C_2 \deflation  C_3$ is also a conflation.
\end{proof}

In the following proposition, we write $\LCAcon$ and $\LCAtd$ for the full subcategories of $\LCA$ given by the connected locally abelian groups and the totally disconnected locally abelian groups.

\begin{proposition}
In the category $\LCA$, there is a cohereditary torsion pair $(\LCAcon, \LCAtd)$.  In particular, the subcategory $\LCAtd$ of totally disconnected locally compact abelian groups is a strongly percolating subcategory of $\LCA$.
\end{proposition}

\begin{proof}
For a $G \in \LCA$, let $G_0$ be the connected component of the identity.  It follows from \cite[Proposition III.4.6.14]{BourbakiTop} that $G_0$ is a (closed) subgroup of $G$ and from \cite[proposition I.11.5.9]{BourbakiTop} that the quotient $G/G_0$ is totally disconnected.  As $\Hom(\LCAcon, \LCAtd)=0$, we see that $(\LCAcon, \LCAtd)$ is indeed a torsion pair.  Furthermore, it follows from \cite[Corollary III.4.6.3]{BourbakiTop} that quotients of totally disconnected locally compact groups are totally disconnected so that $(\LCAcon, \LCAtd)$ is a cohereditary torsion pair (this uses corollary \ref{corollary:TorsionTheories}).

As $\LCA$ is a quasi-abelian category, it follows from proposition \ref{proposition:TorsionfreeIsPercolating} (or proposition \ref{proposition:PercolatingInQA}) that $\LCAtd$ is a strongly deflation-percolating subcategory of $\LCA.$
\end{proof}

\begin{remark}
As $\LCA_{\mathsf{D}} \subseteq \LCAtd$, we find that $\LCA_{\mathsf{D}}$ is an admissibly deflation-percolating subcategory of $\LCAtd$.  It follows from proposition \ref{proposition:PercolatingInLCA} that the category $\LCA_{\mathsf{C}} \cap \LCAtd$ of compact totally disconnected groups is an admissibly inflation-percolating subcategory of $\LCAtd$.
\end{remark}

\subsection{An example coming from filtered modules}\label{Subsection:GliderExample}

In this section, we consider an example of an exact category $\EE$ and an admissibly deflation-percolating subcategory $\AA$ such that the localization $S_\AA^{-1}\EE$ is deflation-exact but not inflation-exact. This shows that in general one cannot expect that localizing with respect to an (admissibly) percolating subcategory preserves two-sided exactness.  This example is based on the theory of glider representations (see for example \cite{CaenepeelVanOystaeyen16,CaenepeelVanOystaeyen19book,CaenepeelVanOystaeyen18} or \cite{HenrardvanRoosmalen20a}). 

Let $k$ be a field and let $R$ be the matrix ring \[R=\begin{pmatrix}
k & 0 & 0\\
k[t]_{\leq 2} & k& 0\\
k[t] & k[t] & k[t]
\end{pmatrix}.\] We write $E_{i,j}$ for the $3\times 3$-matrix defined by $(E_{i,j})_{k,l}=\delta_{i,k}\delta_{j,l}$ where $\delta_{i,k}$ is the Kronecker delta. We write $e_1,e_2,e_3$ for the primitive orthogonal idempotents, i.e. $e_i=E_{i,i}$. Let $\EE$ be the abelian category of left $R$-modules. Given an $R$-module $M$, we have that $M\cong e_1M+e_2M+e_3M$ as a $k$-vector space. Note that $e_3M$ is a $k[t]$-module. Let $\EE$ be the full subcategory of $\EE$ of all left $R$-modules $M$ such that the maps
\begin{eqnarray*}
 \iota_1\colon e_1M\hookrightarrow e_2M &:& m\mapsto E_{2,1}m\\
 \iota_2\colon e_2M\hookrightarrow e_3M &:& m\mapsto E_{3,2}m
\end{eqnarray*}
are injective. For simplicity, we write an object of $\EE$ as $e_1M\hookrightarrow e_2M\hookrightarrow e_3M$.  One readily verifies that $\EE$ is extension-closed in $\EE$ and therefore inherits a natural exact structure.  Using \cite[proposition~B.3]{BondalVandenBergh03} we see that $\EE$ is in fact a quasi-abelian category.  Indeed, one can verify that $\EE$ is closed under subobjects and contains all projective $R$-modules.  It follows that $\EE$ arises as the the torsion-free part of a cotilting torsion pair.

Let $\AA$ be the full subcategory of $\EE$ consisting of all $R$-modules such that $e_1M=0=e_2M$.  Clearly, $\AA$ is equivalent to the abelian category of $k[t]$-modules.  Consider the map $\phi$ in $\EE$ given by the following commutative diagram:
\[\xymatrix{
0\ar@{^{(}->}[r]\ar[d]&0\ar@{^{(}->}[r]\ar[d]&k\ar[d]^{1_k}\\
k\ar@{^{(}->}[r] & k\ar@{^{(}->}[r] & k
}\] One readily verifies that $\ker(\phi)=0=\coker(\phi)$ in $\EE$. It follows that if $\phi$ is admissible, it is an isomorphism.  However, $\phi$ does not admit a right inverse.  It follows that $\AA$ is not inflation-percolating in $\EE$. On the other hand, it is easy to see that $\AA$ is admissibly deflation-percolating in $\EE$.  Hence, we can describe the localization $\EE/\AA$ using theorem \ref{theorem:Maintheorem}.  Moreover, proposition \ref{proposition:Saturation} implies that the right multiplicative system of $\AA^{-1}$-isomorphisms is saturated.

\begin{lemma}\label{Lemma:DesciptionIsoInQuotient}
	Let $M,N\in \Ob(\EE/\AA)$ such that $M\cong N$. Then $e_1M\cong e_1N$ and $e_2M\cong e_2N$ as $k$-vector spaces.
\end{lemma}

\begin{proof}
	Let $f\colon X\rightarrow Y$ be an $\AA^{-1}$-isomorphism in $\EE/\AA$ and write $f_i$ for the induced map $e_iX\rightarrow e_iY$. Since $\ker(f),\coker(f)\in \AA$, we have that $\ker(f_j)=0=\coker(f_j)$ for $j=1$ or $2$.  It follows that $f_1$ and $f_2$ are isomorphisms of $k$-vector spaces.
	
	Assume that $M\cong N$ in $\EE/\AA$ and let $(g\colon L \rightarrow N, s\colon L\rightarrow M)\in \Hom_{S_{\AA}^{-1}\EE}(M,N)$ be an isomorphism in $\EE/\AA$. Since $Q(g)$ is also an isomorphism in $\EE/\AA$ and $S_{\AA}$ is saturated, $g$ is an $\AA^{-1}$-isomorphism. It follows that $g_1,g_2,s_1$ and $s_2$ are isomorphisms and hence $e_1M\cong e_1N$ and $e_2M\cong e_2N$ as $k$-vector spaces.
\end{proof}

We now show that the localization $\EE/\AA$ is not inflation-exact by explicitly showing the failure of axiom \ref{L1}. Consider the following commutative diagram in $\EE$
\[\xymatrix{
	tRe_2 \ar@{>->}[d]_{f} && 0\ar@{^{(}->}[r]\ar[d] & kt\ar@{^{(}->}[r]\ar[d]^{\begin{psmallmatrix}1\\0\end{psmallmatrix}} & tk[t]\ar[d]^{\begin{psmallmatrix}1\\0\end{psmallmatrix}}\\
	tRe_2\oplus tRe_2\ar[d]_{g}^{\rotatebox{90}{$\sim$}} && 0 \ar@{^{(}->}[r]\ar[d] & kt \oplus kt \ar@{^{(}->}[r]\ar[d]^{\begin{psmallmatrix}1&0\\0&t\end{psmallmatrix}} & tk[t]\oplus tk[t]\ar[d]^{\begin{psmallmatrix}1&t\end{psmallmatrix}}\\
	M \ar@{>->}[d]_{h}&& 0 \ar@{^{(}->}[r]\ar[d] & kt\oplus kt^2 \ar@{^{(}->}[r]\ar[d]^{\begin{psmallmatrix}0&0\\1&0\\0&1 \end{psmallmatrix}} & tk[t]\ar[d]_1\\
	Re_1 && k \ar@{^{(}->}[r] & k\oplus kt \oplus kt^2 \ar@{^{(}->}[r] & k[t]
}\]
One can verify that $f\colon tRe_2\rightarrow tRe_2\oplus tRe_2$ is an inflation, $g$ is an $\AA^{-1}$-isomorphism, and $h \colon M\rightarrow Re_1$ is an inflation. It follows that the composition $Re_2\xrightarrow{f} Re_2\oplus Re_2\xrightarrow{g} M$ descends to an inflation in $\EE/\AA$. The cokernel of $hgf$ in $\EE/\AA$ is given by $k\hookrightarrow k\hookrightarrow k$. A direct computation shows that $\ker(\coker(hgf))$ is given by $0\hookrightarrow kt\oplus kt^2\hookrightarrow tk[t]$. By lemma \ref{Lemma:DesciptionIsoInQuotient}, $\ker(\coker(hgf))\not\cong tRe_2$ in $\EE/\AA$. It follows that $hgf$ is not an inflation in $\EE/\AA$. This shows that axiom \ref{L1} is not satisfied.